\documentclass[reqno,11pt]{amsart}
\usepackage{amssymb}
\usepackage{amsmath}
\usepackage{amsthm}
\usepackage[left=1.5cm, right=1.5cm, top=1.5cm, bottom=1.5cm]{geometry}

\newtheorem{Theorem}{Theorem } [section]
\newtheorem{lemma}[Theorem]{Lemma}

\newtheorem{remark}{Remark}
\newtheorem{Definition}{Definition}

\numberwithin{equation}{section}

\newcommand{\vertiii}[1]{{\vert\kern-0.25ex\vert\kern-0.25ex\vert #1 
    \vert\kern-0.25ex\vert\kern-0.25ex\vert}}

\DeclareMathOperator{\supp}{supp}

\DeclareMathOperator{\re}{Re}

\begin{document}

\title{Local well-posedness and regularity properties for an initial-boundary value problem associated to the fifth order Korteweg-de Vries equation}
\author{Eddye Bustamante, Jos\'e Jim\'enez Urrea and Jorge Mej\'{\i}a}
\subjclass[2000]{35Q53}

\keywords{Aquí van las palabras clave}
\address{Eddye Bustamante M., Jos\'e Jim\'enez Urrea, Jorge Mej\'{\i}a L. \newline
Departamento de Matem\'aticas\\Universidad Nacional de Colombia\newline
A. A. 3840 Medell\'{\i}n, Colombia}
\email{eabusta0@unal.edu.co, jmjimene@unal.edu.co, jemejia@unal.edu.co}

\begin{abstract} In this work we prove that the initial-boundary value problem (IBVP) for the fifth order Korteweg-de Vries equation
\begin{align*}
\left. \begin{array}{rlr}
u_t+\partial_x^5 u+u\partial_x u&\hspace{-2mm}=0,&\quad x\in\mathbb R^+,\; t\in\mathbb R^+,\\
u(x,0)&\hspace{-2mm}=g(x),&\\
u(0,t)=h_1(t),\, \partial_x u(0,t)&\hspace{-2mm}=h_2(t),\,\partial_x^2 u(0,t)=h_3(t),
\end{array} \right\}
\end{align*}
is locally well posed, when the data $g$, $h_1$, $h_2$, $h_3$ are taken in such a way that $g\in H^s(\mathbb R_x^+)$, and $h_{j+1}\in H^{\frac{s+2-j}5}(\mathbb R_t^+)$, $j=0,1,2$, $s\in [0,\frac{11}4)\setminus \{\frac12,\frac32,\frac52\}$, and satisfy the following compatibility conditions:
\begin{enumerate}
\item[(i)] $g(0)=h_1(0)$ if $\frac12<s<\frac32$;
\item[(ii)] $g(0)=h_1(0)$, $g'(0)=h_2(0)$ if $\frac32<s<\frac52$; and
\item[(ii)] $g(0)=h_1(0)$, $g'(0)=h_2(0)$, $g''(0)=h_3(0)$ if $\frac52<s<\frac{11}4$. \end{enumerate}
Besides, we prove that the nonlinear part of the solution is smoother than the initial datum $g$.
\end{abstract}

\maketitle
\section{Introduction}

In this article we consider the initial-boundary value problem (IBVP) corresponding to the fifth order Korteweg-de Vries (KdV) equation 
\begin{align}
\left. \begin{array}{rlr}
u_t+\partial_x^5 u+u\partial_x u&\hspace{-2mm}=0,&\quad x\in\mathbb R^+,\; t\in\mathbb R^+,\\
u(x,0)&\hspace{-2mm}=g(x),&\\
u(0,t)=h_1(t),\, \partial_x u(0,t)&\hspace{-2mm}=h_2(t),\, \partial_x^2 u(0,t)=h_3(t),
\end{array} \right\}\label{maineq}
\end{align}
where the initial datum $g$ and the boundary data $h_{j+1}$, $j=0,1,2$ are taken in the Sobolev spaces $H^s(\mathbb R^+_x)$, $H^{\frac{s+2-j}5}(\mathbb R^+_t)$, $j=0,1,2$; respectively, $s\in[0,\frac{11}4)\setminus\{\frac12,\frac32,\frac52\}$, and satisfy the following compatibility conditions:
\begin{enumerate}
\item[(i)] $g(0)=h_1(0)$ if $\frac12<s<\frac32$;
\item[(ii)] $g(0)=h_1(0)$, $g'(0)=h_2(0)$ if $\frac32<s<\frac52$; and
\item[(ii)] $g(0)=h_1(0)$, $g'(0)=h_2(0)$, $g''(0)=h_3(0)$ if $\frac52<s<\frac{11}4$. 
\end{enumerate}
Here $\mathbb R^+:=(0,+\infty)$.\\

This kind of problems has been studied in recent years in the case of the third order KdV equation, among others, by Colliander and Kenig in \cite{CK2002}, Faminskii in \cite{F1999a}, \cite{F1999b}, and \cite{F2005}, Bona, Sun, and Zhang in \cite{BSZ2002}, and Holmer in \cite{H2006}.\\

For other evolution equations, as the Schrödinger equation, this mixed problem has been considered by Holmer in \cite{H2005}, Erdoğan and Tzirakis in \cite{ET2016}, and Bona, Sun, and Zhang in \cite{BSZ2018}.\\

In \cite{CK2002}, a new method to solve initial-boundary problems is introduced. The idea of this method is to turn the IBVP into an initial value problem (IVP) by adding an adequate force term of the type $\delta_0(x)f(t)$ to the equation, where $\delta_0$ denotes the Dirac mass at $x=0$, and the function $f(t)$, $t\in\mathbb R$ is related to the boundary data in $t>0$, through the Riemann-Liouville fractional integral. On the other hand, the initial datum in this new IVP, defined for $x\in\mathbb R$, is an adequate extension of the initial datum of the IBVP defined for $x\in\mathbb R^+$.\\

This reformulation of the problem allows to use the known theory for the IVPs in the context of mixed initial-boundary value problems. The solution of the IVP is searched in the framework of the modified Bourgain spaces of tempered distributions $u$ in $\mathbb R^2$, such that
$$\|(\langle \xi \rangle^s \langle \tau - \xi^3 \rangle^b+\chi(\xi) \langle \tau \rangle^\alpha)\widehat u(\xi,\tau)\|_{L^2_{\xi\tau}}<\infty,$$
where $\,\widehat{\text{}}\,$ is the Fourier transform in $\mathbb R^2$, $(\xi,\tau)$ is the variable in the frequency space corresponding to the space-time variable $(x,t)$, $\chi$ is the characteristic function of the interval $[-1,1]$, $0\leq s\leq 1$, $0<b<\frac12$, and $\frac12<\alpha<\frac23$. Notation $\langle \cdot \rangle$ is an abbreviation for $(1+|\cdot|)$.\\

Colliander and Kenig prove local well posedness of the IBVP associated to the KdV equation when the initial datum is in $H^s(\mathbb R^+_x)$, the boundary datum is in $H^{\frac{s+1}3}(\mathbb R^+_t)$, with $0\leq s\leq 1$, $s\neq\frac12$, and these data satisfy the compatibility condition when $s>\frac12$.\\

The IBVP for the third order KdV equation, with additional transport term $a\partial_xu$, $a\in\mathbb R$, has been studied in detail by Faminskii in several articles. Particularly, in \cite{F2005}, this problem is addressed with the help of special solutions of boundary potential type $J(x,t,\mu)$ for the inhomogeneous linear problem

\begin{align}
\left. \begin{array}{rlr}
u_t+\partial_x^3 u+a\partial_x u&\hspace{-2mm}=0,&\quad x>0,\; t>0,\\
u(x,0)&\hspace{-2mm}=0,&\\
u(0,t)&\hspace{-2mm}=\mu(t).&
\end{array} \right\}\label{V2-0.2}
\end{align}

and it is established global well posedness of the IBVP when the initial datum belongs to $H^s(\mathbb R^+_x)$, and the boundary datum belongs to $H^{\frac{s+1}3+\epsilon}(\mathbb R^+_t)$, with $s\geq 0$, $s\neq 3k+\frac12$, $k\geq 0$ is an integer number, where $\epsilon>0$ is arbitrary small in the case $s<1$, and $\epsilon=0$ in the case $s\geq 1$. Besides the initial and boundary data satisfy the compatibility condition when $s>\frac12$. In \cite{F2005}, solutions are also searched in functional spaces of the Bourgain type that are slight modifications of the spaces used by Colliander and Kenig in \cite{CK2002}.\\

In \cite{BSZ2002}, Bona, Sun, and Zhang, built a boundary potential $J(x,t,\mu)$ for the problem \eqref{V2-0.2} with $a=1$, by means of the Laplace transform and they managed to prove global well posedness for the IBVP associated to the equation
$$\partial_t u+\partial_x^3 u+\partial_x u+u\partial_x u=0,$$
when the initial datum belongs to $H^s(\mathbb R_x^+)$, the boundary datum belongs to $H^{\frac{s+1}3}(\mathbb R^+_t)$, with $s\geq 3$, and besides those data satisfy the compatibility condition.\\

In \cite{ET2016}, Erdoğan, and Tzirakis consider an IBVP for the one dimensional Schrödinger equation with cubic nonlinearity. As in \cite{CK2002}, they use the known technics for the Cauchy problem, when the spatial variable runs along $\mathbb R$. This Cauchy problem is studied using Bourgain functional spaces. An important role in \cite{ET2016} is played by the linear non homogeneous IBVP 
\begin{align}
\left. \begin{array}{rlr}
iu_t+\partial_x^2 u&\hspace{-2mm}=0,&\quad x>0,\; t>0,\\
u(x,0)&\hspace{-2mm}=0,&\\
u(0,t)&\hspace{-2mm}=h(t),&
\end{array} \right\}\label{V2-0.3}
\end{align}

for which its solution is found by means of the Laplace transform, following the ideas of Bona et al. in \cite{BSZ2018}. Next, this solution is extended to $\mathbb R^2$ in a convenient way, given as a result a function $\tilde W_0^t(0,h)(\cdot_x,\cdot_t)$ defined on $\mathbb R^2$. Then, they write the solution of the IBVP associated to the nonlinear equation 

$$i\partial_tu + \partial_x^2u + \lambda|u|^2u = 0$$

with initial datum $g(x)$, $x>0$; and boundary datum $h(t)$, $t>0$, as a superposition of the Duhamel's formula for a nonlinear Cauchy problem with initial datum an adequate extension $g_l$ of $g$, and the extension $\tilde W_0^t(0,h-p)$ of the solution of a linear problem of type \eqref{V2-0.3}, with boundary data $h-p$, being $p(t)$ the result of evaluating the Duhamel's formula at $x=0$. This way Erdoğan and Tzirakis combine the Laplace transform method used in \cite{BSZ2018} with the Bourgain spaces method used in \cite{CK2002}.\\

In \cite{FK2011}, Faminskii and Kuvshinov study an IBVP associated to the Kawahara equation
$$u_t-\partial_x^5 u + b \partial_x^3 u + a\partial_x u+ g(u) \partial_xu = f(x,t),$$
in the half line $(0,+\infty)$, and in the open interval $(0,1)$. In the case of the half line the initial datum belongs to the Sobolev space $H^k(\mathbb R^+_x)$ with $k\geq 2$ integer, while in the case of the interval $(0,1)$ the initial datum belongs to $H^k(0,1)$ with $k\geq 0$ integer. The regularity of the boundary data depends on the index $k$, in a way that is natural when considering the associated linear problem. In both cases Faminskii and Kuvshinov manage to prove well posedness of the IBVP.\\

Cavalcante and Kwak in \cite{CK2020} consider two IBVP for the Kawahara equation with $a=b=0$, $f=0$ and $g(u):=u$, in the half lines $(0,+\infty)$ and $(-\infty,0)$. The method used by them is the same as Colliander and Kenig in \cite{CK2002} which consists of converting the IBVP into an IVP, adding an appropriate forcing term to the differential equation. Cavalcante and Kwak prove local well-posedness for both IBVP, when the initial data are taken in $H^s(\mathbb R^+)$ and $H^s(\mathbb R^-)$ ($\mathbb R^-:=(-\infty,0)$), respectively, with $s\in(-\frac74,\frac52)\setminus \{\frac12,\frac32\}$ and the boundary data have a regularity, which depends on the index $s$.\\

In our work, using the approach of \cite{ET2016}, we prove that the IBVP \eqref{maineq} is locally well posed, when the initial datum is taken in $H^s(\mathbb R^+_x)$ and the boundary data $h_{j+1}$, $j=0,1,2$ are taken in $H^{\frac{s+2-j}5}(\mathbb R_t^+)$, satisfy the compatibility conditions, and $s\in [0,\frac{11}4)\setminus\{\frac12,\frac32,\frac52\}$. Besides, using a gain of regularity of the bilinear form $\partial_x (uv)$, we succeed to establish, as in \cite{ET2016}, that the nonlinear part of the solution is smoother than the initial datum.\\

It is important to point out that the IBVP considered in this paper is equivalent to the IBVP for the Kawahara equation in the half line $(-\infty,0)$, studied by Cavalcante and Kwak, but our approach is different. Furthermore, although we do not consider Sobolev spaces of negative indices, we managed to extend the local well-posedness to indices $s$ in the interval $[\frac52,\frac{11}4)$. We work in the context of the modified Bourgain spaces $X^{s,b,\alpha}\equiv X^{s,b,\alpha}(\mathbb R^2)$ of tempered distributions $u$ in $\mathbb R^2_{xt}$ such that
$$\| (\langle \xi \rangle^s \langle \tau + \xi^5 \rangle^b+\chi(\xi) \langle \tau\rangle^\alpha) \widehat u(\xi,\tau) \|_{L^2_{\xi\tau}}<\infty,$$
where $s\in [0,\tfrac{11}4)\setminus\{\frac12,\frac32,\frac52\}$, $b=\frac12^-$ (i.e., there exists $\epsilon>0$ such that $\frac12-\epsilon\leq b<\frac12$), $\alpha=\frac12^+$, and $\chi$ is the characteristic function of the interval $[-1,1]$. Specifically, if $u(t) \equiv u(\cdot_x,t)$, we prove that there exists
$$u\in X^{s,b,\alpha}(\mathbb R^2)\cap C(\mathbb R_t;H^s(\mathbb R_x))\text{ with } \partial_x^j u\in C(\mathbb R_x;H^{\frac{s+2-j}5}(\mathbb R_t)),\, j=0,1,2,$$
such that, for some $T>0$, and $t\in[0,T]$,
\begin{align}
u(t)=W_{\mathbb R}(t)g_l + \int_0^t W_{\mathbb R}(t-t') ( -\tfrac12 \partial_xu(t')^2 ) dt' + W_0^t(0,h_1-p_1,h_2-p_2,h_3-p_3)(t),\label{V2-0.4}
\end{align}

where $g_l \in H^s(\mathbb R_x)$ is an extension of $g\in H^s(\mathbb R_x^+)$, $[W_{\mathbb R}(\cdot_t)g_l](\cdot_x)$ is the solution of the linear Cauchy problem

\begin{align}
\left. \begin{array}{rlr}
u_t+\partial_x^5 u&\hspace{-2mm}=0,&\quad (x,t)\in\mathbb R^2,\\
u(x,0)&\hspace{-2mm}=g_l(x),&\\
\end{array} \right\}\label{V2-0.5}
\end{align}

given by
$$[W_{\mathbb R}(t)g_l](x)=\mathcal F^{-1}_x[e^{-it(\cdot_\xi)^5}\widehat g_l(\cdot_\xi)](x),$$
being $\, \widehat{\text{}}\,$, $\mathcal F_x^{-1}$ the direct and inverse Fourier transforms, respectively, with respect to the spatial variable,
$$W_0^t(0,h_1-p_1,h_2-p_2,h_3-p_3)(t)\equiv W_0^t(0,h_1-p_1,h_2-p_2,h_3-p_3)(\cdot_x,t),\quad t\in\mathbb R$$
is an adequate extension to $\mathbb R^2$ of the solution of the linear non homogenous problem
\begin{align}
\left. \begin{array}{rlr}
u_t+\partial_x^5 u&\hspace{-2mm}=0,&\quad x>0,\; t>0,\\
u(x,0)&\hspace{-2mm}=0,&\\
u(0,t)&\hspace{-2mm}=h_1(t)-p_1(t),&\\
\partial_x u(0,t)&\hspace{-2mm}=h_2(t)-p_2(t),&\\
\partial_x^2 u(0,t)&\hspace{-2mm}=h_3(t)-p_3(t),
\end{array} \right\}\label{V2-0.6}
\end{align}

with
\begin{align*}
p_{j+1}(t):=\partial_x^j \left[(W_{\mathbb R}(t)g_l)(0)+\int_0^t[W_{\mathbb R}(t-t')(-\frac12\partial_x u(t')^2)](0) dt'\right],\; j=0,1,2.
\end{align*}

The restriction of $u$ to $[0,+\infty)\times[0,T]$ will be a local solution in time of the IBVP \eqref{maineq}.\\

Regarding the relationship between the regularity of the boundary data and the regularity of the initial data, we must point out that it is natural if we take into account that in the linear case we have:

$$\| D_t^{\frac{s+2-j}5}[W_{\mathbb R}(\cdot_t)g_l](0) \|_{L^2(\mathbb R_t)}\leq C\|g_l\|_{H^s(\mathbb R_x)}, \, j=0,1,2.$$

On the other hand, the number of boundary conditions is natural, after observing that in the linear case, using integration by parts and smooth solutions with adequate decay at infinity, the uniqueness of the solution of this IBVP require these three boundary conditions.\\

This work is organized as follows: in Section \ref{V2-S1} we define the functional context that we will use, this is, the Sobolev spaces in $(0,+\infty)$ to which the initial datum $g$, and the boundary data $h$ belong, and the Bourgain spaces $X^{s,b,\alpha}$ where solutions $u$ are searched. On the other hand, we precise our concept of local solution in time for the IBVP \eqref{maineq}. Section \ref{ILBVP} is dedicated to the study of the inhomogeneous linear problem \eqref{V2-0.6}. In this study we use the Laplace transform, and basic results of complex variable functions. Besides, we define an adequate extension to $\mathbb R^2$, $W_0^t(0,h_1-p_1,h_2-p_2,h_3-p_3)(\cdot_x,\cdot_t)$, of the solution to the problem \eqref{V2-0.6}. In Section \ref{LNLE} we deal with the linear and non linear estimates of the terms of the Duhamel's formula, associated to the Cauchy problem for the fifth order KdV equation, in the norms of the spaces $X^{s,b}$, $X^{s,b,\alpha}$, the space of continuous functions of the time variable $t$, $C(\mathbb R_t;H^s(\mathbb R_x))$, and the spaces of continuous functions of the spatial variable $x$, $C(\mathbb R_x;H^{\frac{s+2-j}5}(\mathbb R_t))$, $j=0,1,2$. In Section \ref{V2-S4} we prove the existence of a solution $u$ in $X^{s,b,\alpha}\cap C(\mathbb R_t;H^s(\mathbb R_x))$ with $\partial_x^j u\in C(\mathbb R_x;H^{\frac{s+2-j}5}(\mathbb R_t))$, $j=0,1,2$, of an integral equation that coincides with equation \eqref{V2-0.4} for values of $t$ in the interval $[0,T]$, which restriction $u|_{[0,+\infty)\times[0,T]}$ is a local solution in time of the IBVP \eqref{maineq}. Besides, we prove that this local solution is a generalized solution of the IBVP \eqref{maineq} in the sense used by Faminskii in \cite{F2005}. In Section \ref{Uniq}, following the method used by Faminskii in \cite{F2005}, we prove the uniqueness of generalized solutions of the IBVP \eqref{maineq} in $X^{0,b,\alpha}$. In this manner, we can affirm that the local solution found, depends neither on the extension $g_l$ of $g$ used, nor on the particular function $\rho$ used to define the extension $W_0^t(0,h_1-p_1,h_2-p_2,h_3-p_3)$ (see Section \ref{ILBVP}). Finally, in Section \ref{Regular}, we prove that the nonlinear part of the solution of the IBVP \eqref{maineq} is smoother than the initial datum.

\begin{remark}\label{V3-R1} Throughout this article the letter $C$ will denote diverse positive constants which may change from line to line and depend on parameters which are clearly established in each case.
\end{remark}

\textbf{Acknowledgement}. The authors thank A.V. Faminskii for his generosity in pointing out the need to have three boundary conditions in order for the problem to be well posed.

\section{Functional spaces and notion of solution of the IBVP}\label{V2-S1}

For $s\in\mathbb R$, $H^s:=H^s(\mathbb R)$ will denote the Sobolev space of order $s$ and $L^2$ type, with norm
$$\|u\|^2_{H^s}:=\int_{\mathbb R}(1+|\xi|)^{2s}|\widehat u(\xi)|^2 d\xi,$$
where $\widehat u$ is the Fourier transform of $u$.\\

For $s\geq 0$, and $\mathbb R^+:=(0,+\infty)$ we define the Sobolev space $H^s(\mathbb R^+)$, as the set of restrictions to the interval $(0,+\infty)$, of the elements of $H^s(\mathbb R)$; i.e.,
$$H^s(\mathbb R^+):=\{f:\text{there exists }F\in H^s(\mathbb R)\text{ such that }f=\left.F\right|_{\mathbb R^+}\}$$
(see \cite{CK2002}, Section 2).\\

The $H^s(\mathbb R^+)$ norm is defined by
$$\|f\|_{H^s(\mathbb R^+)}:=\inf\{\|F\|_{H^s(\mathbb R)}:F\in H^s(\mathbb R)\text{ and }F(x)=f(x)\text{ a.e. }x\in\mathbb R^+\}.$$

If $f\in H^s(\mathbb R^+)$ for some $s>\frac12$, let us take an extension $F\in H^s(\mathbb R)$ of $f$. By the Sobolev immersion theorem, $F$ is continuous on $\mathbb R$ and therefore $F(0)$ is well defined. This value is independent of the extension $F$ of $f$, and we will denote it by $f(0)$.\\

Now, let us define the spaces $H_0^s(\mathbb R^+)$, and $C^\infty_0(\mathbb R^+)$ by
\begin{align*}
H_0^s(\mathbb R^+)&:=\{f\in H^s(\mathbb R):\supp f\subset [0,+\infty)\},\\
C_0^\infty(\mathbb R^+)&:=\{f\in C_0^\infty(\mathbb R):\supp f\subset [0,+\infty)\}.
\end{align*}

The space $C^\infty_0(\mathbb R^+)$ is dense in the space $H^s_0(\mathbb R^+)$ (see \cite{CK2002}).\\

In fact, $f\in H_0^s(\mathbb R^+)$ if and only if $f\in H^s(\mathbb R)$, and $f=\chi_{(0,+\infty)}f$, where $\chi_{(0,+\infty)}$ is the characteristic function of the interval $(0,+\infty)$.\\

From now on, given $f\in H^s_0(\mathbb R^+)$, we will identify $f$ with its restriction to $\mathbb R^+$, $\left.f\right|_{\mathbb R^+}$.\\

The following lemma concerning extensions of elements of $H^s(\mathbb R^+)$ can be seen in \cite{CK2002}.

\begin{lemma}\label{V2-L1.1} \begin{enumerate}
\item[(i)] Let $0\leq s<\frac12$. There exists $C>0$ such that for every $h\in H^s(\mathbb R^+)$
\begin{align}
\|\chi_{(0,+\infty)}h\|_{H^s(\mathbb R)}\leq C\|h\|_{H^s(\mathbb R^+)}.\label{V2-1.1}
\end{align}
Here, $\chi_{(0,+\infty)}h$ is defined a.e. by $(\chi_{(0,+\infty)}h)(x)=0$ a.e. if $x<0$, and $(\chi_{(0,+\infty)}h)(x)=h(x)$ a.e. if $x>0$.\\

In particular for $0\leq s<\frac12$, and $h\in H^s(\mathbb R^+)$ we have that $\chi_{(0,+\infty)}h\in H^s(\mathbb R)$.
\item[(ii)] Let $\frac12<s<\frac32$. There exists $C>0$ such that for every $h\in H^s(\mathbb R^+)$, with $h(0)=0$,
\begin{align}
\|\chi_{(0,+\infty)}h\|_{H^s(\mathbb R)}\leq C\|h\|_{H^s(\mathbb R^+)}.\label{V2-1.2}
\end{align}

i.e., if $h\in H^s(\mathbb R^+)$, and $h(0)=0$, then $\chi_{(0,+\infty)}h\in H^s(\mathbb R)$, and \eqref{V2-1.2} holds.
\end{enumerate}
\end{lemma}

For $s,b,\alpha\in \mathbb R$ we will define the Bourgain spaces $X^{s,b}\equiv X^{s,b}(\mathbb R^2)$, and $Y^{s,-b,\alpha}\equiv Y^{s,-b,\alpha}(\mathbb R^2)$, associated to the fifth order KdV equation as the spaces of tempered distributions $u\in S'(\mathbb R^2)$, such that
\begin{align*}
\|u\|^2_{X^{s,b}}:=\int_{\mathbb R^2}\langle\xi\rangle^{2s}\langle\tau+\xi^5\rangle^{2b}|\widehat u(\xi,\tau)|^2d\xi d\tau<\infty,
\end{align*}
and
\begin{align*}
\|u\|_{Y^{s,-b,\alpha}}:= &\|\langle \xi\rangle^s \langle \tau+\xi^5\rangle^{-b} \widehat u(\xi,\tau)\|_{L^2_{\xi\tau}} + \|\chi(\xi)\langle \tau \rangle^{\alpha-1} \widehat u(\xi,\tau)\|_{L^2_{\xi\tau}}\\
&+\left( \int_{-\infty}^{+\infty} \langle \xi \rangle^{2s} \left(\int_{-\infty}^{+\infty} \frac{|\widehat u(\xi,\tau)|}{\langle \tau + \xi^5 \rangle} d\tau \right)^2 d\xi \right)^{\frac12}<+\infty,
\end{align*}

respectively, where $\widehat u$ is the Fourier transform of $u$, being $u$ a distribution of the spatial variable $x$, and the time variable $t$, $x\to \xi$, $t\to \tau$, and $\chi$ is the characteristic function of the interval $[-1,1]$. Here the symbol $\langle\cdot\rangle$ is an abbreviation of $(1+|\cdot|)$, i.e., $\langle\cdot\rangle:=1+|\cdot|.$\\

It can be seen that, if $b>\frac12$, then the space $X^{s,b}$ is continuously embedded into the space $C_b(\mathbb R_t,H^s(\mathbb R))$ of continuous and bounded functions from the variable $t$ into the space $H^s(\mathbb R_x)$.\\

To construct solutions $u=u(x,t)\in\mathbb R$ of the IBVP \eqref{maineq} we first consider the linear non homogeneous problem
\begin{align}
\left. \begin{array}{rlr}
u_t+\partial_x^5 u&\hspace{-2mm}=0,&\quad x\in\mathbb R^+,\; t\in\mathbb R^+,\\
u(x,0)&\hspace{-2mm}=0,&\\
u(0,t)=h_1(t),\, \partial_x u(0,t)=h_2(t),\, \partial_x^2u(0,t)&\hspace{-2mm}=h_3(t),
\end{array} \right\}\label{maineqlin}
\end{align}

where $h_{j+1}\in H^{\frac{s+2-j}5}(\mathbb R^+_t)$, $s\geq0$, $j=0,1,2$, with the compatibility conditions $h_1(0)=0$ if $\frac12<s<\frac32$, $h_1(0)=h_2(0)=0$ if $\frac32<s<\frac52$, and $h_1(0)=h_2(0)=h_3(0)$ for $s>\frac52$.\\

We will denote by $W_0^t(0,h_1,h_2,h_3)(\cdot_x,\cdot_t)$ a certain extension to $\mathbb R^2$ of the solution of problem \eqref{maineqlin} (See section \ref{ILBVP}).\\

For $\tilde g\in H^s(\mathbb R)$, we will denote by $[W_{\mathbb R}(t)\tilde g](x)$ the solution $u=u(x,t)\in\mathbb R$ of the linear Cauchy problem
\begin{align}
\left. \begin{array}{rlr}
u_t+\partial_x^5 u&\hspace{-2mm}=0,&\quad x\in\mathbb R,\; t\in\mathbb R,\\
u(x,0)&\hspace{-2mm}=\tilde g(x),&
\end{array} \right\}\label{V2-1.5}
\end{align}

i.e.,
\begin{align}
[W_{\mathbb R}(t)\tilde g](x)=\mathcal F_x^{-1}\left[e^{-it (\cdot_\xi)^5}\widehat{\tilde g}(\cdot_\xi) \right](x),\label{V2-1.6}
\end{align}

where $\,\widehat{\text{}}\,$ and $\mathcal F_x^{-1}$ denote the direct and the inverse Fourier transform, respectively.\\

Let us recall that $\{W_{\mathbb R}(t)\}_{t\in\mathbb R}$ is a group of isometries in $H^s(\mathbb R)$.\\

For $t\in\mathbb R$, $T>0$, let us consider the integral equation
\begin{align}
u(t)=\eta(t)W_{\mathbb R}(t)g_l+\eta(t)\int_0^t W_{\mathbb R}(t-t')F_T(u(t'))dt'+\eta(t)W_0^t(0,h_1-p_1,h_2-p_2,h_3-p_3)(t),\label{V2-1.7}
\end{align}

where
\begin{align}
\eta(\cdot_t)\in C_0^\infty(\mathbb R_t)\text{ is such that }\supp\eta\subset[-1,1]\text{ and }\eta\equiv1\text{ in }[-\tfrac12,\tfrac12],\label{V2-1.8}
\end{align}
$g_l\in H^s(\mathbb R_x)$ is an extension of $g$ in $H^s(\mathbb R_x^+)$ such that
\begin{align}
\|g_l\|_{H^s(\mathbb R_x)}&\leq C\|g\|_{H^s(\mathbb R^+_x)},\label{V2-1.9}\\
F_T(u(t'))&=\eta(\tfrac{t'}{2T})\left(-\frac12\partial_x(u(t'))^2\right),\label{V2-1.10}\\
p_{j+1}(t)&=\eta(t) \partial_x^j \left[W_{\mathbb R}(t)g_l \right](0)+ \eta(t) \partial_x^j \int_0^t [W_{\mathbb R}(t-t') F_T(u(t'))](0) dt',\quad j=0,1,2. \label{V2-1.11}
\end{align}

Here, for fixed $t$,
\begin{align*}
u(t)&\equiv u(\cdot_x,t),\\
W_{\mathbb R}(t)g_l&\equiv [W_{\mathbb R}(t)g_l](\cdot_x),\\
W_0^t(0,h_1-p_1,h_2-p_2,h_3-p_3)(t)&\equiv W_0^t(0,h_1-p_1,h_2-p_2,h_3-p_3)(\cdot_x,t),
\end{align*}
are functions of the spatial variable $x$.\\

Our main goal will be to prove that there exists $T$, $0<T\leq\frac12$, such that the integral equation \eqref{V2-1.7} has a unique solution $u$ in a certain Banach space, when $g$ and $h_{j+1}$, $j=0,1,2$, are taken in adequate Sobolev spaces of the interval $(0,+\infty)$, $g$ in the spatial variable $x$, and $h_{j+1}$, $j=0,1,2$ in the temporal variable $t$.\\

The restriction of $u$ to $\mathbb R^+\times[0,T]$ is called local solution in time of the IBVP \eqref{maineq}.\\

In Section \ref{V2-S4} it will be clear in which sense this restriction is solution of the IBVP \eqref{maineq}.

\section{An inhomogeneous linear initial-boundary value problem}\label{ILBVP}

In this section we deduce an explicit formula for the solution of the IBVP \eqref{maineqlin}, where $h_j\in C_0^\infty(\mathbb R^+_t)$, $j=1,2,3$. To do this, we formally apply to the equation \eqref{maineqlin} Laplace transform in the time variable $t$.\\

We will denote by $\tilde u(x,\cdot_\lambda)$ the Laplace transform of $u(x,\cdot_t)$, which is defined by
$$\tilde u(x,\lambda)\equiv \mathcal L\{u(x,\cdot_t)\}(\lambda)=\int_0^{+\infty} u(x,t)e^{-\lambda t}dt,$$
where $\lambda\in\mathbb C$ with $\text{Re }\lambda\geq 0$ (here $\text{Re }\lambda$ means real part of $\lambda$).\\

Applying Laplace transform to the equation in \eqref{maineqlin} we obtain
$$\int_0^{+\infty} \partial_t u(x,t) e^{-\lambda t}dt+\partial_x^5 \tilde u(x,\lambda)=0.$$

Using integration by parts in the integral of the previous equation it follows that
\begin{align}
e^{-\lambda t} u(x,t) \Big|_{t=0}^{t=+\infty}+\int_0^{+\infty}\lambda e^{-\lambda t} u(x,t)dt+\partial_x^5 \tilde u(x,\lambda)=0.\label{V2-2}
\end{align}

It is clear that
$$e^{-\lambda t}u(x,t)\Big|_{t=0}=u(x,0)=0.$$

Additionally, let us assume that
$$\lim_{t\to+\infty}e^{-\lambda t}u(x,t)=0.$$

Then, in this case, equation \eqref{V2-2} becomes
\begin{align}
\lambda \tilde u(x,\lambda)+\partial_x^5 \tilde u(x,\lambda)=0,\label{V2-3}
\end{align}

where $\tilde u$ is such that
\begin{align}
\tilde u(0,\lambda)=\tilde h_1(\lambda),\quad \partial_x \tilde u(0,\lambda)=\tilde h_2(\lambda),\quad \partial_x^2 \tilde u(0,\lambda)=\tilde h_3(\lambda).\label{V2-4}
\end{align}

Let us denote by $r_j=r_j(\lambda)$ ($j=1,2,3,4,5$) the roots of the characteristic equation $\lambda+r^5=0$. Then the general solution of equation \eqref{V2-3} is given by $\tilde u(x,\lambda)=\sum_{j=1}^5 c_j e^{r_j x}$, where $c_j$ is an arbitrary constant.\\

We are interested in solutions of \eqref{V2-3} that satisfy \eqref{V2-4}, such that $\re(r_j)\leq 0$; i.e., we are looking for solutions $\tilde u$ such that


\begin{align}
\tilde u(x,\lambda)=\sum_{\{j:\re r_j\leq 0\}}c_je^{r_j x},\label{V2-6}
\end{align}

and
\begin{align}
\sum_{\{j:\re r_j\leq 0\}} c_j=\tilde h_1(\lambda),\quad \sum_{\{j:\re r_j\leq 0\}} c_j r_j =\tilde h_2(\lambda),\quad \sum_{\{j:\re r_j\leq 0\}} c_j r_j^2 = \tilde h_3(\lambda).\label{V3-6}
\end{align}

Let us apply inverse Laplace transform in \eqref{V2-6}, to obtain formally the solution of the problem \eqref{maineqlin},
\begin{align}
u(x,t)=\frac1{2\pi i}\int_{-\infty i+\gamma}^{+\infty i+\gamma} e^{\lambda t} \sum_{\{j:\re r_j\leq 0\}} c_j e^{r_j(\lambda)x} d\lambda,\quad x>0,\quad t>0,\label{V2-8}
\end{align}
where $\lambda=i\beta+\gamma$, with fixed $\gamma>0$, and $\beta\in\mathbb R$.\\

Since $\lambda+r^5=0$, it is true that $(i\beta+\gamma)+r^5=0$. Thus, taking limit when $\gamma\to 0^+$, we have that $i\beta+r^5=0$, with $\re r\leq 0$. This way
\begin{align}
u(x,t)=\frac1{2\pi}\int_{-\infty}^{+\infty} e^{i\beta t} \sum_{\{j:\re r_j\leq 0\}} c_j(i\beta) e^{r_j(i\beta)x} d\beta,\quad x>0,\quad t>0.\label{V2-9}
\end{align}
For $\beta<0$, the roots of $i\beta+r^5=0$ with $\re r\leq 0$, are
\begin{align}
r_1=|\beta|^{\frac15}e^{i\frac\pi2}, \quad r_2=|\beta|^{\frac15}e^{i\frac{9\pi}{10}}, \quad r_3=|\beta|^{\frac15}e^{i\frac{13\pi}{10}},\label{V3-8.1}
\end{align}

while, the coefficients $c_j(i\beta)$ ($j=1,2,3$), are such that
\begin{align}
\sum_{j=1}^3 c_j(i\beta)=\tilde h_1(i\beta),\quad \sum_{j=1}^3 r_j c_j(i\beta)=\tilde h_2(i\beta), \quad \sum_{j=1}^3 r_j^2 c_j(i\beta)=\tilde h_3(i\beta).\label{V3-9}
\end{align}

For $\beta>0$, the roots of $i\beta+r^5=0$ with $\re r\leq 0$, are
\begin{align}
r_1=\beta^{\frac15}e^{i\frac{7\pi}{10}}, \quad r_2=\beta^{\frac15}e^{i\frac{11\pi}{10}}, \quad r_3=\beta^{\frac15}e^{i\frac{3\pi}{2}},\label{V3-8.2}
\end{align}

and the coefficients $c_j(i\beta)$ ($j=1,2,3$), are the corresponding solutions of the linear system \eqref{V3-9}. This system has unique solution since

\begin{align*}
\left|
\begin{array}{ccc}
1 & 1 & 1\\
r_1 & r_2 & r_3\\
r_1^2 & r_2^2 & r_3^2
\end{array}
\right|=(r_3-r_2)(r_3-r_1)(r_2-r_1)\neq 0.
\end{align*}

From \eqref{V2-9}, we have that
\begin{align}
\notag u(x,t)=&\frac1{2\pi}\int_{-\infty}^{0}e^{i\beta t} \left( c_1(i\beta) e^{|\beta|^{\frac15}e^{i\frac{\pi}{2}}x}+c_2(i\beta) e^{|\beta|^{\frac15}e^{i\frac{9\pi}{10}}x}+c_3(i\beta) e^{|\beta|^{\frac15}e^{i\frac{13\pi}{10}}x}\right)d\beta\\
&+\frac1{2\pi}\int_0^{+\infty}e^{i\beta t} \left( c_1(i\beta) e^{\beta^{\frac15}e^{i\frac{7\pi}{10}}x}+c_2(i\beta) e^{\beta^{\frac15}e^{i\frac{11\pi}{10}}x}+c_3(i\beta) e^{\beta^{\frac15}e^{i\frac{3\pi}{2}}x}\right)d\beta.\label{V2-10}
\end{align}

Since
$$\tilde h_j(i\beta)=\int_{0}^{+\infty} e^{-i\beta t}h_j(t)dt=\int_{-\infty}^{+\infty} e^{-i\beta t}h_j(t)dt=\sqrt{2\pi}\,\widehat h_j(\beta),$$
and $\widehat h_j\in S(\mathbb R)$ (Schwartz space) (j=1,2,3), taking into account that $\re(r_j(i\beta))\leq 0$, then the function $u$ defined by \eqref{V2-10} for $x\geq 0$, $t\geq 0$, is well defined, is infinitely differentiable in $(0,+\infty)\times (0,+\infty)$, its derivatives can be obtained by means of differentiation under the integral sign, and in this manner it can be easily proved that $u$, given by \eqref{V2-10}, satisfies the differential equation $u_t(x,t)+\partial_x^5u(x,t)=0$ for $x\geq 0$, $t\geq 0$.\\

Let us see that $u$ satisfies the boundary conditions. Taking into account \eqref{V3-9}, it follows that

\begin{align*}
u(0,t)&=\frac1{{2\pi}}\int_{-\infty}^{0} e^{i\beta t}\left( c_1(i\beta)+c_2(i\beta)+c_3(i\beta)\right)d\beta+\frac1{{2\pi}}\int_{0}^{+\infty} e^{i\beta t}\left( c_1(i\beta)+c_2(i\beta)+c_3(i\beta)\right)d\beta\\
&=\frac1{{2\pi}}\int_{-\infty}^{+\infty} e^{i\beta t} \tilde h_1(i\beta) d\beta.
\end{align*}

But we have that
$$\tilde h_1(i\beta)=\int_0^{+\infty} e^{-i\beta t} h_1(t) dt=\int_{-\infty}^{+\infty} e^{-i\beta t} h_1(t) dt=\sqrt{2\pi}\, \widehat h_1(\beta).$$

Therefore
$$u(0,t)=\frac1{\sqrt{2\pi}}\int_{-\infty}^{+\infty} e^{i\beta t} \widehat h_1(\beta) d\beta= h_1(t).$$

Using differentiation under the integral sign and taking into account \eqref{V3-9}, it follows that $\partial_x u(0,t)=h_2(t)$ and $\partial_x^2 u(0,t)=h_3(t)$.\\

It remains to prove that $u$ given by \eqref{V2-10} is such that $u(x,0)=0$. We have that
\begin{align*}
u(x,0)=&\frac1{2\pi} \int_{-\infty}^0 \left( c_1(i\beta) e^{|\beta|^{\frac15}e^{i\frac{\pi}{2}}x}+c_2(i\beta) e^{|\beta|^{\frac15}e^{i\frac{9\pi}{10}}x}+c_3(i\beta) e^{|\beta|^{\frac15}e^{i\frac{13\pi}{10}}x} \right) d\beta\\
&+\frac1{2\pi} \int_0^{+\infty} \left( c_1(i\beta) e^{\beta^{\frac15}e^{i\frac{7\pi}{10}}x}+c_2(i\beta) e^{\beta^{\frac15}e^{i\frac{11\pi}{10}}x}+c_3(i\beta) e^{\beta^{\frac15}e^{i\frac{3\pi}{2}}x} \right) d\beta.
\end{align*}

Therefore, to see that $u(x,0)=0$, it is enough to prove that
\begin{align}
\int_{-\infty}^0 c_1(i\beta) e^{|\beta|^{\frac15}e^{i\frac{\pi}{2}}x} d\beta + \int_0^\infty c_1(i\beta) e^{\beta^{\frac15}e^{i\frac{7\pi}{10}}x} d\beta &=0,\label{V3-11}\\
\int_{-\infty}^0 c_2(i\beta) e^{|\beta|^{\frac15}e^{i\frac{9\pi}{10}}x} d\beta + \int_0^\infty c_2(i\beta) e^{\beta^{\frac15}e^{i\frac{11\pi}{10}}x} d\beta &=0,\label{V3-12}\\
\int_{-\infty}^0 c_3(i\beta) e^{|\beta|^{\frac15}e^{i\frac{13\pi}{10}}x} d\beta + \int_0^\infty c_3(i\beta) e^{\beta^{\frac15}e^{i\frac{3\pi}{2}}x} d\beta &=0.\label{V3-13}
\end{align}

We will prove equality \eqref{V3-11}, being the proof of \eqref{V3-12} and \eqref{V3-13} similar.\\

Making a change of variables, it can be seen that
\begin{align}
\int_{-\infty}^0 c_1(i\beta) e^{|\beta|^{\frac15}e^{i\frac{\pi}{2}}x} d\beta + \int_0^{+\infty} c_1(i\beta) e^{\beta^{\frac15}e^{i\frac{7\pi}{10}}x} d\beta = 5\left( \int_0^{+\infty} c_1(-i\beta^5) \beta^4 e^{\beta e^{i\frac{\pi}{2}}x} d\beta + \int_0^{+\infty} c_1(i\beta^5) \beta^4 e^{\beta e^{i\frac{7\pi}{10}}x}d\beta\right).\label{V3-14}
\end{align}

We will show that the expression in the parenthesis vanishes. From \eqref{V3-9}, \eqref{V3-8.1}, and \eqref{V3-8.2}, using Cramer's rule, it follows that

\begin{align}
\notag &\left( \int_0^{+\infty} c_1(-i\beta^5) \beta^4 e^{\beta e^{i\frac{\pi}{2}}x} d\beta + \int_0^{+\infty} c_1(i\beta^5) \beta^4 e^{\beta e^{i\frac{7\pi}{10}}x}d\beta\right)=\\
\notag &\int_0^{+\infty} \left(\frac{\beta^4 \tilde h_1(-i\beta^5) e^{i\frac{6\pi}{5}}}{(e^{i\frac{2\pi}{5}}-1)(e^{i\frac{4\pi}{5}}-1)}-\frac{\beta^3 \tilde h_2(-i\beta^5) e^{-i\frac{\pi}{10}}}{(e^{i\frac{2\pi}{5}}-1)^2}+\frac{\beta^2 \tilde h_3(-i\beta^5) e^{-i\pi}}{(e^{i\frac{2\pi}{5}}-1)(e^{i\frac{4\pi}{5}}-1)} \right) e^{\beta e^{i\frac{\pi}{2}}x} d\beta\\
&+\int_0^{+\infty} \left( \frac{\beta^4 \tilde h_1(i\beta^5) e^{i\frac{6\pi}{5}}}{(e^{i\frac{2\pi}{5}}-1)(e^{i\frac{4\pi}{5}}-1)}-\frac{\beta^3 \tilde h_2(i\beta^5) e^{-i\frac{3\pi}{10}}}{(e^{i\frac{2\pi}{5}}-1)^2}+\frac{\beta^2 \tilde h_3(i\beta^5) e^{-i\frac{7\pi}{5}}}{(e^{i\frac{2\pi}{5}}-1)(e^{i\frac{4\pi}{5}}-1)} \right)e^{\beta e^{i\frac{7\pi}{10}}x} d\beta=0.\label{V3-15}
\end{align}

In this manner it is enough to prove that:
\begin{align}
\int_0^{+\infty} \beta^4 \tilde h_1(-i\beta^5) e^{\beta e^{i\frac{\pi}{2}}x} d\beta + \int_0^{+\infty} \beta^4 \tilde h_1(i\beta^5) e^{\beta e^{i\frac{7\pi}{10}}x} d\beta =&0, \label{V3-16}\\
\int_0^{+\infty} \beta^3 \tilde h_2(-i\beta^5) e^{-i\frac{\pi}{10}}e^{\beta e^{i\frac{\pi}{2}}x} d\beta + \int_0^{+\infty} \beta^3 \tilde h_2(i\beta^5) e^{-i\frac{3\pi}{10}}e^{\beta e^{i\frac{7\pi}{10}}x} d\beta =&0,\text{ and} \label{V3-17}\\
\int_0^{+\infty} \beta^2 \tilde h_3(-i\beta^5) e^{-i\pi} e^{\beta e^{i\frac{\pi}{2}}x} d\beta + \int_0^{+\infty} \beta^2 \tilde h_3(i\beta^5) e^{-i\frac{7\pi}{5}}e^{\beta e^{i\frac{7\pi}{10}}x} d\beta =&0. \label{V3-18}
\end{align}

Since $h_j\in C_0^{\infty}(\mathbb R^+)$ ($j=1,2,3$), and $h'_1\in C_0^\infty(\mathbb R^+)$, there exist $\epsilon>0$, and $K>0$, such that
\begin{align*}
|h_j(t)|\leq K e^{-\epsilon t}\quad \forall t\geq 0,\quad \forall j=1,2,3,\quad \text{and}\quad |h'_1(t)|\leq Ke^{-\epsilon t}\quad \forall t\geq 0.
\end{align*}
In consequence (see \cite{B1973}, Theorem 5.1, p. 206), the Laplace transforms of $h_j$ ($j=1,2,3$), and $h'_1$, which we will denote by $\tilde h_j$, and $\widetilde {h'_1}$, respectively, are analytic functions in the half-plane $\{z\in\mathbb C:\text{Re }z>-\epsilon\}$, and
\begin{align}
|\tilde h_j(z)|\leq\frac K{\text{Re }z+\epsilon}\quad (j=1,2,3)\quad \text{and}\quad |\widetilde {h'_1}(z)|\leq\frac K{\text{Re }z+\epsilon}.\label{V2-12}
\end{align}

Besides, by properties of the Laplace transform,
\begin{align}
\widetilde {h'_1}(z)=z\tilde h_1(z).\label{V3-19}
\end{align}

To prove \eqref{V3-16}, it is enough to prove that
\begin{align}
\int_{0}^{+\infty} \beta^4 \tilde h_1(-i\beta^5) e^{\beta e^{i\frac{\pi}{2}}x} d\beta=-\int_{0}^{+\infty} \beta^4 \tilde h_1(i\beta^5) e^{\beta e^{i\frac{7\pi}{10}} x} d\beta.\label{V3-20}
\end{align}

Without loss of generality we can assume that $x=1$. Let us consider the analytic function $f_1$, defined by
$$f_1(z)=z^4 \tilde h_1(iz^5)e^{e^{i\frac{7\pi}{10}}z},$$
and, for $R>0$, the closed path
\begin{align}
\Gamma_R=I_R+C_R+II_R,\label{V3-21}
\end{align}

where
\begin{align*}
I_R:&\quad z(x)=x,\quad 0\leq x\leq R,\\
C_R:&\quad z(\theta)=Re^{-i\theta},\quad 0\leq \theta\leq \frac\pi5,\\
II_R:&\quad z(t)=(R-t)e^{-i\frac\pi5},\quad 0\leq t\leq R.
\end{align*}
Since $\displaystyle{\int_{\Gamma_R}f_1(z)dz=0}$, we have that
\begin{align}
\int_{I_R}f_1(z)dz=-\int_{C_R}f_1(z)dz-\int_{II_R}f_1(z)dz.\label{V2-13}
\end{align}

Let us observe that
\begin{align}
\int_{I_R}f_1(z)dz&=\int_0^R f_1(\beta)d\beta=\int_0^R \beta^4 \tilde h_1(i\beta^5) e^{e^{i\frac{7\pi}{10}}\beta}d\beta\label{V2-14},\\
\notag-\int_{II_R}f_1(z)dz&=\int_0^R f_1(\beta e^{-i\frac\pi5})d(\beta e^{-i\frac\pi5})=\int_0^R \beta^4 e^{-i\frac{4\pi}{5}} \tilde h_1(i\beta^5 e^{-i\pi}) e^{e^{i\frac{7\pi}{10}}e^{-i\frac{\pi}{5}}\beta}e^{-i\frac{\pi}{5}}d\beta\\
&=-\int_0^R \beta^4 \tilde h (-i\beta^5) e^{e^{i\frac{\pi}{2}}\beta}d\beta.\label{V2-15}
\end{align}

Therefore, if we manage to prove that
$$\lim_{R\to+\infty}\left| \int_{C_R}f_1(z)dz\right|=0,$$
taking into account \eqref{V2-13}, \eqref{V2-14}, and \eqref{V2-15}, we would conclude that \eqref{V3-20} with $x=1$ is true.\\

Let us estimate the integral of $f_1$ on $C_R$. Taking into account \eqref{V3-19}, and \eqref{V2-12}, we have\\
\begin{align}
\notag\left|\int_{C_R} f_1(z)dz \right|&=\left|\int_0^{\pi/5} R^4 e^{-4i\theta}\tilde h_1(iR^5 e^{-5i\theta}) e^{e^{i\frac{7\pi}{10}} e^{-i\theta}R} (-iRe^{-i\theta}) d\theta \right|=\left|\int_0^{\pi/5} iR^5 e^{-5i\theta}\tilde h_1(iR^5 e^{-5i\theta}) e^{e^{i\frac{7\pi}{10}} e^{-i\theta}R}  d\theta \right|\\
\notag&=\left|\int_0^{\pi/5} \tilde{h'_1}(iR^5 e^{-5i\theta}) e^{e^{i\frac{7\pi}{10}}e^{-i\theta}R} d\theta\right|\leq \int_0^{\pi/5} e^{R\cos(\frac{7\pi}{10}-\theta)}|\tilde{h'_1}(iR^5 e^{-5i\theta})|d\theta\\
\notag&\leq \int_0^{\pi/5} e^{R\cos(\frac{7\pi}{10}-\theta)}\frac {K}{\re(iR^5 e^{-5i\theta})+\epsilon} d\theta = \int_0^{\pi/5} e^{R\cos(\frac{7\pi}{10}-\theta)}\frac {K}{R^5 \sin(5\theta)+\epsilon} d\theta\\
\notag&=\int_{\pi/2}^{7\pi/10} e^{R\cos u} \frac K{R^5 \sin(\frac{7\pi}2-5u)+\epsilon} du=\int_0^{\pi/5} e^{-R \sin t} \frac K{R^5 \sin(\pi-5t)+\epsilon} dt\\
&=\int_0^{\pi/5} e^{-R\sin t} \frac K{R^5 \sin(5t)+\epsilon} dt \to 0,\label{V3-25}
\end{align}

when $R\to +\infty$, which proves \eqref{V3-16} for $x=1$. To prove \eqref{V3-17}, and \eqref{V3-18}, we proceed in a similar way, as we did for proving \eqref{V3-16}. Again, it is enough to consider $x=1$, and the analytic functions
$$f_2(z):=z^3 \tilde h_2(iz^5) e^{e^{i\frac{7\pi}{10}}z},\quad f_3(z):=z^2 \tilde h_2(iz^5) e^{e^{i\frac{7\pi}{10}}z},$$
respectively. In this manner we have proved that \eqref{V3-11} is true. Since \eqref{V3-12}, and \eqref{V3-13} are also valid, we conclude that $u(x,0)=0$; i.e., $u(x,t)$ given by \eqref{V2-10} is a classical solution of the IBVP \eqref{maineqlin}.\\

From \eqref{V2-10}, we can write
\begin{align}
u(x,t)=u_1(x,t)+u_2(x,t)+u_3(x,t),\label{V3-26}
\end{align}

where
\begin{align}
u_1(x,t)\equiv u_{11}(x,t)+u_{12}(x,t),\label{V3-27}
\end{align}

with
\begin{align}
u_{11}(x,t)=&\frac1{\sqrt{2\pi}} \int_{-\infty}^0 e^{i\beta t} \widehat h_1(\beta) \frac{e^{i\frac{6\pi}{5}}e^{|\beta|^{\frac15}e^{i\frac\pi2} x}}{(e^{i\frac{2\pi}5}-1)(e^{i\frac{4\pi}5}-1)} d\beta + \frac1{\sqrt{2\pi}} \int_0^{+\infty} e^{i\beta t} \widehat h_1(\beta) \frac{e^{\beta^{\frac15}e^{i\frac{3\pi}2 }x}}{(e^{i\frac{2\pi}5}-1)(e^{i\frac{4\pi}5}-1)} d\beta,\label{V3-28}\\
\notag u_{12}(x,t)=&\frac1{\sqrt{2\pi}} \int_{-\infty}^0 e^{i\beta t} \widehat h_1(\beta) \left\{ -\frac{e^{i\frac{2\pi}{5}}e^{|\beta|^{\frac15}e^{i\frac{9\pi}{10}} x}}{(e^{i\frac{2\pi}5}-1)^2} + \frac{e^{|\beta|^{\frac15}e^{i\frac{13\pi}{10}} x}}{(e^{i\frac{2\pi}5}-1)(e^{i\frac{4\pi}5}-1)}\right\} d\beta\\
\notag &+ \frac1{\sqrt{2\pi}} \int_0^{+\infty} e^{i\beta t} \widehat h_1(\beta) \left\{ \frac{e^{i\frac{6\pi}{5}}e^{\beta^{\frac15}e^{i\frac{7\pi}{10}} x}}{(e^{i\frac{2\pi}5}-1)(e^{i\frac{4\pi}5}-1)} - \frac{e^{i\frac{2\pi}5}e^{\beta^{\frac15}e^{i\frac{11\pi}{10}} x}}{(e^{i\frac{2\pi}5}-1)^2}\right\} d\beta\\
 \equiv&\frac1{\sqrt{2\pi}} \int_{-\infty}^{+\infty} e^{i\beta t} \widehat h_1(\beta) d_{12}(\beta,x) d\beta,\label{V3-29}
\end{align}

\begin{align}
u_2(x,t)\equiv u_{21}(x,t) + u_{22}(x,t),\label{V3-30}
\end{align}

with
\begin{align}
 u_{21}(x,t) =& \frac1{\sqrt{2\pi}} \int_{-\infty}^0 e^{i\beta t} \frac{\widehat h_2(\beta)}{|\beta|^{\frac15}} \left(\frac{-e^{|\beta|^{\frac15}e^{i\frac\pi2}x}}{(e^{i\frac{2\pi}5}-1)^2 e^{i\frac\pi{10}}} \right) d\beta + \frac1{\sqrt{2\pi}} \int_0^{+\infty} e^{i\beta t} \frac{\widehat h_2(\beta)}{\beta^{\frac15}} \left(\frac{-e^{\beta^{\frac15}e^{i\frac{3\pi}2}x}}{(e^{i\frac{2\pi}5}-1)^2 e^{i\frac{11\pi}{10}}} \right) d\beta,\label{V3-31}\\
\notag u_{22}(x,t)=&\frac1{\sqrt{2\pi}} \int_{-\infty}^0 e^{i\beta t} \frac{\widehat h_2(\beta)}{|\beta|^{\frac15}} \left\{ \frac{(e^{i\frac{4\pi}5}+1)e^{|\beta|^{\frac15}e^{i\frac{9\pi}{10}} x}}{(e^{i\frac{2\pi}5}-1)^2 e^{i\frac{9\pi}{10}}} - \frac{e^{|\beta|^{\frac15}e^{i\frac{13\pi}{10}} x}}{(e^{i\frac{2\pi}5}-1)^2 e^{i\frac{9\pi}{10}}}\right\} d\beta\\
\notag &+ \frac1{\sqrt{2\pi}} \int_0^{+\infty} e^{i\beta t} \frac{\widehat h_2(\beta)}{\beta^{\frac15}} \left\{ -\frac{e^{\beta^{\frac15}e^{i\frac{7\pi}{10}}x}}{(e^{i\frac{2\pi}5}-1)^2 e^{i\frac{3\pi}{10}}} + \frac{(e^{i\frac{4\pi}5}+1) e^{\beta^{\frac15} e^{i\frac{11\pi}{10} }x}}{(e^{i\frac{2\pi}5}-1)^2 e^{i\frac{11\pi}{10}}}\right\} d\beta\\
 \equiv&\frac1{\sqrt{2\pi}} \int_{-\infty}^{+\infty} e^{i\beta t} \frac{\widehat h_2(\beta)}{|\beta|^{\frac15}} d_{22}(\beta,x) d\beta,\label{V3-32}
\end{align}

\begin{align*}
u_3(x,t)\equiv u_{31}(x,t) + u_{32}(x,t),
\end{align*}

with
\begin{align*}
 u_{31}(x,t) =& \frac1{\sqrt{2\pi}} \int_{-\infty}^0 e^{i\beta t} \frac{\widehat h_3(\beta)}{|\beta|^{\frac25}} \left(\frac{e^{|\beta|^{\frac15}e^{i\frac\pi2}x}}{(e^{i\frac{2\pi}5}-1)(e^{i\frac{4\pi}5}-1)}e^{i\pi} \right) d\beta + \frac1{\sqrt{2\pi}} \int_0^{+\infty} e^{i\beta t} \frac{\widehat h_3(\beta)}{\beta^{\frac25}} \left(\frac{e^{\beta^{\frac15}e^{i\frac{3\pi}2}x}}{(e^{i\frac{2\pi}5}-1)(e^{i\frac{4\pi}5}-1) e^{i\frac{9\pi}{5}}} \right) d\beta,\\
\notag u_{32}(x,t)=&\frac1{\sqrt{2\pi}} \int_{-\infty}^0 e^{i\beta t} \frac{\widehat h_3(\beta)}{|\beta|^{\frac25}} \left\{ -\frac{e^{|\beta|^{\frac15}e^{i\frac{9\pi}{10} }x}}{(e^{i\frac{2\pi}5}-1)^2 e^{i\frac{7\pi}{5}}} + \frac{e^{|\beta|^{\frac15}e^{i\frac{13\pi}{10}} x}}{(e^{i\frac{2\pi}5}-1)(e^{i\frac{4\pi}5}-1) e^{i\frac{7\pi}{5}}}\right\} d\beta\\
\notag &+ \frac1{\sqrt{2\pi}} \int_0^{+\infty} e^{i\beta t} \frac{\widehat h_3(\beta)}{\beta^{\frac25}} \left\{ \frac{e^{\beta^{\frac15}e^{i\frac{7\pi}{10}}x}}{(e^{i\frac{2\pi}5}-1)(e^{i\frac{4\pi}5}-1) e^{i\frac{7\pi}{5}}} - \frac{e^{\beta^{\frac15} e^{i\frac{11\pi}{10} }x}}{(e^{i\frac{2\pi}5}-1)^2 e^{i\frac{9\pi}{5}}}\right\} d\beta\\
 \equiv&\frac1{\sqrt{2\pi}} \int_{-\infty}^{+\infty} e^{i\beta t} \frac{\widehat h_3(\beta)}{|\beta|^{\frac25}} d_{32}(\beta,x) d\beta.
\end{align*}

Let us observe that the expressions for $u_{11}(x,t)$, $u_{21}(x,t)$, and $u_{31}(x,t)$ make sense for $(x,t)\in\mathbb R^2$, while the expressions for $u_{12}(x,t)$, $u_{22}(x,t)$, and $u_{32}(x,t)$ just make sense for $x$ in a bounded from below set, and $t\in\mathbb R$.\\

Let us consider a function $\rho\in C^\infty(\mathbb R)$ such that $0\leq \rho\leq 1$, $\supp \rho\subset[-2,+\infty)$ and $\rho\equiv 1$ in $[0,+\infty)$. Let us extend $u_{12}$, $u_{22}$, and $u_{33}$ to $\mathbb R^2$ by

\begin{align}
U_{12}(x,t)&=\frac1{\sqrt{2\pi}} \int_{-\infty}^{+\infty} e^{i\beta t} \widehat h_1(\beta) d_{12}(\beta,x) \rho(|\beta|^{\frac15}x) d\beta,\label{V3-33}\\
U_{22}(x,t)&=\frac1{\sqrt{2\pi}} \int_{-\infty}^{+\infty} e^{i\beta t} \frac{\widehat h_2(\beta)}{|\beta|^{\frac15}} d_{22}(\beta,x) \rho(|\beta|^{\frac15}x) d\beta,\label{V3-34}\\
U_{32}(x,t)&=\frac1{\sqrt{2\pi}} \int_{-\infty}^{+\infty} e^{i\beta t} \frac{\widehat h_3(\beta)}{|\beta|^{\frac25}} d_{32}(\beta,x) \rho(|\beta|^{\frac15}x) d\beta,\label{V3-35}
\end{align}
respectively.\\

Now we consider the extension of $u$ to $\mathbb R^2$, that we will denote by $W_0^t(0,h_1,h_2,h_3)$, defined by
\begin{align}
W_0^t(0,h_1,h_2,h_3)(x,t):=u_{11}(x,t)+U_{12}(x,t)+u_{21}(x,t)+U_{22}(x,t)+u_{31}(x,t)+U_{32}(x,t),\label{N3.38}
\end{align}

for $(x,t)\in\mathbb R^2$.\\

Next we prove two lemmas. The first of them, related with Kato type regularizing properties for the group $\{W_{\mathbb R}(t)\}_{t\in\mathbb R}$ associated to the linear part of the fifth order KdV equation, tells us what is the regularity that the boundary data $h_{j+1}$, $j=0,1,2$, must have when the initial data belongs to $H^s(\mathbb R^+)$. In part (i) of the second lemma, we prove that the extension $W_0^t(0,h_1,h_2,h_3)$ of the solution of the IBVP \eqref{maineqlin} is a continuous function from the variable $t$ with values in $H^s(\mathbb R_x)$, $s\geq 0$, i.e., $W_0^t(0,h_1,h_2,h_3)\in C(\mathbb R_t; H^s(\mathbb R_x))$. In part (ii) of the second lemma we see that the functions $\partial_x^j W_0^t(0,h_1,h_2,h_3)$, $j=0,1,2$, are elements of $C(\mathbb R_x;H^{\frac{s+2-j}{5}}(0,T))$ for $0< T\leq \frac12$, respectively. This allows us to give sense later to the boundary conditions, when $h_{j+1}\in H^{\frac{s+2-j}{5}}(\mathbb R_t^+)$, $j=0,1,2$, are such that $\chi_{(0,+\infty)}h_{j+1}\in H_0^{\frac{s+2-j}{5}}(\mathbb R_t^+)$, $j=0,1,2$. In the proof of Lemma \ref{V3-L2}, part (ii), the Kato type regularizing property of the group given in Lemma \ref{V3-L1} is used.

\begin{lemma}\label{V3-L1} Let $s\geq 0$, $g\in H^s(\mathbb R_x)$, and $\eta(\cdot_t)\in C_0^\infty(\mathbb R)$ such that $\supp \eta\subset[-1,1]$ and $\eta\equiv 1$ in $[-\frac12,\frac12]$. Then, for $j=0,1,2$,
$$\eta(\cdot_t)\partial_x^j [W_{\mathbb R}(\cdot_t)g](\cdot_x)\in C(\mathbb R_x;H^{\frac{s+2-j}5}(\mathbb R_t)),$$
and
\begin{align}
\|\eta(\cdot_t)\partial_x^j[W_{\mathbb R}(\cdot_t)g](\cdot_x)\|_{C(\mathbb R_x;H^{\frac{s+2-j}5}(\mathbb R_t))}\leq C_{s,j}\|g\|_{H^s(\mathbb R_x)},\label{V3-37}
\end{align}
with $C_{s,j}>0$ independent of $g$.
\end{lemma}

\begin{proof} It is easy to see that, for fixed $x\in\mathbb R$,
\begin{align*}
\mathcal F_t\left( \eta(\cdot_t)\partial_x^j[W_{\mathbb R}(\cdot_t)g](x)\right)(\tau)=&C\int_{|\xi|<1}e^{ix\xi} (i\xi)^j \widehat g(\xi)\widehat\eta(\tau+\xi^5)d\xi+C\int_{-\infty}^{-1}e^{ix\xi} (i\xi)^j \widehat g(\xi)\widehat\eta(\tau+\xi^5)d\xi\\
&+C\int_{1}^{+\infty}e^{ix\xi} (i\xi)^j \widehat g(\xi)\widehat\eta(\tau+\xi^5)d\xi \equiv I(x,\tau)+II(x,\tau)+III(x,\tau).
\end{align*}

Then
\begin{align}
\|\eta(\cdot_t) \partial_x^j[W_{\mathbb R}(\cdot_t)g](x)\|_{H^{\frac{s+2-j}{5}}}\leq \|\langle \cdot_\tau \rangle^{\frac{s+2-j}{5}} I(x,\cdot_\tau) \|_{L^2_\tau}+ \|\langle \cdot_\tau \rangle^{\frac{s+2-j}{5}} II(x,\cdot_\tau) \|_{L^2_\tau}+ \|\langle \cdot_\tau \rangle^{\frac{s+2-j}{5}} III(x,\cdot_\tau) \|_{L^2_\tau}.\label{V3-40}
\end{align}

Making the change of variables $\tau'=\tau+\xi^5$, taking into account that, for $|\xi|<1$, $\langle \tau-\xi^5\rangle^{\frac{2s+4-2j}5}\leq 2^{\frac{2s+4-2j}5}\langle\tau\rangle^{\frac{2s+4-2j}5}$, by using the Cauchy Schwarz inequality, it is possible to establish that
\begin{align}
\|\langle \cdot_\tau\rangle^{\frac{s+2-j}5}I(x,\cdot_\tau)\|_{L^2_\tau}\leq C_{s,j} \int_{|\xi|<1}|\widehat g(\xi)|d\xi\leq C_{s,j}\left(\int_{|\xi|<1}d\xi \right)^{\frac12} \left(\int_{|\xi|<1}|\widehat g(\xi)|^2d\xi \right)^{\frac12}\leq C_{s,j}\|g\|_{H^s(\mathbb R_x)}.\label{V3-41}
\end{align}

To estimate $\|\langle \cdot_\tau \rangle^{\frac{s+2-j}{5}} II(x,\cdot_\tau) \|_{L^2_\tau}$ we make the change of variables $\gamma=-\xi^5$, and we define
$$
k_j(\tau):=\langle\tau\rangle^{\frac{s+2-j}5}|\widehat\eta(\tau)|,\quad
l(\tau):=\left\{
\begin{array}{lr}
|\tau|^{\frac{s-2}5} |\widehat g(\sqrt[5]{-\tau})| &\text{for }\tau\geq 1,\\
0&\text{for }\tau<1,
\end{array}
\right.
$$

in order to conclude that
\begin{align}
\|\langle \cdot_\tau \rangle^{\frac{s+2-j}{5}} II(x,\cdot_\tau) \|_{L^2_\tau}&\leq \|(k_j\ast l)\|_{L^2_\tau}\leq \|k_j\|_{L^1_\tau}\|l\|_{L^2_\tau}\leq C_{s,j}\left(\int_1^{+\infty}|\tau|^{\frac{2(s-2)}5}|\widehat g(\sqrt[5]{-\tau})|^2 d\tau\right)^{\frac12}\leq C_{s,j}\|g\|_{H^s(\mathbb R_x)}.\label{V3-44}
\end{align}

In a similar way,
\begin{align}
\|\langle \cdot_\tau \rangle^{\frac{s+2-j}{5}} III(x,\cdot_\tau) \|_{L^2_\tau}\leq C_{s,j}\|g\|_{H^s(\mathbb R_x)}.\label{V3-44b}
\end{align}
From \eqref{V3-40}, \eqref{V3-41}, \eqref{V3-44}, and \eqref{V3-44b}; we conclude that
\begin{align}
\|\eta(\cdot_t)\partial_x^j [W_{\mathbb R}(\cdot_t)g](x)\|_{H^{\frac{s+2-j}5}(\mathbb R_t)}\leq C_{s,j}\|g\|_{H^s(\mathbb R_x)}.\label{V3-45}
\end{align}

Using the uniform estimate \eqref{V3-45} (taking into account that $C_{s,j}$ is independent of $x$), and the dominated convergence theorem, it follows that
$$\eta(\cdot_t)\partial_x^j[W_{\mathbb R}(\cdot_t)g](\cdot_x)\in C(\mathbb R_x;H^{\frac{s+2-j}5}(\mathbb R_t)).$$

Besides, from \eqref{V3-45} follows \eqref{V3-37}. Lemma \ref{V3-L1} is proved.
\end{proof}

\begin{lemma}\label{V3-L2} Let $s\geq0$, and $h_{j+1}\in C_0^{\infty}(\mathbb R^+)$, $j=0,1,2$. Then
\begin{enumerate}
\item[(i)] $W_0^t(0,h_1,h_2,h_3)\in C(\mathbb R_t,H^s(\mathbb R_x))$, and there exists $C>0$, independent of $h_{j+1}$, $j=0,1,2$, and of $t$, such that, for every $t\in\mathbb R$,
\begin{align}
\|W_0^t(0,h_1,h_2,h_3)(\cdot_x,t)\|_{H^s(\mathbb R_x)}\leq C\left( \sum_{j=0}^2 \|h_{j+1}\|_{H^{\frac{s+2-j}5}(\mathbb R_t)}\right).\label{V3-46}
\end{align}
\item[(ii)] If $\eta(\cdot_t)\in C_0^\infty(\mathbb R)$ is as in Lemma \ref{V3-L1}, then $\eta(\cdot_t) \partial_x^j W_0^t(0,h_1,h_2,h_3)(\cdot_x,\cdot_t) \in C(\mathbb R_x,H^{\frac{s+2-j}5}(\mathbb R_t))$, $j=0,1,2$, and there exists $C>0$, independent of $h_{j+1}$, $j=0,1,2$, and of $x$, such that, for every $x\in\mathbb R$,
\begin{align}
\|\eta(\cdot_t) \partial_x^j W_0^t(0,h_1,h_2,h_3)(x,\cdot_t)\|_{H^{\frac{s+2-j}5}(\mathbb R_t)}\leq C \left( \sum_{j=0}^2 \|h_{j+1}\|_{H^{\frac{s+2-j}5}(\mathbb R_t)}\right).\label{V3-47}
\end{align}
\end{enumerate}
\end{lemma}

\begin{proof} To illustrate the proof method, we will only provide estimates for the terms
\begin{align*}
(W_1 h_1)(x,t):=&u_{11}(x,t),\\
(W_2 h_2)(x,t):=& \frac1{\sqrt{2\pi}} \int_{-\infty}^0 e^{i\beta t} \frac{\widehat h_2(\beta)}{|\beta|^{\frac15}} \frac{(e^{i\frac{4\pi}{5}}+1) e^{|\beta|^{\frac15}e^{i\frac{9\pi}{10}}x}}{(e^{i\frac{2\pi}5}-1)^2 e^{i\frac{9\pi}{10}}} \rho(|\beta|^{\frac15} x) d\beta\\
=& C\int_{-\infty}^0 e^{i\beta t} \frac{\widehat h_2(\beta)}{|\beta|^{\frac15}} e^{-|\beta|^{\frac15}x \cos(\frac\pi{10})}e^{i|\beta|^{\frac15}x \sin(\frac\pi{10})} \rho(|\beta|^{\frac15}x) d\beta,
\end{align*}
as the estimates for the other terms comprising $W_0^t(0,h_1,h_2,h_3)(x,t)$ are similar.\\

\begin{enumerate}
\item[(i)] For $t\in\mathbb R$ let us estimate $\|(W_1h_1)(\cdot_x,t)\|_{H^s(\mathbb R_x)}$. Defining functions $\psi_1$ and $\psi_2$ through their Fourier transforms, respectively by
$$\widehat \psi_1(\beta):=\beta^4 \widehat h_1(-\beta^5) \chi_{[0,+\infty)}(\beta),\quad  \text{and} \quad \widehat \psi_2(\beta):= \beta^4 \widehat h_1(\beta^5) \chi_{[0,+\infty)}(\beta),$$

from the definition of $(W_1h_1)(x,t)$, we see that
\begin{align}
(W_1h_1)(x,t)=u_{11}(x,t)=C_{11} [W_{\mathbb R}(t) \psi_1](x) + C_{22} [W_{\mathbb R}(-t) \psi_2](-x).\label{V3-48}
\end{align}

This way
\begin{align}
\notag \|(W_1h_1)(\cdot_x,t)\|_{H^s(\mathbb R_x)}\leq & C \left\{ \left( \int_{-\infty}^{+\infty} \langle \beta \rangle^{2s} \beta^8 |\widehat h_1(-\beta^5)|^2 d\beta \right)^{\frac12} +  \left( \int_{-\infty}^{+\infty} \langle \beta \rangle^{2s} \beta^8 |\widehat h_1(\beta^5)|^2 d\beta \right)^{\frac12} \right\}\\
\leq & C \left( \int_{-\infty}^{+\infty} \langle \beta^{\frac15} \rangle^{2s} \beta^{\frac85} |\widehat h_1(\beta)|^2 \beta^{-\frac45} d\beta \right)^{\frac12}\leq C \|h_1\|_{H^{\frac{s+2}{5}}(\mathbb R_t)}. \label{V3-51}
\end{align}

Since the map $t\mapsto W_{\mathbb R}(t) g$ for $g\in H^s(\mathbb R_x)$ fixed, is continuous from $\mathbb R_t$ into $H^s(\mathbb R_x)$, then, from \eqref{V3-48}, we conclude that $W_1h_1\in C(\mathbb R_t; H^s(\mathbb R_x))$.\\

Let us estimate now $\| (W_2h_2)(\cdot_x,t) \|_{H^s(\mathbb R_x)}$. Let $g$ be the Schwartz space function defined by $g(x):=e^{-x\cos(\frac\pi{10})}\rho(x)$. Taking into account the change of variables $\beta'=-\beta^5$, and defining a function $\psi$ through its Fourier transform by
\begin{align*}
\widehat \psi(\beta):= \beta^3 \widehat h_2(-\beta^5) \chi_{(0,+\infty)}(\beta),
\end{align*}

we can write
\begin{align*}
(W_2h_2)(x,t) =& C\int_{-\infty}^{+\infty} e^{-i\beta^5 t} \widehat h_2(-\beta^5) g(\beta x) e^{i\beta x \sin(\frac\pi{10})} \beta^3 \chi_{(0,+\infty)}(\beta) d\beta\\
=&C \int_{-\infty}^{+\infty} g(\beta x) e^{i\beta x \sin(\frac\pi{10})} \mathcal F(W_{\mathbb R}(t)\psi)(\beta) d\beta.
\end{align*}

Let us observe that
\begin{align*}
\|\psi\|_{H^s(\mathbb R)} = \left( \int_0^{+\infty} \langle \beta \rangle^{2s} \beta^6 |\widehat h_2(-\beta^5)|^2 d\beta \right)^{\frac12} = C \left( \int_{-\infty}^0 \langle \beta^{\frac15} \rangle^{2s} \beta^{\frac65} |\widehat h_2(\beta)|^2 \frac1{\beta^{\frac45}} d\beta \right)^{\frac12}\leq  C \|h_2\|_{H^{\frac{s+1}5}(\mathbb R_t)}.
\end{align*}

Taking into account that $\|W_{\mathbb R}(t) \psi\|_{H^s}=\|\psi\|_{H^s}$, we would conclude that
\begin{align}
\| (W_2h_2)(\cdot_x,t) \|_{H^s(\mathbb R_x)}\leq C \|h_2\|_{H^{\frac{s+1}{5}}(\mathbb R_t)},\label{V3-54}
\end{align}

with $C$ independent of $t$, if we prove that the linear operator $T$ defined by
$$(T\varphi)(x) := \int_{\mathbb R} g(\beta x) e^{i\beta x \sin(\frac\pi{10})} \widehat \varphi(\beta) d\beta,$$

is bounded from $H^s(\mathbb R_x)$ into $H^s(\mathbb R_x)$, for $s\geq 0$.\\

For $s=0$, making the change of variables $\beta'=\beta x$, we have that
\begin{align*}
\|T\varphi\|_{L^2_x} \leq & \int_{\mathbb R} |g(\beta')| \left( \int_\mathbb R   x^{-2} |\widehat\varphi(x^{-1}\beta')|^2 dx  \right)^{1/2}d\beta'=\int_{\mathbb R}|g(\beta')| \left( \int_\mathbb R   \frac1{|\beta'|}|\widehat\varphi(y)|^2 dy  \right)^{1/2}d\beta'\leq C\|\varphi\|_{L^2_x}.
\end{align*}

For $s\in\mathbb N$, noticing that
\begin{align*}
\frac{d^s}{dx^s}(T\varphi)(x)=\int_{\mathbb R} \theta^{(s)}(\beta x) \beta^s \widehat \varphi(\beta) d\beta,
\end{align*}

where $\theta(x) := g(x)e^{ix\sin(\frac\pi{10})}$, and therefore, for $x\neq 0$, $\theta^{(s)}(x\cdot_{\beta})(\cdot_\beta)^s$ is a Schwartz function of the variable $\beta$, it can be seen that $T:H^s(\mathbb R_x)\to H^s(\mathbb R_x)$ is a continuous operator.\\

For $s>0$ the proof that $T: H^s(\mathbb R_x)\to H^s(\mathbb R_x)$ is continuous follows by interpolation.\\

This way we can affirm that the map $t\mapsto (W_2h_2)(\cdot_x,t)$ is continuos from $\mathbb R_t$ into $H^s(\mathbb R_x)$, since it is the composition of the continuous map $t\mapsto W_{\mathbb R}(t)\varphi$ from $\mathbb R_t$ into $H^s(\mathbb R_x)$ and the continuous map $T:H^s(\mathbb R_x)\to H^s(\mathbb R_x)$.\\

We have proved that $W_2h_2\in C(\mathbb R_t; H^s(\mathbb R_x))$, and that \eqref{V3-54} is valid, with $C$ independent of $t$.\\

\item[(ii)] For $x\in\mathbb R$ fixed, let us estimate $\| \eta(\cdot_t) \partial_x^j(W_1h_1)(x,\cdot_t)\|_{H^{\frac{s+2-j}5}(\mathbb R_t)}$, $j=0,1,2$.\\

Using \eqref{V3-48} it is clear that
$$\eta(\cdot_t) \partial_x^j(W_1h_1)(x,\cdot_t) = C_{11} \eta(\cdot_t) \partial_x^j[W_{\mathbb R}(\cdot_t) \psi_1](x) + C_{22} \eta(\cdot_t) \partial_x^j[W_{\mathbb R}({\scriptstyle -}\cdot_t)\psi_2](-x),$$

with $\psi_1$, $\psi_2$ in $H^s(\mathbb R_x)$ as in part (i). Hence, by Lemma \ref{V3-L1}, and \eqref{V3-51}, $\eta(\cdot_t) \partial_x^j(W_1h_1)(\cdot_x,\cdot_t)\in C(\mathbb R_x; H^{\frac{s+2-j}{5}}(\mathbb R_t))$, $j=0,1,2$, and
\begin{align}
\| \eta(\cdot_t) \partial_x^j(W_1h_1)(x,\cdot_t)\|_{H^{\frac{s+2-j}{5}}(\mathbb R_t)} \leq & C_{s,j} \left( \|\psi_1\|_{H^s(\mathbb R_x)} + \|\psi_2\|_{H^s(\mathbb R_x)} \right) \leq C_{s,j} \|h_1\|_{H^{\frac{s+2}5}(\mathbb R_t)}.\label{V3-59}
\end{align}

For $x\in\mathbb R$ fixed, let us estimate $\|\eta(\cdot_t) \partial_x^j (W_2h_2)(x,\cdot_t)\|_{H^{\frac{s+2-j}{5}}(\mathbb R_t)}$, $j=0,1,2$.\\

Using the expressions for $W_2h_2(x,t)$ and $\psi$ given in part (i), and defining $\theta(x) :=  g(x) e^{i \sin(\frac\pi{10}) x}$, we have that
\begin{align*}
\partial_x^j W_2h_2(x,t) = C \partial_x^j \int_{-\infty}^{+\infty} \theta(\beta x) e^{-i \beta^5 t}\widehat\psi(\beta)d\beta = C\int_{-\infty}^{+\infty} \theta^{(j)} (\beta x) \beta^j e^{-i\beta^5 t} \widehat \psi(\beta) d\beta,
\end{align*}

Then, for $x\in\mathbb R$ fixed,

\begin{align}
\notag\mathcal F_t[\eta(\cdot_t) \partial_x^j (W_2h_2)(x,\cdot_t)](\tau) = &\frac1{\sqrt{2\pi}}\int_{-\infty}^{+\infty} e^{-i\tau t} \eta(t) \partial_x^j (W_2h_2)(x,t)dt= C\int_{\mathbb R}\theta^{(j)}(\beta x) \beta^j \widehat\psi(\beta) \widehat\eta(\tau+\beta^5)d\beta\\
\notag= & C\int_{|\xi|<1} \theta^{(j)}(x\xi) \xi^j \widehat\psi(\xi)\widehat\eta(\tau+\xi^5)d\xi+C\int_{|\xi|\geq 1}\theta^{(j)}(x\xi) \xi^j \widehat\psi(\xi)\widehat\eta(\tau+\xi^5)d\xi\\
\equiv & I(x,\tau)+II(x,\tau).\label{V3-60}
\end{align}

This way, for $x\in\mathbb R$,
\begin{align}
\left\|\eta(\cdot_t) \partial_x^j (W_2h_2)(x,\cdot_t) \right\|_{H^{\frac{s+2-j}5}(\mathbb R_t)}\leq \left\| \langle \cdot_\tau\rangle^{\frac{s+2-j}5}I(x,\cdot_\tau) \right\|_{L^2_\tau}+\left\| \langle \cdot_\tau\rangle^{\frac{s+2-j}5}II(x,\cdot_\tau) \right\|_{L^2_\tau}.\label{V3-61}
\end{align}

Proceeding as in the proof of Lemma \ref{V3-L1}, it can be seen that 
\begin{align}
\left\| \langle \cdot_\tau\rangle^{\frac{s+2-j}5}I(x,\cdot_\tau) \right\|_{L^2_\tau}\leq C_{sj} \|h_2\|_{H^{\frac15}(\mathbb R_t)}\leq C_{sj} \|h_2\|_{H^{\frac{s+1}5(\mathbb R_t)}}.\label{V3-62},
\end{align}

and
\begin{align}
\|\langle \cdot_\tau\rangle^{\frac{s+2-j}5}II(x,\cdot_\tau)\|_{L^2_\tau}\leq C_{s,j} \|h_2\|_{H^{\frac{s+1}5}(\mathbb R_t)}.\label{V3-64}
\end{align}

Using \eqref{V3-61} and the uniform estimates \eqref{V3-62} and \eqref{V3-64} (taking into account that $C_{s,j}$ is independent of $x$) and the dominated convergence theorem, we conclude that $\eta(\cdot_t)\partial_x^j (W_2h_2)(\cdot_x,\cdot_t)\in C(\mathbb R_x; H^{\frac{s+2-j}5}(\mathbb R_t))$.\\

Part (ii) is proved.

\end{enumerate}

\end{proof}

\begin{remark}\label{V3-R2} Let $s\geq 0$ and $h_{j+1}\in H^{\frac{s+2-j}5}(\mathbb R^+_t)$, $j=0,1,2$, such that $\chi_{(0,+\infty)}h_{j+1}\in H_0^{\frac{s+2-j}5}(\mathbb R_t)$. Since $C_0^\infty(\mathbb R_t^+)$ is dense in $H_0^{\frac{s+2-j}5}(\mathbb R_t^+)$, there exists a sequence $\{h_{j+1,n}\}_{n\in\mathbb N}$ in $C_0^\infty (\mathbb R_t^+)$ such that $h_{j+1,n}\to \chi_{(0,+\infty)}h_{j+1}$ in $H^{\frac{s+2-j}5}(\mathbb R_t)$.\\

From \eqref{V3-46} and \eqref{V3-47} in Lemma \ref{V3-L2}, we conclude that
$$\{W_0^t(0,h_{1,n},h_{2,n}, h_{3,n})\}_{n\in\mathbb N}\quad \text{and} \quad \{\eta(\cdot_t, \partial_x^j W_0^t(0,h_{1,n}, h_{2,n}, h_{3,n})\}_{n\in\mathbb N},\quad j=0,1,2,$$
are Cauchy sequences in $C(\mathbb R_t;H^s(\mathbb R_x))$ and $C(\mathbb R_x;H^{\frac{s+2-j}5}(\mathbb R_t))$, respectively.\\

Let us define $W_0^t(0,h_1,h_2,h_3)$ in $C(\mathbb R_t;H^s(\mathbb R_x))$ by
$$W_0^t(0,h_1,h_2,h_3):=\lim_{n\to+\infty}W_0^t(0,h_{1,n}, h_{2,n}, h_{3,n}) \text{ in }C(\mathbb R_t;H^s(\mathbb R_x)),$$
and $\eta(\cdot_t) \partial_x^j W_0^t(0,h_1,h_2,h_3)$ in $C(\mathbb R_x;H^{\frac{s+2-j}5}(\mathbb R_t))$, $j=0,1,2$, by
$$\eta(\cdot_t) \partial_x^j W_0^t(0,h_1,h_2,h_3):=\lim_{n\to+\infty}\eta(\cdot_t) \partial_x^j W_0^t(0,h_{1,n},h_{2,n},h_{3,n}) \text{ in }C(\mathbb R_x;H^{\frac{s+2-j}5}(\mathbb R_t)).$$

It is clear that if $h_{j+1}\in H^{\frac{s+2-j}5}(\mathbb R_t^+)$, $j=0,1,2$, are such that $\chi_{(0,+\infty)}h_{j+1}\in H_0^{\frac{s+2-j}5}(\mathbb R^+_t)$, then, from the former convergences, and \eqref{V3-46} and \eqref{V3-47} in Lemma \ref{V3-L2}, it follows that
\begin{align}
\|W_0^t(0,h_1,h_2,h_3)(\cdot_x,t)\|_{H^s(\mathbb R_x)} \leq & C \sum_{j=0}^2 \|\chi_{(0,+\infty)} h_{j+1}\|_{H^{\frac{s+2-j}5}(\mathbb R_t)},\text{ and}\label{V3-66}\\
\|\eta(\cdot_t) \partial_x^j W_0^t(0,h_1,h_2,h_3)(x,\cdot t)\|_{H^{\frac{s+2-j}5}(\mathbb R_t)} \leq & C \sum_{j=0}^2 \|\chi_{(0,+\infty)} h_{j+1}\|_{H^{\frac{s+2-j}5}(\mathbb R_t)}.\label{V3-67}
\end{align}
\end{remark}

We finish this section by estimating the norm of $\eta(\cdot_t)W_0^t(0,h_1,h_2,h_3)$ in the modified Bourgain space $X^{s,b,\alpha}$.

\begin{lemma}\label{V3-L4} Let $s\geq0$, $0<b\leq\frac12$, $\frac12<\alpha<\frac35$, and $\eta(\cdot_t)\in C_0^\infty(\mathbb R)$ as in Lemma \ref{V3-L1}. For $j=0,1,2$, let us assume that $h_{j+1}\in H^{\frac{s+2-j}5}(\mathbb R_t^+)$ are such that $\chi_{(0,+\infty)}h_{j+1}\in H_0^{\frac{s+2-j}5}(\mathbb R_t^+)$. Then there exists $C>0$, not depending on $h_{j+1}$, $j=0,1,2$, such that
\begin{align}
\|\eta(\cdot_t)W_0^t(0,h_1,h_2,h_3)(\cdot_x,\cdot_t)\|_{X^{s,b,\alpha}}\leq C \left( \sum_{j=0}^2 \|\chi_{(0,+\infty)}h_{j+1}\|_{H^{\frac{s+2-j}5}(\mathbb R_t)}\right).\label{V3-70}
\end{align}
\end{lemma}

\begin{proof} As in the proof of Lemma \ref{V3-L2}, we will prove estimative \eqref{V3-70} for $W_2h_2$, and $W_1h_1$ instead of $W_0^t(0,h_1,h_2,h_3)$, where $W_jh_j$ are defined within the proof of Lemma \ref{V3-L2}. Since $C_0^\infty(\mathbb R_t^+)$ is dense in $H^{\frac{s+2-j}5}_0(\mathbb R_t^+)$, we will assume that $h_1,h_2,h_3\in C^\infty_0(\mathbb R_t^+)$.

\begin{enumerate}
\item[(i)] Let us prove that
\begin{align}
\|\eta(\cdot_t)(W_2h_2)(\cdot_x,\cdot_t)\|_{X^{s,b}}\leq C\|\chi_{(0,+\infty)}h_2\|_{H^{\frac{s+1}5}(\mathbb R_t)},\label{V3-71}
\end{align}

According to the definition of $W_2h_2$, we have that
$$(W_2h_2)(x,t):=C\int_{-\infty}^{+\infty} e^{-i\beta^5 t} \theta(x\beta) \widehat\psi(\beta) d\beta,$$

where
\begin{align*}
\theta(\beta)=&g(\beta)e^{i\beta \sin(\frac\pi{10})}=e^{-\beta \cos(\frac\pi{10})}\rho(\beta)e^{i\beta\sin(\frac\pi{10})}, \text{ and}\\
\widehat\psi(\beta)=&\beta^3 \widehat h_2(-\beta^5)\chi_{(0,+\infty)}(\beta).
\end{align*}

Let us prove estimate \eqref{V3-71} when $s=0$ and $b=\tfrac12$. Let us observe that
\begin{align*}
(\eta(\cdot_t)(W_2h_2)(\cdot_x,\cdot_t))^\wedge(\xi,\tau)=C\int_0^{+\infty} \widehat\eta(\tau+\beta^5)\widehat\theta\left( \tfrac\xi\beta\right)\beta^2 \widehat h_2(-\beta^5)d\beta.
\end{align*}
Since $\widehat\theta$ is a Schwartz space function, then
$$\left|\widehat \theta\left(\tfrac\xi\beta\right) \right|\leq \frac C{1+\frac{|\xi|^5}{|\beta|^5}}=C\frac{|\beta|^5}{|\beta|^5+|\xi|^5}.$$

Besides, taking into account that $\widehat \eta \in S(\mathbb R)$ (Schwartz space), given $\tilde \alpha \in \mathbb N$, there exists $C_{\tilde \alpha}$ such that
$$|\widehat \eta (\tau - \beta^5)| \leq C_{\tilde \alpha} \langle \tau - \beta^5 \rangle^{-{\tilde\alpha}}.$$
Therefore, making the change of variables $\beta'=-\beta$, and observing that $\langle \tau +\xi^5 \rangle \leq 2 \langle \tau - \beta'^5 \rangle \langle \beta'^5 + \xi^5\rangle$, we can conclude that
\begin{align*}
\|\eta(\cdot_t) (W_2h_2)(\cdot_x, \cdot_t)\|_{X^{0,\frac12}} \leq C_{\tilde\alpha} \left\| \int_{-\infty}^0 \langle \cdot_\tau - \beta^5 \rangle^{\frac12} \langle \beta^5 + (\cdot_\xi)^5 \rangle^{\frac12} \langle \cdot_\tau - \beta^5 \rangle^{-{\tilde\alpha}} \frac{|\beta|^7}{|\beta|^5 + |\cdot_\xi|^5} |\widehat h_2(\beta^5)| d\beta  \right\|_{L^2_{\xi\tau}}
\end{align*}

We will consider the sets
\begin{align*}
A_1:=\{\beta\in (-\infty,0):|\beta|^5+|\xi|^5\geq 1\},\quad A_2:=\{\beta\in (-\infty,0):|\beta|^5+|\xi|^5< 1\}.
\end{align*}

If $\beta\in A_1$, we have that
$$\langle \beta^5 + \xi^5 \rangle  \leq 2(|\beta|^5 + |\xi|^5),$$
therefore
\begin{align*}
&\left\| \int_{A_1} \langle \cdot_\tau - \beta^5 \rangle^{\frac12} \langle \beta^5 + (\cdot_\xi)^5 \rangle^{\frac12} \langle \cdot_\tau - \beta^5 \rangle^{-{\tilde\alpha}} \frac{|\beta|^7}{|\beta|^5 + |\cdot_\xi|^5} |\widehat h_2(\beta^5)| d\beta \right\|_{L^2_{\xi\tau}}\\
&\leq C \left\| \int_{A_1} \langle \cdot_\tau - \beta^5 \rangle^{\frac12-{\tilde\alpha}}  \frac{|\beta|^7}{(|\beta|^5 + |\cdot_\xi|^5)^{\frac12}} |\widehat h_2(\beta^5)| d\beta \right\|_{L^2_{\xi\tau}}\\
&\leq C \left\| \int_{A_1} \langle \cdot_\tau - \beta^5 \rangle^{\frac12-{\tilde\alpha}} |\beta|^7 \left\| \frac{1}{(|\beta|^5 + |\cdot_\xi|^5)^{\frac12}}\right\|_{L^2_\xi} |\widehat h_2(\beta^5)| d\beta \right\|_{L^2_{\tau}}
\end{align*}

Let us observe that
\begin{align*}
\left\| \frac1{(|\beta|^5+|\cdot_\xi|^5)^{\frac12}}\right\|_{L^2_\xi}&\leq \left\{2\left( \int_0^{|\beta|}\frac1{|\beta|^5}d\xi+\int_{|\beta|}^{+\infty}\frac1{\xi^5}d\xi\right) \right\}^{\frac12}\leq C\frac1{|\beta|^2}.
\end{align*}

Hence, for $\tilde\alpha\geq 2$,
\begin{align*}
 &\left\|\int_{A_1}\langle\cdot_\tau-\beta^5\rangle^{\frac12-{\tilde \alpha}} |\beta|^7\left\|  \frac1{(|\beta|^5+|\xi|^5)^{\frac12}}\right\|_{L^2_\xi}|\widehat h_2(\beta^5)|d\beta \right\|_{L^2_\tau}\\
&\leq C\left\| \int_{-\infty}^0 \langle \cdot_\tau-\beta^5\rangle^{\frac12-{\tilde \alpha}}|\beta|^7 \frac1{\beta^2}  |\widehat h_2(\beta^5)|d\beta\right\|_{L^2_\tau}\leq C\left\| \int_{-\infty}^0 \langle \cdot_\tau-\gamma\rangle^{\frac12-{\tilde \alpha}} |\gamma| |\widehat h_2(\gamma)| \frac1{\gamma^{\frac45}}d\gamma\right\|_{L^2_\tau}\\
 &\leq C\left\| \int_{-\infty}^0 \langle \cdot_\tau-\gamma\rangle^{\frac12-{\tilde \alpha}}|\gamma|^{\frac15} |\widehat h_2(\gamma)| d\gamma\right\|_{L^2_\tau}\leq C\|\langle\cdot_\tau\rangle^{\frac12-{\tilde \alpha}}\|_{L^1_\tau}\| |\cdot_\tau|^{\frac15}| \widehat h_2 (\cdot_\tau)|\|_{L^2_\tau}\leq C_{\tilde \alpha} \| h_2\|_{H^{\frac15}(\mathbb R_t)}.
\end{align*}

If $\beta\in A_2$, then
$$\langle \beta^5 + \xi^5 \rangle  < 2,$$
hence
\begin{align*}
&\left\|\int_{A_2} \langle \cdot_\tau - \beta^5 \rangle^{\frac12} \langle \beta^5 + (\cdot_\xi)^5 \rangle^{\frac12} \langle \cdot_\tau - \beta^5 \rangle^{-{\tilde \alpha}} \frac{|\beta|^7}{|\beta|^5+|\cdot_\xi|^5} |\widehat h_2(\beta^5)| d\beta\right\|_{L^2_{\xi\tau}}\\
&\leq C\left\|  \int_{-1}^0 \langle\cdot_\tau-\beta^5\rangle^{\frac12-{\tilde \alpha}} |\beta|^7 \left\| \frac1{|\beta|^5+|\cdot_\xi|^5}  \right\|_{L^2_{|\xi|\leq 1}} |\widehat h_2(\beta^5)|d\beta \right\|_{L^2_{\tau}}.
\end{align*}

Since
\begin{align*}
\left\| \frac1{|\beta|^5+|\cdot_\xi|^5}\right\|_{L^2_{|\xi|\leq 1}}\leq C\frac1{|\beta|^{\frac92}},
\end{align*}

then
\begin{align*}
&\left\| \int_{A_2} \langle \cdot_\tau - \beta^5\rangle^{\frac12} \langle \beta^5 + (\cdot_\xi)^5 \rangle^{\frac12} \langle \cdot_\tau - \beta^5 \rangle^{-{\tilde \alpha}} \frac{|\beta|^7}{|\beta|^5 + |\cdot_\xi|^5} |\widehat h_2(\beta^5)| d\beta \right\|_{L^2_{\xi\tau}}\leq C\left\| \int_{-1}^0 \langle\cdot_\tau-\beta^5\rangle^{\frac12-{\tilde \alpha}} |\beta|^{\frac52} |\widehat h_2(\beta^5)| d\beta\right\|_{L^2_\tau}\\
&\leq C \left\| \int_{-1}^0 \langle\cdot_\tau-\gamma\rangle^{\frac12-{\tilde \alpha}}|\gamma|^{\frac 1{2}} | \widehat h_2(\gamma)| \frac1{\gamma^{\frac45}}d\gamma\right\|_{L^2_\tau}\leq C\int_{-1}^0 \left\| \langle\cdot_\tau-\gamma\rangle^{\frac12-{\tilde \alpha}}\right\|_{L^2_\tau}   |\gamma|^{-\frac3{10}} | \widehat h_2(\gamma)|  d\gamma.\\
\end{align*}

Taking into account that, for ${\tilde \alpha}>1$,
\begin{align*}
\|\langle\cdot_\tau-\gamma\rangle^{\frac12-{\tilde \alpha}}\|_{L^2_\tau}= C_{\tilde \alpha}<\infty,
\end{align*}

we obtain
\begin{align*}
&\left\| \int_{A_2} \langle \cdot_\tau-\beta^5\rangle^{\frac12} \langle \beta^5 + (\cdot_\xi)^5 \rangle^{\frac12} \langle \cdot_\tau - \beta^5 \rangle^{-{\tilde \alpha}} \frac{|\beta|^7}{|\beta|^5+|\cdot_\xi|^5} |\widehat h_2(\beta^5)| d\beta\right\|_{L^2_{\xi\tau}} \leq  C{\tilde \alpha} \int_{-1}^0 |\gamma|^{-\frac3{10}} |\widehat h_2(\gamma)|d\gamma\\
&\leq  C_{\tilde \alpha}\left(\int_{-1}^0|\gamma|^{-\frac35}d\gamma \right)^{\frac12}\left(\int_{-1}^0|\widehat h_2(\gamma)|^2d\gamma \right)^{\frac12}\leq  C_{\tilde\alpha} \|h_2\|_{H^{\frac15}(\mathbb R_t)},
\end{align*}

which proves \eqref{V3-71} for $s=0$, and $b=\frac12$.\\

For $s\in\mathbb N$, since
$$\partial_x^s (\eta(t)(W_2h_2)(x,t))= C \eta(t) \int_{-\infty}^{+\infty} e^{-i\beta^5 t} \theta^{(s)}(x\beta) \beta^5 \widehat \psi(\beta) d\beta,$$

the proof of estimate \eqref{V3-71} is similar to the case $s=0$.\\

When $s>0$ is arbitrary, estimate \eqref{V3-71} is obtained by interpolation.\\

Let us prove now that

\begin{align}
\left( \int_{-\infty}^{+\infty} \int_{-1}^{1} \langle \tau \rangle^{2\alpha} |(\eta(\cdot_t) (W_2h_2)(\cdot_x,\cdot_t))^{\wedge}(\xi,\tau)|^2 d\xi d\tau \right)^{\frac12} \leq C  \| \chi_{(0,+\infty)} h_2\|_{H^{\frac{1}5}(\mathbb R_t)} \label{V4.2-1}
\end{align}

Since $\widehat \theta$, and $\widehat \eta$ are in $S(\mathbb R)$, let us observe that

\begin{align}
\notag&\left( \int_{-\infty}^{+\infty} \int_{-1}^1 \langle \tau \rangle^{2\alpha} |[\eta(\cdot_t)(W_2h_2)(\cdot_x,\cdot_t)]^{\wedge}(\xi,\tau)|^2 d\xi d\tau \right)^{\frac12}\\
\notag &=C \left(\int_{-\infty}^{+\infty} \int_{-1}^1 \langle \tau \rangle^{2\alpha} \left|\int_0^{+\infty} \widehat \eta(\tau + \beta^5) \widehat \theta(\tfrac\xi\beta) \beta^2 \widehat h_2(-\beta^5) d\beta \right|^2 d\xi d\tau \right)^{\frac12}\\
&\leq C \left(\int_{-\infty}^{+\infty} \int_{-1}^1 \left(\int_{-\infty}^0 \frac{\langle \tau \rangle^{\alpha}}{\langle \tau - \beta^5 \rangle^2} \frac{|\beta|^7}{|\beta|^5 + |\xi^5|} |\widehat h_2(\beta^5)|d\beta \right)^2 d\xi d\tau \right)^{\frac12}
\label{V4.2-2}
\end{align}

Let $B_1:=(-1,0)$, and $B_2:=(-\infty,-1]$. Let us estimate
\begin{align*}
\left( \int_{-\infty}^{+\infty} \int_{-1}^1 \left( \int_{B_j} \frac{\langle \tau \rangle^\alpha}{\langle \tau - \beta^5\rangle^2} \frac{|\beta|^7}{|\beta|^5+|\xi|^5} |\widehat h_2 (\beta^5)| d\beta \right)^2 d\xi d\tau \right)^{\frac12},\quad j=1,2.
\end{align*}

For the set $B_1$ we have that
\begin{align*}
&\left( \int_{-\infty}^{+\infty} \int_{-1}^1 \left( \int_{B_1} \frac{\langle \tau \rangle^\alpha}{\langle \tau - \beta^5\rangle^2} \frac{|\beta|^7}{|\beta|^5+|\xi|^5} |\widehat h_2 (\beta^5)| d\beta \right)^2 d\xi d\tau \right)^{\frac12}\\
&\leq C \left\| \int_{B_1} \frac{\langle \tau \rangle^\alpha}{\langle \tau - \beta^5 \rangle^2} |\beta|^7 \left\| \frac1{|\beta|^5+|\xi|^5} \right\|_{L^2_{|\xi|\leq 1}} |\widehat h_2(\beta^5)| d\beta \right\|_{L^2_\tau}.
\end{align*}
For $|\beta|<1$, it follows that
$$\left\| \frac1{|\beta|^5 + |\xi|^5} \right\|_{L^2_{|\xi|\leq 1}} \leq \frac C{|\beta|^{\frac52}},$$

therefore
\begin{align}
\notag&\left( \int_{-\infty}^{+\infty} \int_{-1}^1 \left( \int_{B_1} \frac{\langle \tau \rangle^\alpha}{\langle \tau - \beta^5\rangle^2} \frac{|\beta|^7}{|\beta|^5+|\xi|^5} |\widehat h_2 (\beta^5)| d\beta \right)^2 d\xi d\tau \right)^{\frac12}\leq C \left\| \int_{B_1} \frac{\langle \tau \rangle^\alpha}{\langle \tau - \beta^5 \rangle^2} \frac{|\beta|^7}{|\beta|^{\frac52}} |\widehat h_2(\beta^5)| d\beta \right\|_{L^2_\tau}\\
\notag&\leq C \left[ \int_{|\tau|<2} \left( \int_{B_1} \frac{\langle \tau \rangle^\alpha}{\langle \tau - \beta^5 \rangle^2} |\beta|^{\frac92} |\widehat h_2 (\beta^5)| d\beta \right)^2 d\tau + \int_{|\tau|>2} \left( \int_{B_1} \frac{\langle \tau \rangle^\alpha}{\langle \tau - \beta^5\rangle^2} |\beta|^{\frac92} |\widehat h_2(\beta^5)| d\beta \right)^2 d\tau \right]^{\frac12}\\
\notag&\leq C \left[ C_\alpha \int_{|\tau|<2} \left( \int_{B_1}  |\beta|^{\frac92} |\widehat h_2 (\beta^5)| d\beta \right)^2 d\tau + C \int_{|\tau|>2} \left( \int_{B_1} \langle \tau - \beta^5 \rangle^{\alpha-2} |\beta|^{\frac92} |\widehat h_2(\beta^5)| d\beta \right)^2 d\tau \right]^{\frac12}\\
\notag&\leq C \left[ C_\alpha \int_{|\tau|<2} \left( \int_{B_1}  |\gamma|^{\frac9{10}} \frac{|\widehat h_2 (\gamma)|}{|\gamma|^{\frac45}} d\gamma \right)^2 d\tau + \xi \int_{|\tau|>2} \left( \int_{B_1} \langle \tau - \gamma \rangle^{\alpha-2} |\gamma|^{\frac9{10}} \frac{|\widehat h_2(\gamma^5)|}{|\gamma|^{\frac45}} d\gamma \right)^2 d\gamma \right]^{\frac12}\\
\notag&\leq C\left[4C_\alpha \left( \int_{B_1} |\gamma|^{\frac1{10}} |\widehat h_2(\gamma)| d\gamma \right)^2 + C \int_{|\tau|>2} \left( \int_{B_1} \langle \tau-\gamma\rangle^{\alpha-2} |\gamma|^{\frac1{10}} |\widehat h_2(\gamma)| d\gamma \right)^2 d\tau \right]^{\frac12}\\
&\leq C\left[\tilde C_\alpha \|h_2\|^2_{H^{\frac1{10}}(\mathbb R_t)} + C \| \langle \cdot \rangle ^{\alpha-2} \ast_{\tau} (|\cdot|^{\frac1{10}} |\widehat h_2(\cdot)|)\|^2_{L^2_\tau} \right]^{\frac12}\leq C_\alpha \|h_2\|_{H^{\frac1{10}}(\mathbb R_t)}.\label{V4.2-3}
\end{align}

For the set $B_2$ we have
\begin{align}
& \left[ \int_{-\infty}^{+\infty} \int_{-1}^{1} \left( \int_{B_2} \frac{\langle \tau \rangle^\alpha}{\langle \tau-\beta^5 \rangle^2} \frac{|\beta|^7}{|\beta|^5+|\xi|^5} |\widehat h_2(\beta^5)| d\beta \right)^2 d\xi d\tau \right]^{\frac12}\\
\leq & \left[ \int_{-\infty}^{+\infty} \int_{-1}^{1} \left( \int_{-\infty}^{-1} \frac{\langle \tau \rangle^\alpha}{\langle \tau-\beta^5 \rangle^2} |\beta|^2 |\widehat h_2(\beta^5)| d\beta \right)^2 d\xi d\tau \right]^{\frac12}= C \left[ \int_{-\infty}^{+\infty}\left( \int_{-\infty}^{-1} \frac{\langle \tau \rangle^\alpha}{\langle \tau-\beta^5 \rangle^2} |\beta|^2 |\widehat h_2(\beta^5)| d\beta \right)^2 d\tau \right]^{\frac12}\\
\notag \leq &C \left[ \int_{|\tau|<2} \left( \int_{-\infty}^{-1} \frac{\langle \tau \rangle^\alpha}{\langle \tau-\gamma \rangle^2} \langle \gamma \rangle^{-\frac25} |\widehat h_2(\gamma)| d\gamma \right)^2 d\tau +  \int_{|\tau|>2} \left(\int_{-\frac{|\tau|}2}^{-1} \frac{\langle \tau \rangle^\alpha}{\langle \tau-\gamma\rangle^2} \langle \gamma\rangle^{-\frac25} |\widehat h_2(\gamma)| d\gamma \right. \right.\\
\notag & \left. \left. +\int_{-\frac{3|\tau|}2}^{-\frac{|\tau|}2} \frac{\langle \tau \rangle^\alpha}{\langle \tau-\gamma\rangle^2} \langle \gamma \rangle^{-\frac25} |\widehat h_2(\gamma)| d\gamma + \int_{-\infty}^{-\frac{3|\tau|}2} \frac{\langle \tau \rangle^\alpha}{\langle \tau - \gamma \rangle^2} \langle \gamma \rangle^{-\frac25} |\widehat h_2(\gamma) d\gamma |\right)^2 d\tau \right]^{\frac12}\\
\notag \leq &C \left[ \int_{|\tau|<2} \left( \int_{-\infty}^{-1} \frac{\langle \tau \rangle^{2\alpha}}{\langle \tau-\gamma \rangle^4} d\gamma \right) \left( \int_{-\infty}^{-1} \langle \gamma \rangle^{-\frac45} |\widehat h_2(\gamma)|^2 d\gamma \right)^2 d\tau + C \int_{|\tau|>2} \left(\int_{-\frac{|\tau|}2}^{-1} \frac{\langle \tau -\gamma \rangle^\alpha}{\langle \tau-\gamma\rangle^2} \langle \gamma\rangle^{-\frac25} |\widehat h_2(\gamma)| d\gamma \right. \right.\\
\notag & \left. \left. +\int_{-\frac{3|\tau|}2}^{-\frac{|\tau|}2} \frac{\langle \gamma \rangle^\alpha}{\langle \tau-\gamma\rangle^2} \langle \gamma \rangle^{-\frac25} |\widehat h_2(\gamma)| d\gamma + \int_{-\infty}^{-\frac{3|\tau|}2} \frac{\langle \tau - \gamma \rangle^\alpha}{\langle \tau - \gamma \rangle^2} \langle \gamma \rangle^{-\frac25} |\widehat h_2(\gamma) |  d\gamma \right)^2 d\tau \right]^{\frac12}\\
\notag\leq & C \left[ \int_{|\tau|<2} \left( \int_{-\infty}^{-1} \frac1{\langle \tau - \gamma \rangle^4} d\gamma  \right)\|h_2\|^2_{H^{-\frac25}(\mathbb R_t)} d\tau \right.\\
\notag&\left. + \int_{|\tau|>2} \left\{\left(  \frac1{\langle\cdot \rangle^{2-\alpha}} \ast \left( \langle \cdot \rangle^{-\frac25} |\widehat h_2(\cdot)|\right)\right)(\tau) + \left( \frac1{\langle \cdot \rangle^2} \ast \left( \langle \cdot\rangle^{\alpha-\frac25} |\widehat h_2(\cdot)|\right)\right) (\tau) \right\}^2  d\tau \right]^{\frac12}\\
\notag \leq & C \left[ \|h_2\|^2_{H^{-\frac25}(\mathbb R_t)} + \left\| \frac1{\langle \cdot \rangle^{2-\alpha}} \right\|^2_{L^1_\tau} \| \langle \cdot \rangle^{-\frac25} |\widehat h_2|\|^2_{L^2_\tau} + \left\| \frac1{\langle \cdot \rangle ^2 } \right\|^2_{L^1_\tau} \|\langle \cdot \rangle^{\alpha-\frac25} |\widehat h_2|\|^2\right]^{\frac12}\\
\leq & C \|\langle \cdot \rangle^{\alpha-\frac25} |\widehat h_2(\cdot)|\|_{L^2} = C\|h_2\|_{H^{\alpha-\frac25}(\mathbb R)}\leq C \| h_2 \|_{H^{\frac15}(\mathbb R_t)}, \label{V4.2-4}
\end{align}

if $\alpha-\frac25\leq \frac15$; i.e., if $\alpha \leq \frac35$.\\

From \eqref{V4.2-2}, \eqref{V4.2-3}, and \eqref{V4.2-4}, we conclude that

\begin{align}
\left( \int_{-\infty}^{+\infty} \int_{-1}^1 \langle \tau \rangle^{2\alpha} | [\eta(\cdot_t) (W_2h_2)(\cdot_x,\cdot_t) ]^\wedge (\xi,\tau) |^2 d\xi d\tau \right)^{\frac12} \leq C \|h_2\|_{H^{\frac15}(\mathbb R_t)}. \label{V4.2-5}
\end{align}

\item[(ii)] Let us prove that
\begin{align}
\|\eta(\cdot_t) (W_1h_1)(\cdot_x,\cdot_t) \|_{X^{s,b}} \leq C\|h_1\|_{H^{\frac{s+2}5}(\mathbb R_t)}.\label{V3-72}
\end{align}

As in (i), it is enough to prove \eqref{V3-72} for $s=0$, and $b=\frac12$. Let us recall the expression for $W_1h_1$ given by \eqref{V3-48}:
\begin{align*}
(W_1h_1)(x,t) = \tilde C_1 \int_0^{+\infty} e^{-i\beta^5t+i\beta x} \beta^4 \widehat h_1(-\beta^5) d\beta + \tilde C_2 \int_{-\infty}^0 e^{-i\beta^5 t + i\beta x} \beta^4 \widehat h_1(-\beta^5) d\beta.
\end{align*}

Then
\begin{align*}
(\eta(\cdot_t)(W_1h_1)(\cdot_x,\cdot_t))^\wedge(\xi,\tau)=&C_1\int_{\mathbb R}\int_{\mathbb R} e^{-i\xi x}e^{-i\tau t}\eta(t)\left( \int_0^\infty e^{-i\beta^5 t + i\beta x} \beta^4 \widehat h_1(-\beta^5) d\beta \right)dx dt\\
&+C_2\int_{\mathbb R}\int_{\mathbb R} e^{-i\xi x}e^{-i\tau t}\eta(t)\left( \int_{-\infty}^0 e^{-i\beta^5 t + i\beta x} \beta^4 \widehat h_1(-\beta^5) d\beta \right)dx dt\\
=& C_1 \int_{\mathbb R_x} e^{-i\xi x} \mathcal F_{\beta}^{-1} [(\cdot_\beta)^4 \widehat h_1({\scriptstyle -}(\cdot_\beta)^5) \chi_{(0,+\infty)}(\cdot_\beta) \widehat \eta(\tau + (\cdot_\beta)^5)](x) dx\\
& + C_2 \int_{\mathbb R_x} e^{-i\xi x} \mathcal F_{\beta}^{-1} [(\cdot_\beta)^4 \widehat h_1({\scriptstyle -}(\cdot_\beta)^5) \chi_{(-\infty,0)}(\cdot_\beta) \widehat \eta(\tau + (\cdot_\beta)^5)](x) dx\\
=& C_1 \xi^4 \widehat h_1(-\xi^5) \chi_{(0,+\infty)}(\xi) \widehat \eta(\tau + \xi^5) + C_2 \xi^4 \widehat h_1(-\xi^5) \chi_{(-\infty,0)}(\xi) \widehat \eta(\tau + \xi^5).
\end{align*}

Therefore
\begin{align*}
\|\eta(\cdot_t)(W_1h_1)(\cdot_x,\cdot_t)\|_{X^{0,\frac12}} \leq & C_1 \left( \int_{\mathbb R_\tau} \int_{\mathbb R_\xi} \langle \tau + \xi^5\rangle \xi^8 |\widehat h_1(-\xi^5)|^2 \chi_{(0,+\infty)}(\xi) |\widehat \eta(\tau + \xi^5)|^2 d\xi d\tau \right)^{\frac12}\\
& + C_2 \left( \int_{\mathbb R_\tau} \int_{\mathbb R_\xi} \langle \tau + \xi^5\rangle \xi^8 |\widehat h_1(-\xi^5)|^2 \chi_{(-\infty,0)}(\xi) |\widehat \eta(\tau + \xi^5)|^2 d\xi d\tau \right)^{\frac12}\\
\leq & \max\{C_1,C_2\} C_{\tilde\alpha} \left( \int_{\mathbb R_\tau} \int_{\mathbb R_\gamma} \langle \tau - \gamma \rangle |\gamma|^{\frac45} |\widehat h_1(\gamma)|^2 \langle \tau - \gamma\rangle^{-2\tilde\alpha} d\gamma d\tau \right)^{\frac12}\\
\leq & C_{\tilde\alpha} \|\langle \cdot_\tau \rangle^{\frac12-\tilde\alpha}\|_{L^1_\tau} \| |\cdot_\tau|^{\frac25} |\widehat h_1(\cdot_\tau)| \|_{L^2_\tau} \leq C_{\tilde\alpha} \|h_1\|_{H^{\frac25}(\mathbb R_t)},
\end{align*}

for $\tilde\alpha\geq 2$, which proves estimate \eqref{V3-72} for $s=0$, and $b=\frac12$.\\

Let us demonstrate now that
\begin{align}
\left( \int_{-\infty}^{+\infty} \int_{-1}^1 \langle \tau \rangle^{2\alpha} |[\eta(\cdot_t) (W_1h_1)(\cdot_x,\cdot_t)]^\wedge(\xi,\tau)|^2 d\xi d\tau \right)^{\frac12} \leq C \|  h_1 \|_{H^{\frac{2}5}(\mathbb R_t)}. \label{V4.2-6}
\end{align}

Let us remember that
\begin{align*}
[\eta(\cdot_t) (W_1h_1)(\cdot_x,\cdot_t)]^{\wedge} (\xi,\tau) & = C_1 \xi^4 \widehat h_1 (-\xi^5) \chi_{(0,+\infty)}(\xi) \widehat \eta(\tau + \xi^5) + C_2 \xi^4 \widehat h_1(-\xi^5) \chi_{(-\infty,0)}(\xi) \widehat \eta(\tau + \xi^5).
\end{align*}

Hence
\begin{align*}
& \left( \int_{-\infty}^{+\infty} \int_{-1}^1 \langle \tau \rangle^{2\alpha} | [\eta(\cdot_t)(W_1h_1)(\cdot_x,\cdot_t)]^\wedge (\xi,\tau) |^2 d\xi d\tau \right)^{\frac12}\\
\leq & C_1 \left( \int_{-\infty}^{+\infty} \int_0^1 \langle \tau \rangle^{2\alpha} \xi^8 |\widehat h_1(-\xi^5)|^2 |\widehat \eta(\tau+\xi^5)|^2 d\xi d\tau \right)^{\frac12} + C_2 \left( \int_{-\infty}^{+\infty} \int_{-1}^0 \langle \tau \rangle^{2\alpha} \xi^8 |\widehat h_1(-\xi^5)|^2 |\widehat \eta(\tau+\xi^5)|^2 d\xi d\tau \right)^{\frac12}\\
= & C_1 \left( \int_{-\infty}^{+\infty} \int_{-1}^0 \langle \tau \rangle^{2\alpha} |\gamma|^{\frac45} |\widehat h_1(\gamma)|^2 |\widehat \eta(\tau-\gamma)|^2 d\gamma d\tau \right)^{\frac12} + C_2 \left( \int_{-\infty}^{+\infty} \int_{0}^1 \langle \tau \rangle^{\alpha} |\gamma|^{\frac45} |\widehat h_1(\gamma)|^2 |\widehat \eta(\tau-\gamma)|^2 d\gamma d\tau \right)^{\frac12}.
\end{align*}

It is enough to estimate the first term on the right hand side of the former inequality, being the estimation of the second one similar.

\begin{align}
\notag&\left( \int_{-\infty}^{+\infty} \int_{-1}^0 \langle \tau \rangle^{2\alpha} |\gamma|^{\frac45} |\widehat h_1(\gamma)|^2 |\widehat \eta(\tau-\gamma)|^2 d\gamma d\tau \right)^{\frac12}\\
\notag&\leq C \left(\int_{|\tau|<2} \int_{-1}^0 \langle \tau \rangle^{2\alpha} |\gamma|^{\frac45} |\widehat h_1(\gamma)|^2 \langle \tau - \gamma\rangle^{-4} d\gamma d\tau\right)^{\frac12} + C \left(\int_{|\tau|>2} \int_{-1}^0 \langle \tau \rangle^{2\alpha} |\gamma|^{\frac45} |\widehat h_1(\gamma)|^2 \langle \tau - \gamma\rangle^{-4} d\gamma d\tau\right)^{\frac12}\\
\notag&\leq C \left(\int_{|\tau|<2} \int_{-1}^0 |\gamma|^{\frac45} |\widehat h_1(\gamma)|^2  d\gamma d\tau\right)^{\frac12} + C \left(\int_{|\tau|>2} \int_{-1}^0 \langle \tau-\gamma \rangle^{2\alpha-4} |\gamma|^{\frac45} |\widehat h_1(\gamma)|^2  d\gamma d\tau\right)^{\frac12}\\
\notag &\leq C \|h_1\|_{H^{\frac25}(\mathbb R_t)} + C \left(\int_{|\tau|>2} (\langle \cdot \rangle^{2\alpha-4}\ast (|\cdot|^{\frac45}|\widehat h_1(\cdot)|^2))(\tau) d\tau\right)^{\frac12}\\
&\leq C \|h_1\|_{H^{\frac25}(\mathbb R_t)} + C \|\langle \cdot \rangle^{2\alpha-4}\|_{L^1_\tau}^{\frac12} \| |\cdot|^{\frac45} |\widehat h_1(\cdot)|^2 \|_{L^1_\tau}^{\frac12}\leq C \|h_1\|_{H^{\frac25}(\mathbb R_t)},\label{V4.2-7}
\end{align}

which proves \eqref{V4.2-6}.
\end{enumerate}

Estimate \eqref{V3-70} follows from estimates \eqref{V3-71}, \eqref{V4.2-1}, \eqref{V4.2-5}, \eqref{V3-72}, and \eqref{V4.2-6}.\\

Lemma \ref{V3-L4} is proved.

\end{proof}

\section{Linear and non linear estimates}\label{LNLE}

In this section we establish estimates for the norms in the spaces $X^{s,b,\alpha}$, and $C(\mathbb R_x;H^{\frac{s+2}5}(\mathbb R_t))$ of the first two terms on the right hand side of integral equation \eqref{V2-1.7}.\\

\begin{lemma}\label{V2-L3.1} Let $s\geq 0$, $0<b<\frac12$, $\alpha\in(\frac12,1]$ and $\eta(\cdot_t)\in C_0^\infty(\mathbb R_t)$ as in Lemma \ref{V3-L1}. Then there exists $C>0$ such that, for every function $g\in H^s(\mathbb R_x)$,
\begin{align}
\|\eta(\cdot_t)[W_{\mathbb R}(\cdot_t)g](\cdot_x)\|_{X^{s,b,\alpha}(\mathbb R^2)}\leq C\|g\|_{H^s(\mathbb R_x)}.\label{V4.2-4.1}
\end{align}
\end{lemma}

\begin{proof}
Estimative \eqref{V4.2-4.1} follows from the following estimatives:

\begin{align}
\| \eta(\cdot_t) [W_{\mathbb R}(\cdot_t) g](\cdot_x)\|_{X^{s,b}(\mathbb R^2)} \leq C \|g\|_{H^s(\mathbb R_x)},\label{V4.2-4.2}
\end{align}
and

\begin{align}
\left( \int_{-\infty}^{+\infty} \int_{-1}^{1} |\tau|^{2\alpha} |[\eta(\cdot_t) [W_{\mathbb R}(\cdot_t)] g]^\wedge(\xi,\tau)|^2 d\xi d\tau \right)^{\frac12} \leq C \|g\|_{H^s(\mathbb R_x)}.\label{V4.2-4.3}
\end{align}

We begin by proving \eqref{V4.2-4.2}. Let us observe that
\begin{align*}
\left(\eta(\cdot_t)[W_{\mathbb R}(\cdot_t)g](\cdot_x)\right)^\wedge(\xi,\tau)=C\int_{\mathbb R_t} e^{-i(\tau+\xi^5)t}\eta(t)\widehat g(\xi) dt=C\widehat g(\xi)\widehat \eta(\tau+\xi^5).
\end{align*}

This way, taking into account that $\widehat \eta\in S(\mathbb R)$,
\begin{align*}
\|\eta(\cdot_t)[W_{\mathbb R}(\cdot_t)g](\cdot_x)\|_{X^{s,b}(\mathbb R^2)}&=C\left( \iint_{\mathbb R^2}\langle\xi\rangle^{2s}\langle\tau+\xi^5\rangle^{2b}|\widehat g(\xi)|^2 |\widehat\eta(\tau+\xi^5)|^2d\tau d\xi\right)^{\frac12}\\
&=C\left( \int_{\mathbb R_{\tau}} \langle\tau\rangle^{2b}|\widehat\eta(\tau)|^2\right)^{\frac12}\|g\|_{H^s(\mathbb R_x)}=C\|g\|_{H^s(\mathbb R_x)}.
\end{align*}

Now, we continue by proving \eqref{V4.2-4.3}:
\begin{align*}
&\left( \int_{-\infty}^{+\infty} \int_{-1}^{1} |\tau|^{2\alpha} |[ \eta(\cdot_t)[W_{\mathbb R}(\cdot_t)g](\cdot_x) ]^\wedge(\xi,\tau)|^2 d\xi d\tau \right)^{\frac12}= C \left( \int_{-\infty}^{+\infty} \int_{-1}^{1} |\tau|^{2\alpha} |\widehat \eta(\tau + \xi^5)|^2 |\widehat g(\xi)|^2 d\xi d\tau \right)^{\frac12}\\
& = C \left(\int_{-1}^1 |\widehat g(\xi)|^2 \left( \int_{-\infty}^{+\infty} |\tau' - \xi^5|^{2\alpha} |\widehat \eta (\tau')|^2 d\tau' \right) d\xi \right)^{\frac12}\leq C\left( \int_{-\infty}^{+\infty} \langle \tau'\rangle^{2\alpha} |\widehat \eta(\tau')|^2 d\tau' \right)^{\frac12} \|g\|_{L^2(\mathbb R_x)}\leq C \|g\|_{H^s(\mathbb R_x)}.
\end{align*}
Lemma \ref{V2-L3.1} is proved.

\end{proof}

\begin{remark}\label{V2-O3.1} For $s,b\in\mathbb R$, let us denote by $H^{s,b}$ the anisotropic Sobolev space, defined by
$$H^{s,b}:=\{f(\cdot_x,\cdot_t)\in S'(\mathbb R^2):\|f\|_{H^{s,b}}:=\|\langle\xi\rangle^s\langle\tau\rangle^b\widehat f(\xi,\tau)\|_{L^2_{\xi \tau}}<\infty\}.$$

It can be seen that
\begin{align}
\|f(\cdot_x,\cdot_t)\|_{X^{s,b}}=\|W_{\mathbb R}({\scriptstyle-}\cdot_t)f(\cdot_x,\cdot_t)\|_{H^{s,b}}.\label{V2-3.9}
\end{align}

\end{remark}

\begin{remark}\label{V2-O3.2} In \cite{GTV1997}, Ginibre, Tsutsumi, and Velo proved that if $\eta(\cdot_t)\in C_0^\infty(\mathbb R_t)$ is as in Lemma \ref{V3-L1}, and operator $L$ is defined by
$$(Lf)(t):=\eta\left(\tfrac tT\right)\int_0^tf(t')dt',$$
for fixed $T\in(0,1]$, then, for $-\frac12<b'\leq 0\leq b\leq b'+1$, it is true that
\begin{align}
\|Lf\|_{H^b(\mathbb R_t)}\leq C T^{1-b+b'}\|f\|_{H^{b'}(\mathbb R_t)},\label{V2-3.10}
\end{align}
where $C>0$ is independent of $f$.
\end{remark}

\begin{lemma}\label{V2-L3.3}
\begin{itemize}
\item[(i)] For $s\geq 0$, $-\frac12<b'\leq 0\leq b\leq b'+1$, $0<T\leq 1$, and $\eta\in C_0^\infty(\mathbb R_t)$ as in Lemma \ref{V3-L1}, there exists $C>0$, such that for every $f\in X^{s,b'}$,
\begin{align}
\left\|\eta\left(\tfrac{\cdot_t}T \right)\int_0^{\cdot_t} [W_{\mathbb R}(\cdot_t{\scriptstyle -}t')f(t')](\cdot_x)dt'\right\|_{X^{s,b}}\leq C T^{1-b+b'}\|f\|_{X^{s,b'}}\label{V4.2-4.4}
\end{align}
\item[(ii)] For $s\geq 0$, $b\in(0,\frac12)$, and $\alpha\in(\frac12,1-b)$, let $b^*\in(b,\frac12)$ such that $\alpha<1-b^*$. Then there exists $C>0$ such that for every $f\in Y^{s,-b^*,\alpha}$,
\begin{align}
\left\|\eta(\cdot_t)\int_0^{\cdot_t} [W_{\mathbb R}(\cdot_t{\scriptstyle -}t')f(t')](\cdot_x)dt'\right\|_{X^{s,b,\alpha}}\leq C \|f\|_{X^{s,-b^*}}\leq C \|f\|_{Y^{s,-b^*,\alpha}}.\label{V4.2-4.5}
\end{align}
\end{itemize}

\end{lemma}

\begin{proof}
\begin{itemize}
\item[(i)] Using \eqref{V2-3.9} and \eqref{V2-3.10} we obtain
\begin{align}
\notag &\left\|\eta\left(\tfrac{\cdot_t}T \right)\int_0^{\cdot_t} [W_{\mathbb R}(\cdot_t{\scriptstyle-}t')f(t')](\cdot_x)dt'\right\|_{X^{s,b}}\\
\notag&=\left\| W_{\mathbb R}({\scriptstyle-}\cdot_t)\left\{ \eta\left(\tfrac{\cdot_t}T \right)\int_0^{\cdot_t} [W_{\mathbb R}(\cdot_t{\scriptstyle-}t')f(t')](\cdot_x)dt'\right\}\right\|_{H^{s,b}}\\
\notag&=\left(\int_{\mathbb R_\xi}\langle\xi\rangle^{2s}\int_{\mathbb R_\tau}\langle\tau\rangle^{2b} \left|\left\{  \eta\left(\tfrac{\cdot_t}T \right)\left(\int_0^{\cdot_t} [W_{\mathbb R}(-t')f(t')](\cdot_x)dt'\right)^{\wedge_x}\right\}^{\wedge_t}(\xi,\tau)\right|^2d\tau d\xi\right)^{\frac12}\\
\notag&=\left(\int_{\mathbb R_\xi}\langle\xi\rangle^{2s} \left\| L\left(\left( [W_{\mathbb R}({\scriptstyle-}\cdot_t)f(\cdot_t)](\cdot_x)\right)^{\wedge_x}(\xi)\right)\right\|^2_{H^b(\mathbb R_t)} d\xi\right)^{\frac12}\\
&\leq CT^{1-b+b'}\left(\int_{\mathbb R_\xi}\langle\xi\rangle^{2s}\left\|\left([W_{\mathbb R}({\scriptstyle-}\cdot_t)f(\cdot_t)](\cdot_x) \right)^{\wedge_x}(\xi) \right\|^2_{H^{b'}(\mathbb R_t)}d\xi \right)^{\frac12}.\label{V2-3.12}
\end{align}

Taking into account that
\begin{align*}
\left\|\left([W_{\mathbb R}({\scriptstyle -}\cdot_t)f(\cdot_t)](\cdot_x) \right)^{\wedge_x}(\xi) \right\|^2_{H^{b'}(\mathbb R_t)}=\int_{\mathbb R_\tau}\langle\tau\rangle^{2b'}\left|\left(e^{it\xi^5}[f(\cdot_x,t)]^{\wedge_x}(\xi) \right)^{\wedge_t}(\tau)\right|^2d\tau,
\end{align*}
and since

\begin{align*}
\left(e^{it\xi^5}[f(\cdot_x,t)]^{\wedge_x}(\xi) \right)^{\wedge_t}(\tau)&=C\int_{\mathbb R_t}e^{-it(\tau-\xi^5)}\int_{\mathbb R_x}e^{-i\xi x}f(x,t)dxdt,
\end{align*}
we conclude that
\begin{align*}
\left\|\left([W_{\mathbb R}({\scriptstyle -}\cdot_t)f(\cdot_t)](\cdot_x) \right)^{\wedge_x}(\xi) \right\|^2_{H^{b'}(\mathbb R_t)}=C\int_{\mathbb R_\tau}\langle\tau\rangle^{2b'}|\widehat f(\xi,\tau-\xi^5)|^2d\tau=C\int_{\mathbb R_\tau}\langle\tau+\xi^5\rangle^{2b'}|\widehat f(\xi,\tau)|^2d\tau.
\end{align*}

From \eqref{V2-3.12}, we obtain
\begin{align*}
\left\|\eta\left(\tfrac{\cdot_t}T \right)\int_0^{\cdot_t} [W_{\mathbb R}(\cdot_t{\scriptstyle -}t')f(t')](\cdot_x)dt'\right\|_{X^{s,b}}\leq CT^{1-b+b'}\|f\|_{X^{s,b'}}.
\end{align*}
\item[(ii)] Let us define $G(x,t):=\eta(t)\int_0^{t} [W_{\mathbb R}(t-t')f(t')](x)dt'$. Then
\begin{align*}
\|G\|_{X^{s,b,\alpha}}\leq \|G\|_{X^{s,b}} + \left( \int_{-\infty}^{+\infty} \int_{-1}^1 \langle \tau \rangle^{2\alpha} |\widehat G(\xi,\tau)|^2 d\xi d\tau \right)^{\frac12}.
\end{align*}

Let us observe that there exists $C>0$ such that for every $\xi\in[-1,1]$,
$$\langle \tau \rangle \leq C \langle \tau + \xi^5 \rangle.$$

Hence
\begin{align}
\|G\|_{X^{s,b,\alpha}} \leq \|G\|_{X^{s,b}} + C \|G\|_{X^{0,\alpha}} \leq C \|G\|_{X^{s,\alpha}}.\label{V4.2-4.6}
\end{align}

Since $-\frac12<-b^*<0\leq \alpha < -b^*+1$, by \eqref{V4.2-4.4} with $T=1$, it follows that
\begin{align}
\|G\|_{X^{s,\alpha}} \leq C \|f\|_{X^{s,-b^*}}.\label{V4.2-4.6b}
\end{align}

Estimative \eqref{V4.2-4.5} follows from \eqref{V4.2-4.6}, \eqref{V4.2-4.6b}, and the definition of $\|\cdot\|_{Y^{s,-b^*,\alpha}}$.

\end{itemize}

Lemma \ref{V2-L3.3} is proved.
\end{proof}

\begin{lemma}\label{V2-L3.4} Let $s\geq 0$, and $b_1,b_2$ such that $-\frac12<b_1<b_2<\frac12$, $T\in(0,1]$, and $\eta\in C_0^\infty(\mathbb R_t)$ as in Lemma \ref{V3-L1}. Then there exists $C>0$ such that, for every $F\in X^{s,b_2}$,
\begin{align}
\left\|\eta\left(\tfrac{\cdot_t}T\right)F(\cdot_x,\cdot_t) \right\|_{X^{s,b_1}}\leq CT^{b_2-b_1}\|F\|_{X^{s,b_2}}.\label{V2-3.13}
\end{align}
Besides, if $-\frac12<-b^*<-b<0$, and $\frac12<\alpha<1-b^*$, then, for $T\in(0,\frac12]$,
\begin{align}
\left\|\eta\left(\tfrac{\cdot_t}{2T}\right)F(\cdot_x,\cdot_t) \right\|_{Y^{s,-b^*,\alpha}}\leq C\left\| |\eta\left(\tfrac{\cdot_t}{2T}\right) F(\cdot_x,\cdot_t) \right\|_{X^{s,-b^*}} \leq CT^{b^*-b}\|F\|_{X^{s,-b}}.\label{V4.2-4.7}
\end{align}
\end{lemma}

\begin{proof} Let us prove \eqref{V2-3.13} for $b_1=0$; i.e.,
\begin{align}
\left\|\eta\left(\tfrac{\cdot_t}T\right)F(\cdot_x,\cdot_t) \right\|_{X^{s,0}}\leq CT^{b_2}\|F\|_{X^{s,b_2}}.\label{V2-3.14}
\end{align}

The general case follows from \eqref{V2-3.14}, using an interpolation argument. Estimative \eqref{V2-3.14} is consequence of the following statements:
\begin{enumerate}
\item[(i)] $\left\|\eta\left(\frac{\cdot_t}T \right)F(\cdot_x,\cdot_t) \right\|_{X^{s,0}}\leq CT^{b_2}\left\| \eta\left(\frac{\cdot_t}T\right)F(\cdot_x,\cdot_t)\right\|_{X^{s,b_2}}$,
\item[(ii)] $\left\|\eta\left(\frac{\cdot_t}T \right)F(\cdot_x,\cdot_t) \right\|_{X^{s,b_2}}\leq C\|F\|_{X^{s,b_2}}$;
\end{enumerate}
where $C>0$ is independent of $T$.\\

Let us prove (i):
\begin{align*}
\left\|\eta\left(\tfrac{\cdot_t}T \right)F(\cdot_x,\cdot_t) \right\|^2_{X^{s,0}}&=\int_{\mathbb R}\int_{-\infty}^{+\infty}\langle\xi\rangle^{2s}\left|[\eta\left( \tfrac{\cdot_t}T\right)F]^{\wedge}(\xi,\tau) \right|^2d\tau d\xi\\
&=\int_{\mathbb R}\langle\xi\rangle^{2s}\int_{-\infty}^{+\infty} \left|\eta\left( \tfrac tT\right)[F(\cdot,t)]^{\wedge_x}(\xi) \right|^2dt d\xi\\
&=\int_{\mathbb R_t} \left|\eta\left( \tfrac tT\right)\right|^2\int_{\mathbb R_\xi}\left| e^{it\xi^5} \langle\xi\rangle^{s} [F(\cdot,t)]^{\wedge_x}(\xi) \right|^2d\xi dt\\
&=C\int_{-T}^T \left|\eta\left( \tfrac tT\right)\right|^2\int_{\mathbb R_x}\left|[W_{\mathbb R}(-t)J^s(F(\cdot,t))] (x)\right|^2dx dt
\end{align*}
where
$$[J^s(F(\cdot,t))]^{\wedge_x}(\xi):=\langle\xi\rangle^s[F(\cdot,t)]^{\wedge_x}(\xi).$$

Hence
\begin{align*}
\left\|\eta\left(\tfrac{\cdot_t}T \right)F(\cdot_x,\cdot_t) \right\|^2_{X^{s,0}}=C\int_{-T}^T\int_{\mathbb R_x}\left| [W_{\mathbb R}(-t)J^s(\eta\left(\tfrac tT\right)F(\cdot,t))](x)\right|^2dxdt.
\end{align*}

Since $0<2b_2<1$, $\frac1{2b_2}>1$. Defining $p:=\frac1{2b_2}$, and $q$ as the conjugate exponent of $p$; i.e., $q:=\frac1{1-2b_2}$, by Hölder's inequality in the $t$ variable, we obtain
\begin{align*}
\left\|\eta\left(\tfrac{\cdot_t}T \right)F(\cdot_x,\cdot_t) \right\|^2_{X^{s,0}}&\leq C T^{2b_2}\left\{\int_{-\infty}^{+\infty}\left(\int_{\mathbb R}\left|[W_{\mathbb R}(-t)J^s(\eta\left(\tfrac tT \right)F(\cdot,t))](x) \right|^2dx \right)^{\frac1{1-2b_2}} dt\right\}^{1-2b_2}\\
&\leq C T^{2b_2}\int_{\mathbb R}\left(\int_{-\infty}^{+\infty}\left|[W_{\mathbb R}(-t)J^s(\eta\left(\tfrac tT \right)F(\cdot,t))](x) \right|^{\frac2{1-2b_2}}  dt\right)^{1-2b_2}dx\\
&=CT^{2b_2}\int_{\mathbb R} \left\|[W_{\mathbb R}({\scriptstyle -}\cdot_t)J^s(\eta\left(\tfrac {\cdot_t}T \right)F(\cdot,\cdot_t))](x) \right\|^2_{L^{\frac2{1-2b_2}}(\mathbb R_t)}dt.
\end{align*}

Taking into account that $H^{b_2}(\mathbb R)\hookrightarrow L^{\frac2{1-2b_2}}(\mathbb R)$ for $0\leq b_2<\frac12$, it follows that
\begin{align*}
\left\|\eta\left(\tfrac{\cdot_t}T \right)F(\cdot_x,\cdot_t) \right\|^2_{X^{s,0}}&\leq C T^{2b_2}\int_{\mathbb R}\int_{-\infty}^{+\infty}\langle\tau\rangle^{2b_2} \left|\left([W_{\mathbb R}({\scriptstyle-}\cdot_t)J^s(\eta\left(\tfrac {\cdot_t}T \right)F(\cdot,\cdot_t))](x)\right)^{\wedge_t} (\tau)\right|^2d\tau dx\\
&=CT^{2b_2}\int_{-\infty}^{+\infty}\langle\tau\rangle^{2b_2}\int_{\mathbb R_{\xi}} \left|\left\{\left([W_{\mathbb R}({\scriptstyle-}\cdot_t)J^s(\eta\left(\tfrac {\cdot_t}T \right)F(\cdot,\cdot_t))](\cdot_x)\right)^{\wedge_t}(\tau)\right\}^{\wedge_x}(\xi) \right|^2d\xi d\tau\\
&=CT^{2b_2}\int_{-\infty}^{+\infty}\langle\tau\rangle^{2b_2}\int_{\mathbb R_{\xi}} \left|\left(e^{it\xi^5}\langle\xi\rangle^s [\eta\left(\tfrac {\cdot_t}T \right)F(\cdot_x,\cdot_t)]^{\wedge_x}(\xi)\right)^{\wedge_t}(\tau) \right|^2d\xi d\tau\\
&=CT^{2b_2}\int_{-\infty}^{+\infty}\langle\tau\rangle^{2b_2}\int_{\mathbb R_{\xi}}\langle\xi\rangle^{2s} \left| [\eta\left(\tfrac {\cdot_t}T \right)F(\cdot_x,\cdot_t)]^{\wedge_{xt}}(\xi,\tau-\xi^5) \right|^2d\xi d\tau\\
&=CT^{2b_2}\int_{\mathbb R_\xi}\langle\xi\rangle^{2s}\int_{-\infty}^{+\infty}\langle\tau+\xi^5\rangle^{2b_2} \left| [\eta\left(\tfrac {\cdot_t}T \right)F(\cdot_x,\cdot_t)]^{\wedge_{xt}}(\xi,\tau) \right|^2d\tau d\xi\\
&=CT^{2b_2}\left\|\eta\left(\tfrac{\cdot_t}T \right)F(\cdot_x,\cdot_t) \right\|^2_{X^{s,b_2}},
\end{align*}
and estimative (i) is proved.\\

Now we prove (ii):
\begin{align*}
\left\|\eta\left(\tfrac{\cdot_t}T \right)F(\cdot_x,\cdot_t) \right\|^2_{X^{s,b_2}}=\int_{\mathbb R}\langle\xi\rangle^{2s}\int_{-\infty}^{+\infty}\langle\tau+\xi^5\rangle^{2b_2}\left| [\eta\left(\tfrac{\cdot_t}T \right)F(\cdot_x,\cdot_t)]^{\wedge}(\xi,\tau)\right|^2 d\tau d\xi.
\end{align*}

Let us observe that
\begin{align*}
[\eta\left(\tfrac{\cdot_t}T\right)F(\cdot_x,\cdot_t)]^\wedge(\xi,\tau)&=C\left[\eta\left(\tfrac {\cdot_t}T\right) (F(\cdot_x,t))^{\wedge_x}(\xi) \right]^{\wedge_t}(\tau)=C\left\{\left[\eta\left(\tfrac {\cdot_t}T\right)\right]^{\wedge_t}\ast_\tau \widehat F(\xi,\cdot_\tau)\right\}(\tau).
\end{align*}

But,
$$\left[\eta\left(\tfrac {\cdot_t}T\right)\right]^{\wedge_t}(\tau)=T\widehat \eta(T\tau).$$

Hence
\begin{align}
\left\|\eta\left(\tfrac{\cdot_t}T \right)F(\cdot_x,\cdot_t) \right\|^2_{X^{s,b_2}}=C\int_{\mathbb R_\xi}\langle\xi\rangle^{2s}\int_{-\infty}^{+\infty}\langle\tau+\xi^5\rangle^{2b_2}\left|\left[(T\widehat\eta(T\cdot)\ast_\tau \widehat F(\xi,\cdot)\right](\tau)\right|^2 d\tau d\xi.\label{V2-3.15}
\end{align}

Let us estimate the previous integral in the $\tau$ variable:
\begin{align}
&\notag \int_{-\infty}^{+\infty}\langle\tau+\xi^5\rangle^{2b_2}\left|\left[(T\widehat\eta(T\cdot)\ast_\tau \widehat F(\xi,\cdot)\right](\tau)\right|^2 d\tau\\
 &\leq C\left\| T\widehat \eta(T\cdot)\ast_\tau\widehat F(\xi,\cdot)\right\|^2_{L^2_\tau}+C\int_{-\infty}^{+\infty}|\tau|^{2b_2}\left|\left[(T\widehat \eta(T\cdot))\ast_\tau \widehat F(\xi,\cdot)\right](\tau-\xi^5)\right|^2d\tau\equiv I(\xi)+II(\xi).\label{V2-3.16}
\end{align}

On the one hand
\begin{align}
I(\xi)\leq C\|T\widehat \eta(T\cdot_\tau)\|^2_{L^1_\tau}\|\widehat F(\xi,\cdot_\tau)\|^2_{L^2_\tau}\leq C\|\widehat F(\xi,\cdot_\tau)\|^2_{L^2_\tau},\label{V2-3.17}
\end{align}

with $C$ independent of $T$.\\

On the other hand, to estimate $II(\xi)$, we use Leibniz formula for fractional derivatives (see \cite{KPV1993b}, Theorem A.12):
\begin{align}
\|D^\alpha(fg)-fD^\alpha g\|_{L^p(\mathbb R)}\leq C\|g\|_{L^\infty(\mathbb R)}\|D^\alpha f\|_{L^p(\mathbb R)}\quad \alpha\in(0,1),\quad 1<p<\infty.\label{V2-3.18}
\end{align}

Let us observe that 
\begin{align}
II(\xi)=C\left\|D^{b_2}_t\left(\mathcal F^{-1}_t([T\widehat\eta(T\cdot)\ast_\tau\widehat F(\xi,\cdot)](\tau-\xi^5))\right)\right\|^2_{L^2_t}.\label{V2-3.19}
\end{align}

But
\begin{align*}
\mathcal F^{-1}_t[(T\widehat \eta(T\cdot)\ast_\tau \widehat F(\xi,\cdot))(\tau-\xi^5)](t)&=C\int_{-\infty}^{+\infty}e^{it\tau}e^{it\xi^5}[T\widehat \eta(T\cdot)\ast_\tau \widehat F(\xi,\cdot)](\tau)d\tau\\
&=Ce^{it\xi^5}\mathcal F^{-1}_t[T\widehat \eta(T\cdot)](t)[F(\cdot_x,t)]^{\wedge_x}(\xi)\\
&=Ce^{it\xi^5}\eta\left(\tfrac tT\right) [F(\cdot_x,t)]^{\wedge_x}(\xi).
\end{align*}

Then, from \eqref{V2-3.19}, using \eqref{V2-3.18}, we have that
\begin{align}
\notag II(\xi)=&C\left\|D_t^{b_2}\left(e^{i\cdot_t\xi^5}[F(\cdot_x,\cdot_t)]^{\wedge_x}(\xi)\eta\left(\tfrac {\cdot_t}T\right)\right)\right\|^2_{L^2_t}\\
\notag \leq &C\left\|D_t^{b_2}\left(e^{i\cdot_t\xi^5}[F(\cdot_x,\cdot_t)]^{\wedge_x}(\xi)\eta\left(\tfrac {\cdot_t} T\right)\right)-e^{i\cdot_t\xi^5}[F(\cdot_x,\cdot_t)]^{\wedge_x}(\xi)D_t^{b_2}\left(\eta\left(\tfrac {\cdot_t}T\right)\right)\right\|^2_{L^2_t}\\
\notag&+C \left\|e^{i\cdot_t\xi^5}[F(\cdot_x,\cdot_t)]^{\wedge_x}(\xi)D_t^{b_2}\left(\eta\left(\tfrac {\cdot_t}T \right)\right) \right\|^2_{L^2_t}\\
\notag \leq & C\left\|\eta\left(\tfrac {\cdot_t}T \right) \right\|^2_{L^\infty_t}\left\|D_t^{b_2}\left(e^{i\cdot_t\xi^5}[F(\cdot_x,\cdot_t)]^{\wedge_x}(\xi) \right) \right\|^2_{L^2_t}+C \left\|e^{i\cdot_t\xi^5}[F(\cdot_x,\cdot_t)]^{\wedge_x}(\xi)D_t^{b_2}\left(\eta\left(\tfrac {\cdot_t}T \right)\right) \right\|^2_{L^2_t}\\
\leq & C\left\|D_t^{b_2}\left(e^{i\cdot_t\xi^5}[F(\cdot_x,\cdot_t)]^{\wedge_x}(\xi) \right) \right\|^2_{L^2_t}+C \left\|e^{i\cdot_t\xi^5}[F(\cdot_x,\cdot_t)]^{\wedge_x}(\xi)D_t^{b_2}\left(\eta\left(\tfrac {\cdot_t}T \right)\right) \right\|^2_{L^2_t}.\label{V2-3.20}
\end{align}

From \eqref{V2-3.17} it is clear that
\begin{align}
\int_{\mathbb R_\xi}\langle\xi\rangle^{2s}I(\xi)d\xi\leq C\int_{\mathbb R_\xi}\langle\xi\rangle^{2s}\|\widehat F(\xi,\cdot_\tau)\|^2_{L^2_\tau}d\xi=C\|F\|^2_{X^{s,0}}.\label{V2-3.21}
\end{align}

Besides, from \eqref{V2-3.20}, using Plancherel's identity, we have that
\begin{align*}
\int_{\mathbb R_\xi}\langle\xi\rangle^{2s}II(\xi)d\xi\leq &C\int_{\mathbb R_\xi}\langle\xi\rangle^{2s}\int_{-\infty}^{+\infty}|\tau|^{2b_2}\left|\left(e^{it\xi^5}[F(\cdot_x,t)]^{\wedge_x}(\xi) \right)^{\wedge_t}(\tau)\right|^2 d\tau d\xi\\
&+C\int_{\mathbb R_\xi}\langle\xi\rangle^{2s}\int_{-\infty}^{+\infty}\left| e^{it\xi^5}[F(\cdot_x,t)]^{\wedge_x}(\xi)D_t^{b_2}\left(\eta\left( \tfrac tT\right)\right)\right|^2 dt d\xi\\
\leq &C\int_{\mathbb R_\xi}\langle\xi\rangle^{2s}\int_{-\infty}^{+\infty}|\tau+\xi^5|^{2b_2}|\widehat F(\xi,\tau)|^2d\tau d\xi\\
&+C\int_{\mathbb R_\xi}\langle\xi\rangle^{2s}\int_{-\infty}^{+\infty}\left| e^{it\xi^5}[F(\cdot_x,t)]^{\wedge_x}(\xi)D_t^{b_2}\left(\eta\left(\tfrac tT \right) \right)\right|^2dt d\xi.
\end{align*}

Using Hölder's inequality in the integral in the $t$ variable with conjugate exponents $p$ and $p'$, we have that
\begin{align*}
\int_{\mathbb R_\xi}\langle\xi\rangle^{2s}II(\xi)d\xi\leq C\|F\|^2_{X^{s,b_2}}+C\int_{\mathbb R_\xi}\langle\xi\rangle^{2s}\left\{\int_{-\infty}^{+\infty}\left| e^{it\xi^5}[F(\cdot_x,t)]^{\wedge_x}(\xi)\right|^{2p}dt \right\}^{\frac1p}\left\{ \int_{-\infty}^{+\infty}|D_t^{b_2}\left(\eta\left( \tfrac tT\right) \right)|^{2p'}dt\right\}^{\frac1{p'}}d\xi.
\end{align*}

Let us choose $p$ such that $\frac12-\frac1{2p}=b_2$; i.e., $\frac1{2p}=\frac{1-2b_2}2$, which is equivalent to $2p=\frac2{1-2b_2}$, which implies $H^{b_2}(\mathbb R_t)\hookrightarrow L^{2p}(\mathbb R_t)$. Then

\begin{align*}
\int_{\mathbb R_\xi}\langle\xi\rangle^{2s}II(\xi)d\xi\leq & C\|F\|^2_{X^{s,b_2}}+C\int_{\mathbb R_\xi}\langle\xi\rangle^{2s}\left\| e^{it\xi^5}[F(\cdot_x,t)]^{\wedge_x}(\xi)\right\|^2_{H^{b_2}(\mathbb R_t)}d\xi\left\{ \int_{-\infty}^{+\infty}|D_t^{b_2}\left(\eta\left(\tfrac tT \right) \right)|^{2p'}dt\right\}^{\frac1{p'}}\\
\leq & C\|F\|^2_{X^{s,b_2}}+C\|F\|^2_{X^{s,b_2}}\left\{ \int_{-\infty}^{+\infty}|D_t^{b_2}\left(\eta\left(\tfrac tT \right) \right)|^{2p'}dt\right\}^{\frac1{p'}}.
\end{align*}

Taking into account that the inverse Fourier transform is a bounded operator from $L^{\frac{2p'}{2p'-1}}$ to $L^{2p'}$, it follows that
\begin{align}
\notag\int_{\mathbb R_\xi}\langle\xi\rangle^{2s}II(\xi)d\xi\leq & C\|F\|^2_{X^{s,b_2}}+C\|F\|^2_{X^{s,b_2}}\left\{ \int_{-\infty}^{+\infty}\left| |\tau|^{b_2}T\widehat \eta(T\tau) \right|^{\frac{2p'}{2p'-1}}d\tau\right\}^{\frac{2p'-1}{p'}}\\
\leq & C\|F\|^2_{X^{s,b_2}}.\label{V2-3.22}
\end{align}

From \eqref{V2-3.15}, \eqref{V2-3.21}, and \eqref{V2-3.22} we conclude that
$$\left\| \eta\left(\tfrac {\cdot_t}T \right)F(\cdot_x,\cdot_t)\right\|_{X^{s,b_2}}\leq C\|F\|_{X^{s,b_2}}\quad T\in(0,1],$$

with $C$ independent of $T$.\\

In consequence of (i), and (ii), we conclude that
$$\left\| \eta\left(\tfrac{\cdot_t}T\right) F(\cdot_x,\cdot_t)\right\|_{X^{s,0}}\leq CT^{b_2}\|F\|_{X^{s,b_2}}.$$

Now, we prove estimative \eqref{V4.2-4.7}. From the definition of $\| \eta\left(\tfrac {\cdot_t}{2T} \right)F(\cdot_x,\cdot_t) \|_{Y^{s,-b^*,\alpha}}$, it is enough to prove that

\begin{align}
\left( \int_{-\infty}^{+\infty} \int_{-1}^1 \frac{|[\eta(\tfrac{\cdot_t}{2T}) F]^\wedge(\xi,\tau)|^2}{\langle \tau \rangle^{2(1-\alpha)}} d\xi d\tau \right)^{\frac12} \leq C \|\eta(\tfrac{\cdot_t}{2T} F)\|_{X^{s,-b^*}},\label{V4.2-4.8}
\end{align}

and
\begin{align}
\left(\int_{-\infty}^{+\infty} \langle \xi \rangle^{2s} \left( \int_{-\infty}^{+\infty} \frac{|[\eta(\tfrac{\cdot_t}{2T})F]^\wedge(\xi,\tau))|}{\langle \tau + \xi^5\rangle} d\tau \right)^2 d\xi \right)^{\frac12} \leq C \|\eta(\tfrac{\cdot_t}{2T})F\|_{X^{s,-b^*}}.\label{V4.2-4.9}
\end{align}

Let us prove \eqref{V4.2-4.8}. If $|\xi|\leq 1$, $\langle \tau \rangle\geq \frac 12 \langle \tau + \xi^5\rangle$.\\

Therefore
\begin{align*}
&\left( \int_{-\infty}^{+\infty} \int_{-1}^1 \frac{|\widehat{\eta(\tfrac{\cdot_t}{2T}) F}(\xi,\tau)|^2}{\langle \tau \rangle^{2(1-\alpha)}} d\xi d\tau \right)^{\frac12}\leq C \left(\int_{-\infty}^{+\infty} \int_{-\infty}^{+\infty} \frac{|\widehat{\eta(\tfrac{\cdot_t}{2T}) F}(\xi,\tau)|^2}{\langle \tau + \xi^5\rangle^{2(1-\alpha)}}d\xi d\tau \right)^{\frac12}\\
&= \left( \int_{-\infty}^{+\infty} \int_{-\infty}^{+\infty} \langle \tau + \xi^5 \rangle^{2(\alpha-1)} |\widehat{\eta(\tfrac{\cdot_t}{2T}) F}(\xi,\tau)|^2 d\xi d\tau\right)^{\frac12} \leq C \left( \int_{-\infty}^{+\infty} \int_{-\infty}^{+\infty} \langle \tau + \xi^5 \rangle^{-2b^*} |\widehat{\eta(\tfrac{\cdot_t}{2T}) F}(\xi,\tau)|^2 d\xi d\tau\right)^{\frac12}\\
&\leq C \|\eta(\tfrac{\cdot_t}{2T} F)\|_{X^{s,-b^*}}.
\end{align*}

Now, we prove \eqref{V4.2-4.9}.
\begin{align*}
&\left(\int_{-\infty}^{+\infty} \langle \xi \rangle^{2s} \left( \int_{-\infty}^{+\infty} \frac{ |\widehat{\eta(\tfrac{\cdot_t}{2T})F}(\xi,\tau)|}{\langle \tau + \xi^5 \rangle} d\tau \right)^2 d\xi \right)^{\frac12}\\
  &\leq \left(\int_{-\infty}^{+\infty} \langle \xi \rangle^{2s} \left(\int_{-\infty}^{+\infty} \frac{d\tau}{\langle \tau + \xi^5 \rangle^{2(1-b^*)}} \right) \left( \int_{-\infty}^{+\infty}  \langle \tau + \xi^5 \rangle^{-2b^*}  |\widehat{\eta(\tfrac{\cdot_t}{2T})F}(\xi,\tau)|^2 d\tau \right) d\xi \right)^{\frac12}\\
  &\leq C \left( \int_{-\infty}^{+\infty} \int_{-\infty}^{+\infty} \langle \xi\rangle^{2s} \langle \tau + \xi^5 \rangle^{-2b^*} | \widehat{\eta(\tfrac{\cdot t}{2T})F}(\xi,\tau)|^2 d\xi d\tau \right)^{\frac12} = C \|\eta(\tfrac{\cdot_t}{2T})F\|_{X^{s,-b^*}},
\end{align*}
since $2(1-b^*)>1$. Lemma \ref{V2-L3.4} is proved

\end{proof}

From now on we will use the following calculus inequalities:
\begin{lemma}\label{CI}
\begin{enumerate}
\item[(i)] For $\beta\geq\gamma\geq0$, and $\quad \beta+\gamma>1$ it follows that
\begin{align}
\int_{\mathbb R}\frac1{\langle x-a_1\rangle^\beta\langle x-a_2\rangle^\gamma}dx\leq C\frac{\phi_\beta(a_1-a_2)}{\langle a_1-a_2\rangle^\gamma},\label{V2-3.39}
\end{align}

where
$$\phi_\beta(a)\sim \left\{ \begin{array}{lr}
1&\text{for }\beta>1,\\
\log(1+\langle a\rangle)&\text{for }\beta=1,\\
\langle a\rangle^{1-\beta}&\text{for }\beta<1.
\end{array}\right.
$$
(For a proof of \eqref{V2-3.39}, see \cite{ET2013}).

\item[(ii)] For $\rho\in(\tfrac12,1)$,

\begin{align}
\int_{\mathbb R}\frac1{\langle x\rangle^\rho\sqrt{|x-a|}}\leq C\frac1{\langle a\rangle^{\rho-\frac12}}.\label{V2-3.71}
\end{align}

(For a proof of \eqref{V2-3.17} see \cite{ET2016}).

\end{enumerate}
\end{lemma}

\begin{lemma}\label{V3-L4.6} Let $\eta\in C_0^\infty(\mathbb R_t)$ as in Lemma \ref{V3-L1}.
\begin{enumerate}
\item[(i)] For $0\leq s\leq \frac12$, $0<b<\frac12$, and $j=0,1,2$, there exists $C>0$ such that
\begin{align}
\left\| \eta(\cdot_t) \partial_x^j \int_0^{\cdot_t}[W_{\mathbb R}(\cdot_t{\scriptstyle-}t')F(t')](\cdot_x)dt'\right\|_{C(\mathbb R_x;H^{\frac{s+2-j}5}(\mathbb R_t))}\leq C\|F\|_{X^{s,-b}}.\label{V3-L4.6.1}
\end{align}
\item[(ii)] For $\frac12\leq s\leq \frac{11}2$, $0<b<\frac12$, and $j=0,1,2$, there exists $C>0$ such that
\begin{align}
\left\| \eta(\cdot_t) \partial_x^j \int_0^{\cdot_t}[W_{\mathbb R}(\cdot_t{\scriptstyle-}t')F(t')](\cdot_x)dt'\right\|_{C(\mathbb R_x;H^{\frac{s+2-j}5}(\mathbb R_t))}\leq C\|F\|_{X^{\frac12,\frac{2s-1-10b}{10}}}+C\|F\|_{X^{s,-b}}\label{V3-L4.6.2}
\end{align}
\end{enumerate}
\end{lemma}

\begin{proof}
\begin{enumerate}
\item[(i)] This proof is based on an argument by Colliander and Kenig in \cite{CK2002}. Let us observe that
\begin{align*}
\partial_x^j \int_0^{t}[W_{\mathbb R}(t-t')F(t')](x)dt'&=C\int_0^t\int_{-\infty}^{+\infty}e^{ix\xi}e^{-i\xi^5(t-t')} (i\xi)^j[F(t')]^{\wedge_x}(\xi)d\xi dt'\\
&=C\int_{-\infty}^{+\infty}\int_{-\infty}^{+\infty}e^{i\xi x} \xi^j \frac{e^{it\tau}-e^{-it\xi^5}}{i(\tau+\xi^5)}\widehat F(\xi,\tau)d\xi d\tau.
\end{align*}

Let us consider $\psi\in C_0^\infty(\mathbb R)$ such that $\psi\equiv 1$ in $[-1,1]$, $0\leq\psi\leq 1$, and $\supp\psi\subset[-2,2]$. We will denote the function $1-\psi$ by $\psi^c$. Then
\begin{align}
\notag \eta(t) \partial_x^j \int_0^t[W_{\mathbb R}(t-t')F(t')](x)dt'=& C \eta(t)\int_{\mathbb R^2}e^{i\xi x} \xi^j \frac{e^{it\tau}-e^{-it\xi^5}}{i(\tau+\xi^5)}\psi(\tau+\xi^5)\widehat F(\xi,\tau)d\xi d\tau\\
\notag&+ C \eta(t)\int_{\mathbb R^2} \xi^j \frac{e^{it\tau}}{i(\tau+\xi^5)} e^{i\xi x}\psi^c(\tau+\xi^5)\widehat F(\xi,\tau)d\xi d\tau\\
\notag&-C\eta(t)\int_{\mathbb R^2} \xi^j \frac{e^{-it\xi^5}}{i(\tau+\xi^5)}e^{i\xi x}\psi^c(\tau+\xi^5)\widehat F(\xi,\tau)d\xi d\tau\\
\equiv&I(x,t)+II(x,t)+III(x,t).\label{V2-3.26}
\end{align}

Let us estimate $\|I(x,\cdot_t)\|_{H^{\frac{s+2-j}5}(\mathbb R_t)}$. We observe that
$$\frac{e^{it\tau}-e^{-it\xi^5}}{i(\tau+\xi^5)}=ie^{it\tau}\sum_{k=1}^\infty\frac{(-it)^k(\tau+\xi^5)^{k-1}}{k!}.$$

Hence
$$I(x,t)= C \eta(t)\int_{\mathbb R_2}e^{i\xi x}i e^{it\tau} \xi^j \sum_{k=1}^\infty \frac{(-it)^k(\tau+\xi^5)^{k-1}}{k!}\psi(\tau+\xi^5)\widehat F(\xi,\tau)d\xi d\tau,$$
and
$$\|I(x,\cdot_t)\|_{H^{\frac{s+2-j}5}(\mathbb R_t)}\leq C \sum_{k=1}^\infty \frac1{k!}\left\|\eta(\cdot_t)t^k\int_{\mathbb R^2} \xi^j e^{i\cdot_t\tau}(\tau+\xi^5)^{k-1}e^{i\xi x}\psi(\tau+\xi^5)\widehat F(\xi,\tau)d\xi d\tau \right\|_{H^{\frac{s+2-j}5}(\mathbb R_t)}.$$

Using Kato-Ponce commutator inequality (see \cite{KP1988}, and \cite{IMS1995}), we obtain, for $0\leq l\leq 2$,
\begin{align}
\|fg\|_{H^l(\mathbb R)}\leq C\|f\|_{H^2(\mathbb R)}\|g\|_{H^l(\mathbb R)}.\label{V2-3.27}
\end{align}

This way, it follows that
\begin{align*}
\|I(x,\cdot_t)\|_{H^{\frac{s+2-j}5}(\mathbb R_t)}\leq C\sum_{k=1}^\infty \frac{\|\eta(\cdot_t)(\cdot_t)^k\|_{H^2(\mathbb R_t)}}{k!}\left\| \int_{\mathbb R^2} \xi^j e^{i\cdot_t\tau}(\tau+\xi^5)^{k-1}e^{i \xi x}\psi(\tau+\xi^5)\widehat F(\xi,\tau)d\xi d\tau\right\|_{H^{\frac{s+2-j}5}(\mathbb R_t)}.
\end{align*}
Since
$$\|\eta(\cdot_t)(\cdot_t)^k\|_{H^2(\mathbb R_t)}\leq C k(k-1)\quad \text{for }k\geq 2,$$

then
\begin{align}
\notag &\|I(x,\cdot_t)\|_{H^{\frac{s+2-j}5}(\mathbb R_t)}\leq C\|\eta(\cdot_t)\cdot_t\|_{H^2(\mathbb R_t)}\left\|\int_{\mathbb R^2} \xi^j e^{i\cdot_t\tau}e^{i \xi x}\psi(\tau+\xi^5) \widehat F(\xi,\tau)d\xi d\tau\right\|_{H^{\frac{s+2-j}5}(\mathbb R_t)}\\
\notag&+C\sum_{k=2}^\infty \frac{k(k-1)}{k!}\left\| \int_{\mathbb R^2}\xi^j e^{i\cdot_t\tau} (\tau+\xi^5)^{k-1}e^{i\xi x}\psi(\tau+\xi^5)\widehat F(\xi,\tau)d\xi d\tau\right\|_{H^{\frac{s+2-j}5}(\mathbb R_t)}\\
\notag\leq & C \left\|\langle\cdot_\tau\rangle^{\frac{s+2-j}5}\int_{\mathbb R} \xi^j e^{i\xi x} \psi(\cdot_\tau+\xi^5)\widehat F(\xi,\cdot_\tau)d\xi \right\|_{L^2_\tau}\\
\notag&+C\sum_{k=2}^\infty\frac1{(k-2)!}\left\|\langle\cdot_\tau\rangle^{\frac{s+2-j}5} \int_{\{\xi:|\tau+\xi^5|\leq 2\}} \xi^j (\cdot_\tau+\xi^5)^{k-1} e^{i\xi x}\psi(\cdot_\tau+\xi^5)\widehat F(\xi,\cdot_\tau)d\xi\right\|_{L^2_\tau}\\
\notag\leq & C \left\|\langle\cdot_\tau\rangle^{\frac{s+2-j}5}\int_{\mathbb R} |\xi|^j  \psi(\cdot_\tau+\xi^5)|\widehat F(\xi,\cdot_\tau)|d\xi \right\|_{L^2_\tau}\\
\notag&+C\sum_{k=2}^\infty\frac{2^{k-1}}{(k-2)!}\left\|\langle\cdot_\tau\rangle^{\frac{s+2-j}5} \int_{\{\xi:|\tau+\xi^5|<2\}}|\xi^j|\psi(\cdot_\tau+\xi^5)|\widehat F(\xi,\cdot_\tau)|d\xi\right\|_{L^2_\tau}\\
\notag \leq & C \left\| \langle \cdot_\tau \rangle^{\frac{s+2-j}{5}} \int_{\{\xi:|\tau+\xi^5|<2\}} |\xi|^j \psi(\cdot_\tau+\xi^5) |\widehat F(\xi,\cdot_\tau)| d\xi\right\|_{L^2_\tau}\\
\notag\leq & C\left[\int_{\mathbb R_\tau}\langle\tau\rangle^{\frac{2s+4-2j}5}\left(\int_{\{\xi:|\tau+\xi^5|<2\}} |\xi|^j \psi(\tau+\xi^5)|\widehat F(\xi,\tau)|d\xi \right)^2 d\tau\right]^{\frac12}\\
\notag\leq & C\left[\int_{\mathbb R_\tau}\langle\tau\rangle^{\frac{2s+4-2j}5}\left(\int_{\{\xi:|\tau+\xi^5|<2\}}\langle\xi\rangle^{-2s+2j}d\xi\right)\left(\int_{\{\xi:|\tau+\xi^5|<2\}}\langle\xi\rangle^{2s}|\widehat F(\xi,\tau)|^2d\xi \right) d\tau\right]^{\frac12}\\
\leq & C\sup_{\tau}\left[\langle\tau\rangle^{\frac{2s+4-2j}5}\left(\int_{\{\xi:|\tau+\xi^5|<2\}}\langle\xi\rangle^{-2s+2j}d\xi\right)\right]^{\frac12}\left[\int_{\mathbb R_\tau}\int_{\{\xi:|\tau+\xi^5|<2\}}\langle\xi\rangle^{2s}|\widehat F(\xi,\tau)|^2d\xi d\tau \right]^{\frac12}.\label{V2-3.28}
\end{align}

Let us observe that, if $|\tau+\xi^5|<2$, and $0\leq b$, then $1\geq \langle\tau+\xi^5\rangle^{-2b}\geq\frac1{3^{2b}}$, therefore
\begin{align}
\notag \left[\int_{\mathbb R_\tau}\int_{\{\xi:|\tau+\xi^5|<2\}}\langle\xi\rangle^{2s}|\widehat F(\xi,\tau)|^2d\xi d\tau \right]^{\frac12}&\leq C\left[ \int_{\mathbb R_\tau}\int_{\mathbb R_\xi}\langle\tau+\xi^5\rangle^{-2b}\langle\xi\rangle^{2s}|\widehat F(\xi,\tau)|^2 d\xi d\tau\right]^{\frac12}\\
&\leq C \|F\|_{X^{s,-b}}.\label{V2-3.29}
\end{align}

Now, let us see that
$$\sup_\tau\left[\langle\tau\rangle^{\frac{2s+4-2j}5}\int_{\{\xi:|\tau+\xi^5|<2\}}\langle\xi\rangle^{-2s+2j}d\xi \right]^{\frac12}<C<\infty.$$

In fact,
\begin{align*}
\sup_\tau\left[\langle\tau\rangle^{\frac{2s+4-2j}5}\int_{\{\xi:|\tau+\xi^5|<2\}}\langle\xi\rangle^{-2s+2j}d\xi \right]^{\frac12}=\sup_{\tau}\left[\langle\tau\rangle^{\frac{2s+4-2j}5}\int_{(-2-\tau)^{\frac15}}^{(2-\tau)^{\frac15}} \langle\xi\rangle^{-2s+2j}d\xi \right]^{\frac12}.
\end{align*}

For $\tau<-3$,
$$\langle\tau\rangle^{\frac{2s+4-2j}5}\int_{(-2-\tau)^{\frac15}}^{(2-\tau)^{\frac15}} \langle\xi\rangle^{-2s+2j}d\xi \leq\langle\tau\rangle^{\frac{2s+4-2j}5}\frac{(1+|2-\tau|^{\frac15})^{2j}}{(1+|2+\tau|^{\frac15})^{2s}} [(2-\tau)^{\frac15}-(-2-\tau)^{\frac15}].$$

But
\begin{align*}
(2-\tau)^{\frac15}&-(-2-\tau)^{\frac15}\leq \frac4{(2-\tau)^{\frac45}}\leq C\frac1{|\tau|^{\frac45}},
\end{align*}
\begin{align*}
\frac1{(1+|2+\tau|^{\frac15})^{2s}}&\leq \frac C{|\tau|^{\frac{2s}5}},\\
(1+|2-\tau|^{\frac15})^{2j}&\leq C |\tau|^{\frac{2j}5}.
\end{align*}

Hence, for $\tau<-3$,
\begin{align}
\langle\tau\rangle^{\frac{2s+4-2j}5}\int_{(-2-\tau)^{\frac15}}^{(2-\tau)^{\frac15}} \langle\xi\rangle^{-2s+2j}d\xi &\leq C|\tau|^{\frac{2s+4-2j}5}\frac1{|\tau|^{\frac{2s}5}} |\tau|^{\frac{2j}5} \frac1{|\tau|^{\frac45}}= C.\label{V2-3.30}
\end{align}

For $\tau>4$,
\begin{align}
\notag \langle\tau\rangle^{\frac{2s+4-2j}5}\int_{(-2-\tau)^{\frac15}}^{(2-\tau)^{\frac15}} \langle\xi\rangle^{-2s+2j}d\xi&=\langle\tau\rangle^{\frac{2s+4-2j}5}\int_{(\tau-2)^{\frac15}}^{(\tau+2)^\frac15}\langle\xi\rangle^{-2s+2j}d\xi\\
\notag&\leq\langle\tau\rangle^{\frac{2s+4-2j}5}\frac{(1+|\tau+2|^{\frac15})^{2j}}{(1+|\tau-2|^{\frac15})^{2s}}[(\tau+2)^{\frac15}-(\tau-2)^{\frac15}]\\
&\leq C|\tau|^{\frac{2s+4-2j}5}\frac {|\tau|^{\frac{2j}5}}{|\tau|^{\frac{2s}5}}\frac1{|\tau|^{\frac45}}= C.\label{V2-3.31}
\end{align}

For $-3\leq \tau\leq4$,
\begin{align}
\langle\tau\rangle^{\frac{2s+4-2j}5}\int_{(-2-\tau)^{\frac15}}^{(2-\tau)^{\frac15}} \langle\xi\rangle^{-2s+2j}d\xi\leq C \langle\tau\rangle^{\frac{(2s+4-2j)}5}\langle\tau\rangle^{\frac{(-2s+2j)}5}\int_{(-2-\tau)^{\frac15}}^{(2-\tau)^{\frac15}} d\xi=C\langle\tau\rangle^{\frac45} [(2-\tau)^{\frac15}-(-2-\tau)^{\frac15}]\leq C.\label{V2-3.32}
\end{align}

Therefore, from \eqref{V2-3.30} to \eqref{V2-3.32} we conclude that
\begin{align}
\sup_\tau\left[\langle\tau\rangle^{\frac{2s+4-2j}5}\int_{\{\xi:|\tau+\xi^5|<2\}}\langle\xi\rangle^{-2s+2j}d\xi \right]^{\frac12}\leq C.\label{V2-3.33}
\end{align}

From \eqref{V2-3.28}, \eqref{V2-3.29}, and \eqref{V2-3.33}, it follows that
\begin{align}
\|I(x,\cdot_t)\|_{H^{\frac{s+2-j}5}(\mathbb R_t)}\leq C\|F\|_{X^{s,-b}}.\label{V2-3.34}
\end{align}

Let us estimate $\|II(x,\cdot_t)\|_{H^{\frac{s+2-j}5}(\mathbb R_t)}$. Taking into account \eqref{V2-3.27} and the fact that $\frac{s+2-j}5<\frac12\leq 2$ for $0\leq s\leq \frac12$ and $j=0,1,2$, we have that

\begin{align*}
\|II(x,\cdot_t)\|_{H^{\frac{s+2-j}5}(\mathbb R_t)}&\leq C\|\eta\|_{H^2(\mathbb R_t)}\left\|\int_{\mathbb R^2} \xi^j \frac{e^{i\cdot_t\tau}}{i(\tau+\xi^5)}e^{i\xi x}\psi^c(\tau+\xi^5)\widehat F(\xi,\tau)d\xi d\tau \right\|_{H^{\frac{s+2-j}5}(\mathbb R_t)}\\
&\leq C\left\| \int_{\mathbb R_\tau} e^{i\cdot_t\tau}\left(\int_{\mathbb R_\xi} \xi^j \frac{e^{i\xi x}}{i(\tau+\xi^5)}\psi^c(\tau+\xi^5)\widehat F(\xi,\tau)d\xi\right) d\tau \right\|_{H^{\frac{s+2-j}5}(\mathbb R_t)}\\
&\leq C\left\|\langle\cdot_\tau\rangle^{\frac{s+2-j}5} \left(\int_{\mathbb R_\xi} \xi^j \frac{e^{i\xi x}}{i(\cdot_\tau+\xi^5)}\psi^c(\cdot_\tau+\xi^5)\widehat F(\xi,\cdot_\tau)d\xi\right)  \right\|_{L^2_\tau}.
\end{align*}
For $\tau\in\mathbb R$ let us define the set $A_\tau$ by $A_\tau:=\{\xi:|\tau+\xi^5|\geq 1\}$. Then

\begin{align}
\notag \|II(x,\cdot_t)\|_{H^{\frac{s+2-j}5}(\mathbb R_t)}&\leq C\left\|\langle\cdot_\tau\rangle^{\frac{s+2-j}5} \left(\int_{A_\tau}\frac{|\xi|^j}{\langle\cdot_\tau+\xi^5\rangle}|\widehat F(\xi,\cdot_\tau)|d\xi\right)  \right\|_{L^2_\tau}\\
\notag&\leq C\left[\int_{\mathbb R}\langle\tau\rangle^{\frac{2s+4-2j}5}\left( \int_{A_\tau}\frac {|\xi|^{2j}}{\langle\tau+\xi^5\rangle^{2-2b}\langle\xi\rangle^{2s}}d\xi\right)\left(\int_{A_\tau}\frac{\langle\xi\rangle^{2s}}{\langle\tau+\xi^5\rangle^{2b}}|\widehat F(\xi,\tau)|^2 d\xi\right)d\tau \right]^{\frac12}\\
&\leq C\sup_\tau\left( \langle\tau\rangle^{\frac{2s+4-2j}5}\int_{A_\tau}\frac{|\xi|^{2j}}{\langle\tau+\xi^5\rangle^{2-2b}\langle\xi\rangle^{2s}}d\xi\right)^{\frac12}\|F\|_{X^{s,-b}}.\label{V2-3.35}
\end{align}

Let us see that
$$\sup_\tau\left( \langle\tau\rangle^{\frac{2s+4-2j}5}\int_{A_\tau}\frac{|\xi|^{2j}}{\langle\tau+\xi^5\rangle^{2-2b}\langle\xi\rangle^{2s}}d\xi\right)^{\frac12}\leq C<\infty.$$

Let us observe that
\begin{align}
\notag\int_{A_\tau}\frac{|\xi|^{2j}}{\langle\tau+\xi^5\rangle^{2-2b}\langle\xi\rangle^{2s}}d\xi=\int_{A_\tau\cap\{\xi:|\xi|<1\}}\frac{|\xi|^{2j}}{\langle\tau+\xi^5\rangle^{2-2b}\langle\xi\rangle^{2s}}d\xi+\int_{A_\tau\cap\{\xi:|\xi|\geq1\}}\frac{|\xi|^{2j}}{\langle\tau+\xi^5\rangle^{2-2b}\langle\xi\rangle^{2s}}d\xi.\\
\leq \frac C{\langle\tau\rangle^{2-2b}}+C\int_{|x|\geq1}\frac{|x^{\frac15}|^{2j}}{\langle x+\tau\rangle^{2-2b}\langle x^{\frac15}\rangle^{2s}|x|^{\frac45}}dx
\leq \frac C{\langle\tau\rangle^{2-2b}}+C\int_{|x|\geq 1}\frac1{\langle x+\tau\rangle^{2-2b}\langle x\rangle^{\frac{2s+4-2j}5}} dx.
\label{V2-3.36}
\end{align}

Let us use the calculus inequality \eqref{V2-3.39}: defining $\beta:=2-2b$, and $\gamma:=\frac{2s+4-2j}5$, since $0\leq s\leq\frac12$, and $b<\frac12$, we have that $\beta\geq \gamma>0$, and $\beta+\gamma>1$, and in consequence from \eqref{V2-3.36} it follows that
\begin{align}
\sup_\tau\left(\langle\tau\rangle^{\frac{2s+4-2j}5}\int_{A_\tau}\frac{|\xi|^{2j}}{\langle\tau+\xi^5\rangle^{2-2b}\langle\xi\rangle^{2s}}d\xi \right)^{\frac12}\leq C \sup_\tau\left(\frac{\langle\tau\rangle^{\frac{2s+4-2j}5}}{\langle\tau\rangle^{2-2b}}+\frac{\langle\tau\rangle^{\frac{2s+4-2j}5}}{\langle\tau\rangle^{\frac{2s+4-2j}5}} \right)^{\frac12}\leq C.\label{V2-3.41}
\end{align}

From \eqref{V2-3.35}, and \eqref{V2-3.41} it follows that
\begin{align}
\|II(x,\cdot_t)\|_{H^{\frac{s+2-j}5}(\mathbb R_t)}\leq C\|F\|_{X^{s,-b}}.\label{V2-3.42}
\end{align}

Let us estimate $\|III(x,\cdot_t)\|_{H^{\frac{s+2-j}5}(\mathbb R_t)}$.\\
\begin{align}
\notag\|III(x,\cdot_t)\|_{H^{\frac{s+2-j}5}(\mathbb R_t)}\leq& C \left\|\eta(\cdot_t)\int_{\mathbb R}\left(\int_{|\xi|\geq 1} \xi^j \frac{e^{-i \cdot_t\xi^5}}{i(\tau+\xi^5)}e^{i\xi x}\psi^c(\tau+\xi^5)\widehat F(\xi,\tau)d\xi \right)d\tau \right\|_{H^{\frac{s+2-j}5}(\mathbb R_t)}\\
\notag&+ C\left\|\eta(\cdot_t)\int_{\mathbb R}\left(\int_{|\xi|< 1} \xi^j \frac{e^{-i\cdot_t\xi^5}}{i(\tau+\xi^5)}e^{i\xi x}\psi^c(\tau+\xi^5)\widehat F(\xi,\tau)d\xi \right)d\tau \right\|_{H^{\frac{s+2-j}5}(\mathbb R_t)}\\
\equiv&III_1(x)+III_2(x).\label{V2-3.43}
\end{align}

Let us estimate $III_1(x):$
\begin{align*}
III_1(x)&\leq C\|\eta\|_{H^2(\mathbb R_t)}\left\|\int_{\mathbb R}\left(\int_{|\xi|\geq 1} \xi^j \frac{e^{-i\cdot_t\xi^5}}{i(\tau+\xi^5)}e^{i\xi x}\psi^c(\tau+\xi^5)\widehat F(\xi,\tau)d\xi \right)d\tau \right\|_{H^{\frac{s+2-j}5}(\mathbb R_t)}\\
&\leq C\left\|\int_{\mathbb R}\left(\int_{|\gamma|\geq 1} (-\gamma)^{\frac j5} \frac{e^{i\cdot_t\gamma}}{i(\tau-\gamma)}e^{i(-\gamma)^{\frac15}x}\psi^c(\tau-\gamma)\widehat F((-\gamma)^{\frac15},\tau)\frac1{\gamma^{\frac45}}d\gamma \right)d\tau \right\|_{H^{\frac{s+2-j}5}(\mathbb R_t)}\\
&= C\left\|\int_{\mathbb R_\gamma} (-1)^{\frac j5} e^{i\cdot_t\gamma} \frac{e^{i(-\gamma)^{\frac15}x}}{\gamma^{\frac{4-j}5}} \chi_{\{\gamma:|\gamma|\geq 1\}}(\gamma)\left(\int_{\mathbb R_\tau}\frac{\psi^c(\tau-\gamma)}{i(\tau-\gamma)}\widehat F((-\gamma)^{\frac15},\tau)d\tau \right)d\gamma \right\|_{H^{\frac{s+2-j}5}(\mathbb R_t)}\\
&= C\left\|\langle\cdot_\gamma\rangle^{\frac{s+2-j}5} \frac{e^{i({\scriptstyle-}\cdot_\gamma)^{\frac15}x}}{(\cdot_\gamma)^{\frac{4-j}5}} \chi_{\{\gamma:|\gamma|\geq 1\}}(\cdot_\gamma)\left(\int_{\{\tau:|\tau-\gamma|\geq1\}}\frac{\psi^c(\tau{\scriptstyle-}\cdot_\gamma)}{i(\tau{\scriptstyle-}\cdot_\gamma)}\widehat F(({\scriptstyle-}\cdot_\gamma)^{\frac15},\tau)d\tau \right) \right\|_{L^2_\gamma}\\
&\leq C\left\|\frac{\langle \cdot_\gamma\rangle^{\frac{s+2-j}5}}{(\cdot_\gamma)^{\frac{4-j}5}}\int_{\{\tau:|\tau-\gamma|\geq1\}}\frac1{|\tau{\scriptstyle-}\cdot_\gamma|}|\widehat F(({\scriptstyle-}\cdot_\gamma)^{\frac15},\tau)|d\tau  \right\|_{L^2_{|\gamma|\geq1}}\\
&= C\left\|\frac{\langle \cdot_\gamma\rangle^{\frac{s+2-j}5}}{(\cdot_\gamma)^{\frac{4-j}5}}\int_{\{\tau:|\tau+\gamma|\geq1\}}\frac1{|\tau+\cdot_\gamma|}|\widehat F(({\scriptstyle-}\cdot_\gamma)^{\frac15},-\tau)|d\tau  \right\|_{L^2_{|\gamma|\geq1}}.
\end{align*}

Since $|\tau+\gamma|\geq1$, $|\tau+\gamma|\sim\langle\tau+\gamma\rangle$. And, since $|\gamma|\geq1$, $|\gamma|\sim\langle\gamma\rangle$. Therefore
\begin{align}
\notag III_1(x)&\leq C\left\| \langle \cdot_\gamma\rangle^{\frac{s-2}5}\int_{\{\tau:|\tau+\gamma|\geq1\}}\frac1{\langle\tau+\cdot_\gamma\rangle}|\widehat F(({\scriptstyle-}\cdot_\gamma)^{\frac15},-\tau)|d\tau\right\|_{L^2_{|\gamma|\geq 1}}\\
\notag &\leq C\left[ \int_{|\gamma|\geq1}\langle\gamma\rangle^{\frac{2s-4}5}\left(\int_{\mathbb R}\frac1{\langle\tau+\gamma\rangle}|\widehat F((-\gamma)^{\frac15},-\tau)|d\tau \right)^2 d\gamma\right]^{\frac12}\\
\notag &\leq C\left[ \int_{|\gamma|\geq1}\langle\gamma\rangle^{\frac{2s-4}5}\left(\int_{\mathbb R}\frac1{\langle\tau+\gamma\rangle^{2-2b}}d\tau \right) \left( \int_{\mathbb R}\frac{|\widehat F((-\gamma)^{\frac15},-\tau)|^2}{\langle\tau+\gamma\rangle^{2b}}d\tau\right) d\gamma\right]^{\frac12}\\
\notag &\leq C\left[ \int_{|\gamma|\geq1}\langle\gamma\rangle^{\frac{2s-4}5} \left( \int_{\mathbb R}\frac{|\widehat F((-\gamma)^{\frac15},-\tau)|^2}{\langle\tau+\gamma\rangle^{2b}}d\tau\right) d\gamma\right]^{\frac12}\\
\notag &\leq C\left[ \int_{|\tilde\gamma|\geq1}\int_{\mathbb R_{\tau'}} \frac{\langle\tilde\gamma^5\rangle^{\frac{2s-4}5}}{\langle\tau'+\tilde \gamma^5\rangle^{2b}}\tilde \gamma^4|\widehat F(\tilde\gamma,\tau')|^2d\tau' d\tilde\gamma\right]^{\frac12}\\
&\leq C\left[ \int_{|\tilde\gamma|\geq1}\int_{\mathbb R_{\tau'}} \frac{\langle\tilde\gamma\rangle^{2s}}{\langle\tau'+\tilde \gamma^5\rangle^{2b}}|\widehat F(\tilde\gamma,\tau')|^2d\tau' d\tilde\gamma\right]^{\frac12}\leq C\|F\|_{X^{s,-b}}.\label{V2-3.44}
\end{align}

Let us estimate now $III_2(x)$:

\begin{align*}
III_2(x)&=\left\| \eta(\cdot_t)\int_{\mathbb R}\left(\int_{|\xi|<1} \xi^j \frac{e^{-i\cdot_t\xi^5}}{i(\tau+\xi^5)}e^{i\xi x}\psi^c(\tau+\xi^5)\widehat F(\xi,\tau)d\xi\right)d\tau\right\|_{H^{\frac{s+2-j}5}(\mathbb R_t)}\\
&\leq \int_{\mathbb R}\int_{|\xi|<1}\|\eta(\cdot_t)e^{-i\cdot_t\xi^5}e^{i\xi x}\|_{H^{\frac{s+2-j}5}(\mathbb R_t)}\frac{|\xi|^{j}}{|\tau+\xi^5|}\psi^c(\tau+\xi^5)|\widehat F(\xi,\tau)|d\xi d\tau.
\end{align*}

Since, for $|\xi|<1$,
\begin{align*}
\|\eta(\cdot_t)e^{-i\cdot_t\xi^5}e^{i\xi x}\|_{H^{\frac{s+2-j}5}(\mathbb R_t)}\leq \|\eta(\cdot_t)e^{-i\cdot_t\xi^5}\|_{H^1(\mathbb R_t)}\leq C,
\end{align*}

then
\begin{align}
\notag III_2(x)&\leq C\int_{\mathbb R}\int_{\mathbb R}\frac 1{\langle\tau+\xi^5\rangle}\chi_{(-1,1)}(\xi)|\widehat F(\xi,\tau)|d\xi d\tau\\
\notag&\leq C\left(\int_{\mathbb R}\int_{\mathbb R}\frac 1{{\langle\tau+\xi^5\rangle}^{2-2b}}\chi_{(-1,1)}(\xi)d\xi d\tau\right)^{\frac12}\left(\int_{\mathbb R}\int_{\mathbb R} \frac{|\widehat F(\xi,\tau)|^2}{\langle\tau+\xi^5\rangle^{2b}}d\xi d\tau\right)^{\frac12}\\
&\leq C\left(\int_{-1}^1\left(\int_{\mathbb R}\frac 1{{\langle\tau+\xi^5\rangle}^{2-2b}}d\tau\right)d\xi \right)^{\frac12}\|F\|_{X^{0,-b}}\leq C\|F\|_{X^{s,-b}},\label{V2-3.45}
\end{align}

for $b<\frac12$. From \eqref{V2-3.43} to \eqref{V2-3.45} we conclude that
\begin{align}
\|III(x,\cdot_t)\|_{H^{\frac{s+2-j}5}(\mathbb R_t)}\leq C\|F\|_{X^{s,-b}}.\label{V2-3.46}
\end{align}

Taking into account \eqref{V2-3.26}, \eqref{V2-3.34}, \eqref{V2-3.42}, and \eqref{V2-3.46}, it follows that
\begin{align}
\|\eta(\cdot_t) \partial_x^j \int_0^{\cdot_t} [W_{\mathbb R}(\cdot_t-t')F(t')](x)dt'\|_{H^{\frac{s+2-j}5}(\mathbb R_t)}\leq C\|F\|_{X^{s,-b}},\label{V2-3.47}
\end{align}

where $C$ is independent of $x$.\\

Continuity in $x$ follows from the uniform bound \eqref{V2-3.47} and from the dominated convergence theorem. This way, statement (i) of Lemma \ref{V3-L4.6} is proved.

\item[(ii)] Let us assume that $s=\frac{11}2$, which is equivalent to $\frac{s+2-j}5=\frac{15-2j}{10}$, and $\frac{2s-1-10b}{10}=1-b$, and let us prove that, for $x\in\mathbb R$ fixed,
\begin{align}
\left\|\eta(\cdot_t) \partial_x^j \int_0^{\cdot_t}\left\{W_{\mathbb R}(\cdot_t{\scriptstyle-}t')F(t') \right\}(x)dt' \right\|_{H^{\frac{15-2j}{10}}(\mathbb R_t)}\leq C\left(\|F\|_{X^{\frac12,1-b}}+\|F\|_{X^{\frac{11}2,-b}} \right),\label{V2-3.48}
\end{align}
where $C$ is independent of $x$.\\

We will use the following inequality, for a good enough function $f$:
\begin{align}
\|f\|^2_{H^{\frac{15-2j}{10}}(\mathbb R_t)}\leq C\int_{\mathbb R} |\widehat f(\tau)|^2 d\tau + C \int_{\mathbb R} |\tau|^{\frac{5-2j}{5}} |\tau \widehat f(\tau)|^2 d\tau\leq C \|f\|^2_{L^2(\mathbb R_t)}+C \|f'\|^2_{H^{\frac{\frac12+2-j}{5}}(\mathbb R_t)} .\label{V2-3.49}
\end{align}

It is clear that
\begin{align*}
\left\|\eta(\cdot_t) \partial_x^j \int_0^{\cdot_t}\left\{W_{\mathbb R}(\cdot_t{\scriptstyle-}t')F(t') \right\}(x)dt' \right\|_{L^2(\mathbb R_t)}\leq \left\|\eta(\cdot_t) \partial_x^j\int_0^{\cdot_t}\left\{W_{\mathbb R}(\cdot_t{\scriptstyle-}t')F(t') \right\}(x)dt' \right\|_{H^{\frac{\frac12+2-j}{5}}(\mathbb R_t)}.
\end{align*}

Then, from the proof of (i) taking $s=\frac12$,
\begin{align}
\left\|\eta(\cdot_t) \partial_x^j \int_0^{\cdot_t}\left\{W_{\mathbb R}(\cdot_t{\scriptstyle-}t')F(t') \right\}(x)dt' \right\|_{L^2(\mathbb R_t)}\leq C\|F\|_{X^{\frac12,-b}}\leq  C\|F\|_{X^{\frac12,1-b}}.\label{V2-3.50}
\end{align}

Let us observe now that
\begin{align}
\notag &\frac d{dt}\left[\eta(t) \partial_x^j \int_0^t \left\{W_{\mathbb R}(t-t')F(t') \right\}(x) dt' \right]\\
\notag=&\eta'(t) \partial_x^j \int_0^t \left\{W_{\mathbb R}(t-t')F(t') \right\}(x)dt'+C\eta(t) \partial_x^j \int_{\mathbb R_2} e^{i\xi x}\frac{e^{it\tau} - e^{-it\xi^5} } {(\tau+\xi^5)}(-\xi^5)\widehat F(\xi,\tau)d\xi d\tau\\
&+C \eta(t) \partial_x^j \int_{\mathbb R^2} e^{i\xi x} \frac{i\xi^5 e^{it\tau}+i\tau e^{it\tau}}{i(\tau+\xi^5)} \widehat F(\xi,\tau) d\xi d\tau. \label{V2-3.51}
\end{align}

Using (i) with $s=\frac12$, we have that
\begin{align}
\left\|\eta'(\cdot_t) \partial_x^j \int_0^{\cdot_t}\left\{ W_{\mathbb R}(\cdot_t{\scriptstyle-}t')F(t')\right\}(x)dt'\right\|_{H^{\frac{\frac12+2-j}{5}}(\mathbb R_t)}&\leq C\|F\|_{X^{\frac12,-b}},\text{ and}\label{V2-3.52}\\
\left\| \eta(\cdot_t) \partial_x^j \int_{\mathbb R^2} e^{i\xi x}\frac{e^{it\tau}-e^{-it\xi^5}}{\tau+\xi^5} (-\xi^5)\widehat F(\xi,\tau)d\xi d\tau\right\|_{H^{\frac{\frac12+2-j}{5}}(\mathbb R_t)}&\leq C\|G_1\|_{X^{\frac12,-b}},\label{V2-3.53}
\end{align}
where $\widehat G_1(\xi,\tau):=-\xi^5\widehat F(\xi,\tau)$.\\

On the other hand, in the same way as it was proved the estimative for $II(x,\cdot_t)_{H^{\frac{s+2-j}5}(\mathbb R_t)}$ with $s=\frac12$, it can be shown that
\begin{align}
\left\|\eta(\cdot_t) \partial_x^j \int_{\mathbb R^2} e^{i\xi x} \frac{e^{i\cdot_t\tau}}{i\langle\tau+\xi^5\rangle} i\langle\tau+\xi^5\rangle \widehat F(\xi,\tau)d\xi d\tau\right\|_{H^{\frac{\frac12+2-j}{5}}(\mathbb R_t)}\leq C\|G_2\|_{X^{\frac12,-b}},\label{V2-3.54}
\end{align}

where $\widehat G_2:= i\langle \tau + \xi^5 \rangle \widehat F(\xi,\tau)$.\\

Since
\begin{align*}
\|G_1\|_{X^{\frac12,-b}}&=\left(\int_{\mathbb R^2}\langle\xi\rangle\langle\tau+\xi^5\rangle^{-2b}|\xi|^{10}|\widehat F(\xi,\tau)|^2 d\xi d\tau \right)^{\frac12}\leq\|F\|_{X^{\frac{11}2,-b}},
\end{align*}
and
\begin{align*}
\|G_2\|_{X^{\frac12,-b}}&=\left(\int_{\mathbb R^2}\langle\xi\rangle\langle\tau+\xi^5\rangle^{-2b}\langle\tau+\xi^5\rangle^2|\widehat F(\xi,\tau)|^2d\xi d\tau \right)^{\frac12}=\|F\|_{X^{\frac12,1-b}},
\end{align*}

then, from \eqref{V2-3.49} to \eqref{V2-3.54}, it follows \eqref{V2-3.48}, i.e. we have proved \eqref{V3-L4.6.2} for $s=\frac{11}2$. Using interpolation, statement (ii) follows for $\frac12<s<\frac{11}2$. Lemma \ref{V3-L4.6} is proved.
\end{enumerate}
\end{proof}

\begin{lemma}\label{V2-L3.6} (Estimative of the bilinear form). Let $s\geq 0$,  $\frac 25\leq b<\frac12$, $0\leq a\leq 10b-4$, and $\frac 12<\alpha\leq 1$. There exists $C>0$ such that
\begin{align}
 \|\partial_x(vw)\|_{X^{s+a,-b}}\leq C\|v\|_{X^{s,b}}\|w\|_{X^{s,b}},\label{V1-3.58}\quad\text{and}\\
 \|\partial_x(vw)\|_{Y^{s,-b,\alpha}}\leq C\|v\|_{X^{s,b,\alpha}}\|w\|_{X^{s,b,\alpha}}.\label{V2-3.58}
\end{align}
\end{lemma}

\begin{proof} 
\begin{enumerate}
\item[(i)] Let us begin by estimating the norm  $\|\partial_x(vw)\|_{X^{s+a,-b}}$. We observe that
$$[\partial_x(vw)]^\wedge(\xi,\tau)=Ci\xi\int_{\mathbb R_2} \widehat v(\xi_1,\tau_1)\widehat w(\xi-\xi_1,\tau-\tau_1)d\xi_1 d\tau_1.$$

Hence
\begin{align}
\notag \|\partial_x(vw)\|^2_{X^{s+a,-b}}&=C\int_{\mathbb R^2_{\xi\tau}}\langle\xi\rangle^{2(s+a)}\langle\tau+\xi^5\rangle^{-2b}\xi^2\left|\int_{\mathbb R^2_{\xi_1\tau_1}}\widehat v(\xi_1,\tau_1)\widehat w(\xi-\xi_1,\tau-\tau_1)d\xi_1 d\tau_1 \right|^2d\xi d\tau\\
&=C\int_{\mathbb R^2_{\xi\tau}}\left| \int_{\mathbb R^2_{\xi_1\tau_1}} \langle\xi\rangle^{(s+a)}\langle\tau+\xi^5\rangle^{-b}|\xi|\widehat v(\xi_1,\tau_1)\widehat w(\xi-\xi_1,\tau-\tau_1)d\xi_1 d\tau_1 \right|^2d\xi d\tau.\label{V2-3.59}
\end{align}

Let $h\in L^2(\mathbb R^2_{\xi\tau})$ an arbitrary function. If we manage to prove that
\begin{align}
\left|\int_{\mathbb R^2_{\xi\tau}} \left[\int_{\mathbb R^2_{\xi_1\tau_1}}\langle\xi\rangle^{(s+a)}\langle\tau+\xi^5\rangle^{-b}|\xi|\widehat v(\xi_1,\tau_1)\widehat w(\xi-\xi_1,\tau-\tau_1)d\xi_1 d\tau_1 \right]h(\xi,\tau)d\xi d\tau\right|\leq C\|v\|_{X^{s,b}}\|w\|_{X^{s,b}}\|h\|_{L^2(\mathbb R^2)},\label{V2-3.60}
\end{align}

then we would have, by a duality argument, that
\begin{align}
\|\partial_x(vw)\|_{X^{s+a,-b}}\leq C\|v\|_{X^{s,b}}\|w\|_{X^{s,b}}.\label{FTY}
\end{align}

Taking into account that, for $s\geq 0$, there exists $C>0$, such that
$$\frac{\langle\xi\rangle^s}{\langle\xi_1\rangle^s\langle\xi-\xi_1\rangle^s}\leq C,$$
then, to establish \eqref{V2-3.60}, it is enough to prove that
\begin{align}
\notag\left|\int_{\mathbb R^2_{\xi\tau}} \left[\int_{\mathbb R^2_{\xi_1\tau_1}}\langle\xi\rangle^{a}\langle\xi_1\rangle^s\langle\tau+\xi^5\rangle^{-b}|\xi|\langle\xi-\xi_1\rangle^s|\widehat v(\xi_1,\tau_1)||\widehat w(\xi-\xi_1,\tau-\tau_1)|d\xi_1 d\tau_1 \right]h(\xi,\tau)d\xi d\tau\right|\\
\leq C\|v\|_{X^{s,b}}\|w\|_{X^{s,b}}\|h\|_{L^{2}(\mathbb R^2)}.\label{V2-3.61}
\end{align}

Since
\begin{align*}
&\left|\int_{\mathbb R^2_{\xi\tau}} \left[\int_{\mathbb R^2_{\xi_1\tau_1}}\langle\xi\rangle^a\langle\xi_1\rangle^s\langle\tau+\xi^5\rangle^{-b}|\xi|\langle\xi-\xi_1\rangle^s|\widehat v(\xi_1,\tau_1)||\widehat w(\xi-\xi_1,\tau-\tau_1)|d\xi_1 d\tau_1 \right]h(\xi,\tau)d\xi d\tau\right|\\
&\leq \int_{\mathbb R^2_{\xi\tau}} \int_{\mathbb R^2_{\xi_1\tau_1}}\frac{\langle\xi\rangle^a|\xi||h(\xi,\tau)| \langle\xi_1\rangle^s\langle\tau_1+\xi_1^5\rangle^{b}|\widehat v(\xi_1,\tau_1)|\langle\xi-\xi_1\rangle^s\langle\tau-\tau_1+(\xi-\xi_1)^5\rangle^b |\widehat w(\xi-\xi_1,\tau-\tau_1)|}{\langle\tau+\xi^5\rangle^b\langle\tau_1+\xi_1^5\rangle^b\langle\tau-\tau_1+(\xi-\xi_1)^5\rangle^b}d\xi_1d\tau_1 d\xi d\tau\\
&\leq\int_{\mathbb R^2_{\xi\tau}}\left[ \int_{\mathbb R^2_{\xi_1\tau_1}}\frac{\langle\xi\rangle^{2a}|\xi|^2|h(\xi,\tau)|^2}{\langle\tau+\xi^5\rangle^{2b}\langle\tau_1+\xi_1^5\rangle^{2b}\langle\tau-\tau_1+(\xi-\xi_1)^5\rangle^{2b}}d\xi_1 d\tau_1\right]^{\frac12}\\
&\hspace{1.3cm}\left[\int_{\mathbb R^2_{\xi_1\tau_1}}\langle\xi_1\rangle^{2s}\langle\tau_1+\xi_1^5\rangle^{2b}|\widehat v(\xi_1,\tau_1)|^2\langle\xi-\xi_1\rangle^{2s}\langle\tau-\tau_1+(\xi-\xi_1)^5\rangle^{2b}|\widehat w(\xi-\xi_1,\tau-\tau_1)|^2d\xi_1 d\tau_1 \right]^{\frac12}d\xi d\tau\\
&\leq\left[\int_{\mathbb R^2_{\xi\tau}}\int_{\mathbb R^2_{\xi_1\tau_1}}\frac{\langle\xi\rangle^{2a}|\xi|^2|h(\xi,\tau)|^2}{\langle\tau+\xi^5\rangle^{2b}\langle\tau_1+\xi_1^5\rangle^{2b}\langle\tau-\tau_1+(\xi-\xi_1)^5\rangle^{2b}}d\xi_1 d\tau_1 d\xi d\tau\right]^{\frac12}\\
&\hspace{0.5cm}\left[\int_{\mathbb R^2_{\xi\tau}}\int_{\mathbb R^2_{\xi_1\tau_1}}\langle\xi_1\rangle^{2s}\langle\tau_1+\xi_1^5\rangle^{2b}|\widehat v(\xi_1,\tau_1)|^2\langle\xi-\xi_1\rangle^{2s}\langle\tau-\tau_1+(\xi-\xi_1)^5\rangle^{2b}|\widehat w(\xi-\xi_1,\tau-\tau_1)|^2d\xi_1 d\tau_1d\xi d\tau \right]^{\frac12}\\
&=\left[ \int_{\mathbb R^2_{\xi\tau}}|h(\xi,\tau)|^2\left(\int_{\mathbb R^2_{\xi_1\tau_1}}\frac{\langle\xi\rangle^{2a}|\xi|^2}{\langle\tau+\xi^5\rangle^{2b}\langle\tau_1+\xi_1^5\rangle^{2b}\langle\tau-\tau_1+(\xi-\xi_1)^5\rangle^{2b}}d\xi_1 d\tau_1 \right)d\xi d\tau\right]^{\frac12}\|v\|_{X^{s,b}}\|w\|_{X^{s,b}},
\end{align*}

then, to prove \eqref{V2-3.61}, it is enough to prove that
\begin{align}
\sup_{(\xi,\tau)\in\mathbb R^2}\int_{\mathbb R^2_{\xi_1\tau_1}}\frac{\langle\xi\rangle^{2a}|\xi|^2}{\langle\tau+\xi^5\rangle^{2b}\langle\tau_1+\xi_1^{5}\rangle^{2b}\langle\tau-\tau_1+(\xi-\xi_1)^5\rangle^{2b}}d\xi_1d\tau_1\leq C.\label{V2-3.62}
\end{align}

Let us observe that
\begin{align*}
\int_{\mathbb R_{\tau_1}}\frac{\langle\xi\rangle^{2a}|\xi|^2}{\langle\tau+\xi^5\rangle^{2b}\langle\tau_1+\xi_1^{5}\rangle^{2b}\langle\tau-\tau_1+(\xi-\xi_1)^5\rangle^{2b}}d\tau_1=\frac{\langle\xi\rangle^{2a}|\xi|^2}{\langle\tau+\xi^5\rangle^{2b}}\int_{\mathbb R_{\tau_1}}\frac1{\langle\tau_1-(-\xi_1^5)\rangle^{2b}\langle\tau_1-(\tau+(\xi-\xi_1)^5)\rangle^{2b}}d\tau_1.
\end{align*}

Using inequality \eqref{V2-3.39}, with $\beta=\gamma=2b<1$, and $\beta+\gamma=4b>1$ $(b>\frac14)$, we conclude that
\begin{align}
\int_{\mathbb R_{\xi_1}}\int_{\mathbb R_{\tau_1}}\frac{\langle\xi\rangle^{2a}|\xi|^2}{\langle\tau+\xi^5\rangle^{2b}\langle\tau_1+\xi_1^5\rangle^{2b}\langle\tau-\tau_1+(\xi-\xi_1)^5\rangle^{2b}}d\tau_1d\xi_1\leq C \int_{\mathbb R_{\xi_1}}\frac{\langle\xi\rangle^{2a}|\xi|^2}{\langle\tau+\xi^5\rangle^{2b}\langle\xi_1^5+\tau+(\xi-\xi_1)^5\rangle^{4b-1}}d\xi_1.\label{V2-3.63}
\end{align}

Let us make the following change of variables in the integral of the right hand side of \eqref{V2-3.63}:
\begin{align}
\notag \mu&:=\xi_1^5+\tau+(\xi-\xi_1)^5=\tau+\xi^5-\frac52\xi\xi_1(\xi-\xi_1)[\xi^2+\xi_1^2+(\xi-\xi_1)^2]\\
&=\tau+\frac1{16}\xi^5+\frac52\xi^3(\xi_1-\tfrac\xi2)^2+5\xi(\xi_1-\tfrac\xi2)^4,\label{V2-3.64}
\end{align}

\begin{align}
d\mu=-\frac52\xi(\xi-2\xi_1)[\xi^2+(\xi-2\xi_1)^2]d\xi_1.\label{V2-3.65}
\end{align}

Therefore
\begin{align}
\notag\int_{\mathbb R_{\xi_1}}\int_{\mathbb R_{\tau_1}}&\frac{\langle\xi\rangle^{2a}|\xi|^2}{\langle\tau+\xi^5\rangle^{2b}\langle\tau_1+\xi_1^5\rangle^{2b}\langle\tau-\tau_1+(\xi-\xi_1)^5\rangle^{2b}}d\tau_1d\xi_1\\
&\leq C\frac{\langle\xi\rangle^{2a}|\xi|^2}{\langle\tau+\xi^5\rangle^{2b}}\int_{A\subset\mathbb R_{\mu}}\frac{1}{(1+|\mu|)^{4b-1}|\xi||\xi-2\xi_1|[\xi^2+(\xi-2\xi_1)^2]}d\mu,\label{V2-3.66}
\end{align}

where $A=[\tau+\tfrac1{16}\xi^5,+\infty)$ when $\xi>0$.\\

Using the last equality of \eqref{V2-3.64}, it can be seen that
$$\left(\xi_1-\tfrac\xi2 \right)^2=-\frac{\xi^2}4+\sqrt{\frac{\xi^5-4\tau+4\mu}{20\xi}},$$
then
\begin{align*}
(\xi-2\xi_1)^2&=-\xi^2+\sqrt{\frac4{5\xi}(\xi^5-4\tau+4\mu)},\\
\xi_1&=\frac\xi2\pm\sqrt{\frac12\sqrt{\frac1{5\xi}(\xi^5-4\tau+4\mu)}-\frac14\xi^2},\\
2\xi_1-\xi&=\pm\sqrt{-\xi^2+\sqrt{\frac4{5\xi}(\xi^5-4\tau+4\mu)}}.
\end{align*}

Therefore, from \eqref{V2-3.66}, we have that
\begin{align}
\notag &\int_{\mathbb R^2_{\xi_1\tau_1}}\frac{\langle\xi\rangle^{2a}|\xi|^2}{\langle\tau+\xi^5\rangle^{2b}\langle\tau_1+\xi_1^5\rangle^{2b}\langle\tau-\tau_1+(\xi-\xi_1)^5\rangle^{2b}}d\tau_1 d\xi_1\\
\notag&\leq C\frac{\langle\xi\rangle^{2a}|\xi|}{\langle\tau+\xi^5\rangle^{2b}}\int_{A\subset \mathbb R_\mu}\frac1{(1+|\mu|)^{4b-1}\sqrt{-\xi^2+\sqrt{\frac4{5\xi}(\xi^5-4\tau+4\mu)}}\sqrt{\frac4{5\xi}(\xi^5-4\tau+4\mu)}}d\mu\\
&\leq C\frac{\langle\xi\rangle^{2a}|\xi|}{\langle\tau+\xi^5\rangle^{2b}}\int_{A\subset \mathbb R_\mu}\frac{\sqrt{\xi^2+\sqrt{\frac4{5\xi}(\xi^5-4\tau+4\mu)}}}{(1+|\mu|)^{4b-1}\sqrt{\frac{-16\tau+16\mu-\xi^5}{5\xi}}\sqrt{\frac4{5\xi}(\xi^5-4\tau+4\mu)}}d\mu.\label{V2-3.67}
\end{align}

Let us assume, without loss of generality, that $\xi>0$, and let us observe that the function
$$\mu:=\varphi(\xi_1)=\tau+\xi_1^5+(\xi-\xi_1)^5$$
has its absolute minimum in $\xi_1=\frac\xi2$. Then, the integration interval $A$ in \eqref{V2-3.67} is
$$A=[\varphi(\tfrac\xi2),+\infty)=[\tau+\tfrac1{16}\xi^5,+\infty).$$

This way
$$\mu\geq\tau+\frac1{16}\xi^5,\quad \xi^5-4\tau+4\mu\geq \xi^5-4\tau+4\tau+\frac14\xi^5=\frac54\xi^5,$$
\begin{align}
\frac4{5\xi}(\xi^5-4\tau+4\mu)&\geq\frac4{5\xi}\frac54\xi^5=\xi^4,\label{V2-3.68}\\
\sqrt{\frac4{5\xi}(\xi^5-4\tau+4\mu)}&\geq\xi^2.\label{V2-3.69}
\end{align}

Hence, taking into account \eqref{V2-3.69}, \eqref{V2-3.68}, from \eqref{V2-3.67}, we conclude that
\begin{align}
\notag &\int_{\mathbb R^2_{\xi_1\tau_1}}\frac{\langle\xi\rangle^{2a}|\xi|^2}{\langle\tau+\xi^5\rangle^{2b}\langle\tau_1+\xi_1^5\rangle^{2b}\langle\tau-\tau_1+(\xi-\xi_1)^5\rangle^{2b}}d\tau_1 d\xi_1\\
\notag &\leq C\frac{\langle\xi\rangle^{2a}|\xi|}{\langle\tau+\xi^5\rangle^{2b}}\int_{\mathbb R_\mu}\frac1{\langle \mu\rangle^{4b-1}\sqrt{\frac{-16\tau+16\mu-\xi^5}{5\xi}}|\xi|}d\mu\\
&\leq C\frac{\langle\xi\rangle^{2a}|\xi|^{\frac12}}{\langle\tau+\xi^5\rangle^{2b}}\int_{\mathbb R_\mu}\frac1{\langle \mu\rangle^{4b-1}\sqrt{|\mu-\tau-\frac1{16}\xi^5|}}d\mu.\label{V2-3.70}
\end{align}

Using the calculus inequality \eqref{V2-3.71} (Lemma \ref{CI}) with $\rho:=4b-1$ ($4b-1\in(\tfrac12,1)\iff \frac38<b<\frac12$) in \eqref{V2-3.70}, it follows that
\begin{align}
\int_{\mathbb R^2_{\xi_1\tau_1}}\frac{\langle\xi\rangle^{2a}|\xi|^2}{\langle\tau+\xi^5\rangle^{2b}\langle\tau_1+\xi_1^5\rangle^{2b}\langle\tau-\tau_1+(\xi-\xi_1)^5\rangle^{2b}}d\tau_1 d\xi_1\leq C\frac{\langle\xi\rangle^{2a}|\xi|^{\frac12}}{\langle\tau+\xi^5\rangle^{2b}\langle\tau+\tfrac1{16}\xi^5\rangle^{4b-\frac32}}.\label{V2-3.72}
\end{align}

It is clear, from \eqref{V2-3.72}, for $|\xi|\leq 1$, and $\frac38< b<\frac12$, that
\begin{align}
\int_{\mathbb R^2_{\xi_1\tau_1}}\frac{\langle\xi\rangle^{2a}|\xi|^2}{\langle\tau+\xi^5\rangle^{2b}\langle\tau_1+\xi_1^5\rangle^{2b}\langle\tau-\tau_1+(\xi-\xi_1)^5\rangle^{2b}}d\tau_1 d\xi_1\leq C.\label{V2-3.73}
\end{align}
Let us assume then that $|\xi|>1$. We observe that
$$|\xi^5|\leq \frac{16}{15}|\tau+\xi^5|+\frac{16}{15}|\tau+\tfrac1{16}\xi^5|,$$
then
\begin{align*}
\frac12|\xi^5|\leq \frac{16}{15}|\tau+\xi^5|\quad \text{or}\quad \frac12|\xi^5|\leq \frac{16}{15}|\tau+\tfrac1{16}\xi^5|.
\end{align*}

For $\frac12|\xi^5|\leq \frac{16}{15}|\tau+\xi^5|$, we have that
\begin{align}
C\frac{\langle\xi\rangle^{2a}|\xi|^{\frac12}}{\langle\tau+\xi^5\rangle^{2b}\langle\tau+\tfrac1{16}\xi^5\rangle^{4b-\frac32}}\leq C\frac{|\xi|^{2a+\frac12}}{\langle\tau+\xi^5\rangle^{2b}}\leq C\frac{|\xi|^{2a+\frac12}}{\langle\tfrac{15}{32}|\xi^5|\rangle^{2b}}\leq C|\xi|^{2a+\frac12-10b}\leq C,\label{V2-3.74}
\end{align}

since $2a+\frac12-10b\leq20b-8+\frac12-10b=10b-\frac{15}2<0$.\\

For $\frac12|\xi^5|\leq\frac{16}{15}|\tau+\tfrac1{16}\xi^5|$, we have that
\begin{align}
C\frac{\langle\xi\rangle^{2a}|\xi|^{\frac12}}{\langle\tau+\xi^5\rangle^{2b}\langle\tau+\tfrac1{16}\xi^5\rangle^{4b-\frac32}}\leq C\frac{|\xi|^{2a+\frac12}}{\langle\tau+\tfrac1{16}\xi^5\rangle^{4b-\frac32}}\leq C\frac{|\xi|^{2a+\frac12}}{\langle\tfrac{15}{32}|\xi^5|\rangle^{4b-\frac32}}\leq C|\xi|^{2a+\frac12-20b+\frac{15}2}\leq C,\label{V2-3.75}
\end{align}
since $2a+\frac12-20b+\frac{15}2\leq20b-8+\frac12-20b+\frac{15}2=0$.\\

This way, from \eqref{V2-3.72} to \eqref{V2-3.75}, it follows that, for $\frac25<b<\frac12$ and $0\leq a\leq 10b-4$, inequality \eqref{V2-3.62} is true, which proves \eqref{FTY}.

\item[(ii)] In order to estimate $\|\partial_x(vw)\|_{Y^{s,-b,\alpha}}$ we must consider consider three terms. The first of them is $\|\partial_x(vw)\|_{X^{s,-b}}$, which is bounded by $C\|v\|_{X^{s,b}}\|w\|_{X^{s,b}}$ (this follows from \eqref{FTY} with $a=0$). Now, we continue by estimating the second term in the definition of $\|\partial_x(vw)\|_{Y^{s,-b,\alpha}}$.\\

Proceeding as in (i), taking $h\in L^2(\mathbb R^2_{\xi\tau})$ arbitrary, if we manage to proof that
\begin{align}
\left|\int_{\mathbb R_\tau} \int_{-1}^1 \left[\int_{\mathbb R^2_{\xi_1\tau_1}}\langle \tau \rangle^{\alpha-1} |\xi| \widehat v(\xi_1,\tau_1)\widehat w(\xi-\xi_1,\tau-\tau_1)d\xi_1 d\tau_1 \right]h(\xi,\tau)d\xi d\tau\right|\leq C\|v\|_{X^{s,b}}\|w\|_{X^{s,b}}\|h\|_{L^2(\mathbb R^2)},\label{V4-4.6.1}
\end{align}
by a duality argument we would obtain

\begin{align}
\left(\int_{-\infty}^{+\infty} \int_{-1}^1 \langle \tau \rangle^{2(\alpha-1)} |[\partial_x(vw)]^\wedge (\xi,\tau) |^2 d\xi d\tau \right)^{\frac12}\leq C \|v\|_{X^{s,b}} \|w\|_{X^{s,b}}\leq C \|v\|_{X^{s,b,\alpha}} \|w\|_{X^{s,b,\alpha}}.\label{STY}
\end{align}

Since $\alpha-1\leq 0$, we observe, applying Cauchy Schwarz inequality, that
\begin{align*}
&\left|\int_{\mathbb R_{\tau}}\int_{-1}^1 \left[\int_{\mathbb R^2_{\xi_1\tau_1}}\langle \tau \rangle^{\alpha-1} |\xi| |\widehat v(\xi_1,\tau_1)||\widehat w(\xi-\xi_1,\tau-\tau_1)|d\xi_1 d\tau_1 \right]h(\xi,\tau)d\xi d\tau\right|\\
&\leq \int_{\mathbb R_{\tau}}\int_{-1}^1 \int_{\mathbb R^2_{\xi_1\tau_1}} \frac{|\xi||h(\xi,\tau)| \langle\tau_1+\xi_1^5\rangle^{b}|\widehat v(\xi_1,\tau_1)|\langle\tau-\tau_1+(\xi-\xi_1)^5\rangle^b |\widehat w(\xi-\xi_1,\tau-\tau_1)|}{\langle\tau_1+\xi_1^5\rangle^b \langle\tau-\tau_1+(\xi-\xi_1)^5\rangle^b}d\xi_1d\tau_1 d\xi d\tau\\
&\leq \left[ \int_{\mathbb R_{\tau}} \int_{-1}^1 |h(\xi,\tau)|^2\left(\int_{\mathbb R^2_{\xi_1\tau_1}}\frac{|\xi|^2}{\langle\tau_1+\xi_1^5\rangle^{2b}\langle\tau-\tau_1+(\xi-\xi_1)^5\rangle^{2b}}d\xi_1 d\tau_1 \right)d\xi d\tau\right]^{\frac12}\|v\|_{X^{s,b}}\|w\|_{X^{s,b}}.
\end{align*}

From \eqref{V2-3.72} with $a=0$, for $|\xi|\leq1$, and $\frac38<b<\frac12$, we have that
\begin{align*}
\int_{\mathbb R^2_{\xi_1\tau_1}}\frac{|\xi|^2}{\langle\tau_1+\xi_1^5\rangle^{2b}\langle\tau-\tau_1+(\xi-\xi_1)^5\rangle^{2b}}d\xi_1 d\tau_1\leq C.
\end{align*}

Therefore
\begin{align*}
\left|\int_{\mathbb R_{\tau}}\int_{-1}^1 \left[\int_{\mathbb R^2_{\xi_1\tau_1}}\langle \tau \rangle^{\alpha-1} |\xi| |\widehat v(\xi_1,\tau_1)||\widehat w(\xi-\xi_1,\tau-\tau_1)|d\xi_1 d\tau_1 \right]h(\xi,\tau)d\xi d\tau\right| \leq C \|v\|_{X^{s,b}} \|w\|_{X^{s,b}} \|h\|_{L^2(\mathbb R^2)},
\end{align*}

wich implies \eqref{STY}.

 Finally, we estimate the third term in the definition of $\| \partial_x(vw)\|_{Y^{s,-b,\alpha}}$. Let us notice that
\begin{align*}
&\left(\int_{-\infty}^{+\infty} \left( \int_{-\infty}^{+\infty} \langle \xi \rangle^{s} \langle \tau + \xi^5 \rangle^{-1} | [\partial_x(vw)]^\wedge(\xi,\tau) | d\tau \right)^2 d\xi \right)^{\frac12}\\
&\leq \left(\int_{-\infty}^{+\infty}\left( \int_{-\infty}^{+\infty} \langle \xi \rangle^{2s} \langle \tau + \xi^5 \rangle^{-2b} | [\partial_x(vw)]^\wedge(\xi,\tau) |^2 d\tau\right) \left( \int_{-\infty}^{+\infty} \frac1{\langle \tau + \xi^5 \rangle^{2(1-b)}} d\tau \right) d\xi \right)^{\frac12}.
\end{align*}

Taking into account that $2(1-b)>1$, using \eqref{FTY}, we obtain
\begin{align*}
\left(\int_{-\infty}^{+\infty} \left( \int_{-\infty}^{+\infty} \langle \xi \rangle^{s} \langle \tau + \xi^5 \rangle^{-1} | [\partial_x(vw)]^\wedge(\xi,\tau) | d\tau \right)^2 d\xi \right)^{\frac12}&\leq C \|\partial_x(vw)\|_{X^{s,-b}}\leq C\|v\|_{X^{s,b,\alpha}} \|w\|_{X^{s,b,\alpha}}.
\end{align*}
 
\end{enumerate}

Lemma \ref{V2-L3.6} is proved.
\end{proof}

\begin{lemma}\label{V2-L3.7} (Another bilinear estimative). Let $\frac12<s<\frac{11}4$,  $0\leq a<\frac{11}4-s$ and $\max\left\{\frac{s+a}5-\frac1{20},\frac25\right\}<b<\frac12$. There exists $C>0$, such that
\begin{align}
\|\partial_x(vw)\|_{X^{\frac12,\frac{2(s+a)-1-10b}{10}}}\leq C\|v\|_{X^{s,b}}\|w\|_{X^{s,b}}.\label{V2-3.76}
\end{align}
\end{lemma}

\begin{proof} To prove \eqref{V2-3.76}, we proceed as in the proof of Lemma \ref{V2-L3.6}, by establishing that
\begin{align}
\sup_{(\xi,\tau)\in\mathbb R^2}\int_{\mathbb R^2_{\xi_1\tau_1}}\int\frac{\langle\xi\rangle^{1-2(s+a)}\langle\xi\rangle^{2a}|\xi|^2}{\langle\tau+\xi^5\rangle^{\frac{10b+1-2(s+a)}5}\langle\tau_1+\xi_1^5\rangle^{2b}\langle\tau-\tau_1+(\xi-\xi_1)^5\rangle^{2b}}d\xi_1 d\tau_1\leq C.\label{V2-3.77}
\end{align}

Let us use inequality \eqref{V2-3.39} with $\beta=\gamma=2b<1$, and $\beta+\gamma=4b>1$ (i.e. $\frac14<b<\frac12$), to obtain
\begin{align*}
&\int_{\mathbb R_{\xi_1}}\left( \int_{\mathbb R_{\tau_1}}\frac1{\langle\tau_1+\xi_1^5\rangle^{2b}\langle\tau-\tau_1+(\xi-\xi_1)^5\rangle^{2b}}d\tau_1\right)\frac{\langle\xi\rangle^{1-2(s+a)}\langle\xi\rangle^{2a}|\xi|^2}{\langle\tau+\xi^5\rangle^{\frac{10b+1-2(s+a)}{5}}}d\xi_1\\
&\leq C\int_{\mathbb R_{\xi_1}}\frac{\langle\xi\rangle^{1-2(s+a)}\langle\xi\rangle^{2a}|\xi|^2}{\langle\tau+\xi^5\rangle^{\frac{10b+1-2(s+a)}5}\langle\xi_1^5+\tau+(\xi-\xi_1)^5\rangle^{4b-1}}d\xi_1.
\end{align*}

As in the proof of Lemma \ref{V2-L3.6}, we have that, if $\frac38<b<\frac12$,
\begin{align}
\notag&\int_{\mathbb R_{\xi_1}}\int_{\mathbb R_{\tau_1}}\frac{\langle\xi\rangle^{1-2(s+a)}\langle\xi\rangle^{2a}|\xi|^2}{\langle\tau+\xi^5\rangle^{\frac{10b+1-2(s+a)}5}\langle\tau_1+\xi_1^5\rangle^{2b}\langle \tau-\tau_1+(\xi-\xi_1)^5\rangle^{2b}}d\tau_1 d\xi_1\\
\notag&\leq C\frac{\langle\xi\rangle^{1-2(s+a)}\langle\xi\rangle^{2a}|\xi|^{\frac12}}{\langle\tau+\xi^5\rangle^{\frac{10b+1-2(s+a)}5}}\int_{\mathbb R_{\mu}}\frac{1}{\langle\mu\rangle^{4b-1}\sqrt{|\mu-\tau-\tfrac1{16}\xi^5|}}d\mu\\
&\leq C\frac{\langle\xi\rangle^{1-2(s+a)}\langle\xi\rangle^{2a}|\xi|^{\frac12}}{\langle\tau+\xi^5\rangle^{\frac{10b+1-2(s+a)}5}\langle\tau+\tfrac1{16}\xi^5\rangle^{4b-\frac32}}.\label{V2-3.78}
\end{align}

Let us observe that $1-2(s+a)+2a=1-2s<0$, $10b+1-2(s+a)>2(s+a)-\frac12+1-2(s+a)=\frac12$, and $4b-\frac32>\frac85-\frac32>0$. Hence, from \eqref{V2-3.78}, if $|\xi|<1$

\begin{align}
\int_{\mathbb R^2_{\xi_1\tau_1}}\frac{\langle\xi\rangle^{1-2(s+a)}\langle\xi\rangle^{2a}|\xi|^2}{\langle\tau+\xi^5\rangle^{\frac{10b+1-2(s+a)}5}\langle\tau_1+\xi_1^5\rangle^{2b}\langle \tau-\tau_1+(\xi-\xi_1)^5\rangle^{2b}}d\xi_1 d\tau_1\leq C.\label{V2-3.79}
\end{align}

Let us assume now that $|\xi|>1$. We have two cases:
\begin{enumerate}
\item[(i)] $\frac12|\xi|^5\leq\frac{16}{15}|\tau+\xi^5|$. In this case we have
\begin{align}
C\frac{\langle\xi\rangle^{1-2(s+a)}\langle\xi\rangle^{2a}|\xi|^{\frac12}}{\langle\tau+\xi^5\rangle^{\frac{10b+1-2(s+a)}5}\langle\tau+\tfrac1{16}\xi^5\rangle^{4b-\frac32}}\leq C|\xi|^{1-2(s+a)+2a+\frac12-10b-1+2(s+a)}\leq C,\label{V2-3.80}
\end{align}

since $2a+\frac12-10b<2a+\frac12-2(s+a)+\frac12=1-2s<0$.\\

\item[(ii)] $\frac12|\xi|^5\leq \frac{16}{15}|\tau+\frac1{16}\xi^5|$. In this case we have
\begin{align}
C\frac{\langle\xi\rangle^{1-2(s+a)}\langle\xi\rangle^{2a}|\xi|^{\frac12}}{\langle\tau+\xi^5\rangle^{\frac{10b+1-2(s+a)}5}\langle\tau+\tfrac1{16}\xi^5\rangle^{4b-\frac32}}\leq C|\xi|^{1-2(s+a)+2a+\frac12-20b+\frac{15}2}\leq C,\label{V2-3.81}
\end{align}

since $1-2(s+a)+2a+\frac12-20b+\frac{15}2=9-2s-20b<9-2s-8=1-2s<0$.\\

From \eqref{V2-3.78} to \eqref{V2-3.81} it follows that inequality \eqref{V2-3.77} is true, which proves Lemma \ref{V2-L3.7}.

\end{enumerate}

\end{proof}

\section{Existence of a local solution in time of the IBVP \eqref{maineq}}\label{V2-S4}

The objective in this section is to prove the existence of a solution of the integral equation \eqref{V2-1.7} and then to show in what sense the restriction of this solution to $\mathbb R^+\times[0,T]$ is a solution of the IBVP \eqref{maineq}.

\begin{Theorem}\label{V3-T5.1} Let $s\in[0,\tfrac{11}4)\setminus \{\frac12,\frac32,\frac52\}$, and let $g\in H^s(\mathbb R_x^+)$, and $h_{j+1}\in H^{\frac{s+2-j}5}(\mathbb R^+_t)$, $j=0,1,2,$ with the additional compatibility conditions:
\begin{enumerate}
\item[(i)] If $\frac12<s<\frac32$, $g(0)=h_1(0)$;
\item[(ii)] If $\frac32<s<\frac52$, $g(0)=h_1(0)$, $g'(0)=h_2(0)$;
\item[(iii)] If $\frac52<s<\frac{11}4$, $g(0)=h_1(0)$, $g'(0)=h_2(0)$, $g''(0)=h_3(0)$.
\end{enumerate}

For $\max\{\tfrac s5-\tfrac1{20},\tfrac25\}<b<\tfrac12$, and $\alpha\in(\frac12,1-b)$ there exist $T\in(0,\tfrac12]$ and an unique $u$, solution of the integral equation \eqref{V2-1.7}, such that
\begin{align}
u\in X^{s,b,\alpha}(\mathbb R^2)\cap C(\mathbb R_t;H^s(\mathbb R_x)),\quad \partial_x^j u\in C(\mathbb R_x;H^{\frac{s+2-j}{5}}(\mathbb R_t)),\, j=0,1,2.\label{V3-5.1}
\end{align}
Besides, for $n\in\mathbb N$, if $u_n$ is the solution of the integral equation associated to the extension $g_{n_l}\in H^s(\mathbb R_x)$ of $g_n\in H^s(\mathbb R^+_x)$, and to the boundary data $h_{n,j+1}\in H^{\frac{s+2-j}5}(\mathbb R_t^+)$, $j=0,1,2$, and $g_n\to g$ in $H^s(\mathbb R^+_x)$, and $h_{n,j+1}\to h_{j+1}$ in $H^{\frac{s+2-j}5}(\mathbb R^+_t)$, $j=0,1,2$, when $n\to +\infty$, then $u_n\to u$ in the spaces considered in \eqref{V3-5.1}.
\end{Theorem}

\begin{proof} We will prove that there exists $T\in (0,\tfrac12]$ such that the operator $\Gamma_T$ defined for $u\in X^{s,b,\alpha}(\mathbb R^2)$ by
\begin{align}
(\Gamma_T u)(t):=\eta(t)W_{\mathbb R}(t)g_l+\eta(t)\int_0^t W_{\mathbb R}(t-t')F_T(u(t'))dt'+\eta(t) W_0^t(0,h_1-p_1,h_2-p_2,h_3-p_3)(t),\label{V3-5.2}
\end{align}
has a fixed point in $X^{s,b,\alpha}(\mathbb R^2)$.\\

To see that $\Gamma_T$ sends some closed ball from $X^{s,b,\alpha}$ into itself, let us recall that using Lemma \ref{V2-L3.1},
\begin{align}
\|\eta(\cdot_t)[W_{\mathbb R}(\cdot_t)g_l](\cdot_x)\|_{X^{s,b,\alpha}}\leq C\|g_l\|_{H^s(\mathbb R_x)}\leq C\|g\|_{H^s(\mathbb R^+_x)}.\label{V3-5.3}
\end{align}

On the other hand, by Lemma \ref{V2-L3.3} (ii), if we take $b^*\in(b,\frac12)$ such that $\frac12<\alpha<1-b^*$,
\begin{align}
\left\|\eta(\cdot_t)\int_0^{\cdot_t}W_{\mathbb R}(\cdot_t{\scriptstyle-}t')F_T(u(t'))dt'\right\|_{X^{s,b,\alpha}}\leq C\|F_T(u)\|_{Y^{s,-b^*,\alpha}}.\label{V3-5.4}
\end{align}

Therefore, by inequality \eqref{V4.2-4.7} from Lemma \ref{V2-L3.4}, and Lemma \ref{V2-L3.6}, for $0<T\leq\frac12$, and $\frac25\leq b<\frac12$, we have that
\begin{align}
\notag\|F_T(u)\|_{Y^{s,-b^*,\alpha}}&\leq C \left\| \eta(\frac{\cdot_t}{2T}) (-\tfrac12 \partial_x u(\cdot_t)^2) \right\|_{X^{s,-b^*}}\leq CT^{b^*-b}\left\| \partial_x u^2(\cdot_t)\right\|_{X^{s,-b}}\leq C T^{b^*-b}\|u\|^2_{X^{s,b}}\\
&\leq C T^{b^*-b} \|u\|^2_{X^{s,b,\alpha}}.\label{V3-5.5}
\end{align}

which lead us, from \eqref{V3-5.4}, and \eqref{V3-5.5} to
\begin{align}
\left\|\eta(\cdot_t)\int_0^{\cdot_t}[W_{\mathbb R}(\cdot_t{\scriptstyle-}t')F_T(u(t'))](\cdot_x) dt' \right\|_{X^{s,b,\alpha}}\leq CT^{b^*-b}\|u\|^2_{X^{s,b,\alpha}}.\label{V3-5.6}
\end{align}

By Lemma \ref{V3-L1}, $\eta(\cdot_t) \partial_x^j [W_{\mathbb R}(\cdot_t)g_l](0)\in H^{\frac{s+2-j}{5}}(\mathbb R_t)$, $j=0,1,2$, and
\begin{align}
\| \eta(\cdot_t) \partial_x^j[W_{\mathbb R}(\cdot_t) g_l](0)\|_{H^{\frac{s+2-j}{5}}(\mathbb R_t)} \leq C \|g\|_{H^s(\mathbb R_x^+)}.\label{V3-5.7}
\end{align}

By Lemma \ref{V3-L4.6} (i), inequality \eqref{V2-3.13} from Lemma \ref{V2-L3.4}, and inequality \eqref{V1-3.58} from Lemma \ref{V2-L3.6}, for $0\leq s\leq \frac12$, $\frac25\leq b<b^*<\frac12$, $j=0,1,2$, and $0<T\leq \frac12$, we have that
\begin{align}
\notag\| \eta(\cdot_t) \partial_x^j \int_0^{\cdot_t} [W_{\mathbb R}(\cdot_t {\scriptstyle-}t') F_T(u(t'))](0) dt' \|_{H^{\frac{s+2-j}{5}}(\mathbb R_t)}&\leq C\| F_T(u(\cdot))\|_{X^{s,-b^*}}\leq C T^{b^*-b} \| \partial_x u^2(\cdot)\|_{X^{s,b}}^2\\
&\leq C T^{b^*-b} \|u\|^2_{X^{s,b}}. \label{V3-5.8}
\end{align}

By Lemma \ref{V3-L4.6} (ii), inequality \eqref{V2-3.13} from Lemma \ref{V2-L3.4}, Lemma \ref{V2-L3.7}, and inequality \eqref{V1-3.58} from Lemma \ref{V2-L3.6}, for $\frac12\leq s<\frac{11}4$, $\max\{\tfrac s5-\tfrac1{20}, \tfrac25\}<b<b^*<\tfrac12$, $j=0,1,2$, and $0<T\leq \frac12$, we obtain
\begin{align}
\notag \| \eta(\cdot_t) \partial_x^j \int_0^{\cdot_t} [W_{\mathbb R}(\cdot_t {\scriptstyle-}t') F_T(u(t'))](0) dt' \|_{H^{\frac{s+2-j}{5}}(\mathbb R_t)} &\leq C \|F_T(u(\cdot))\|_{X^{\frac12,\frac{2s-1-10b^*}{10}}} + C \|F_T(u(\cdot))\|_{X^{s,-b^*}}\\
\notag&\leq C T^{b^*-b} \|\partial_x u^2(\cdot)\|_{X^{\frac12,\frac{2s-1-10b}{10}}}+CT^{b^*-b} \|\partial_x u^2(\cdot)\|_{X^{s,-b}}\\
&\leq C T^{b^*-b} \|u\|_{X^{s,b}}^2. \label{V3-5.9}
\end{align}

From \eqref{V3-5.7} to \eqref{V3-5.9}, we conclude, for $0\leq s\leq \frac{11}4$, $\max\{\frac s5-\frac1{20},\frac25\}<b<\frac12$, and $j=0,1,2$, that $p_{j+1}\in H^{\frac{s+2-j}{5}}(\mathbb R_t)$.
\begin{enumerate}
\item[(i)] For $0\leq s<\frac12$, since $0\leq \frac{s+2-j}5<\frac12$, $j=0,1,2$, then, by Lemma \ref{V2-L1.1} (i),
\begin{align}
\| h_{j+1} - p_{j+1}\|_{H^{\frac{s+2-j}{5}}(\mathbb R_t^+)} \leq \|\chi_{(0,+\infty)}(h_{j+1}-p_{j+1})\|_{H^{\frac{s+2-j}{5}}(\mathbb R_t)}\leq C \|h_{j+1}-p_{j+1}\|_{H^{\frac{s+2-j}{5}}(\mathbb R_t^+)}.\label{V3-5.10}
\end{align}

\item[(ii)] For $\frac12<s<\frac32$, $0\leq \frac{s+2-j}5<\frac12$, $j=1,2$, then \eqref{V3-5.10} is valid for $j=1,2$. If $j=0$, since $\frac12<\frac{s+2}5$, and
\begin{align}
h_1(0)-p_1(0)=h_1(0)-g_l(0)=h_1(0)-g(0)=0,\label{V3-5.11}
\end{align}
then, by Lemma \ref{V2-L1.1} (ii), \eqref{V3-5.10} is also valid.

\item[(iii)] For $\frac32<s<\frac52$, $\frac s5<\frac12$, then \eqref{V3-5.10} is valid for $j=2$. For $j=0$, since $\tfrac12<\tfrac{s+2}5$, and \eqref{V3-5.11} is true, then \eqref{V3-5.10} is valid for $j=0$. For $j=1$, since
\begin{align}
h_2(0)-p_2(0)=h_2(0)-\partial_x g_l(0)=h_2(0)-g'(0)=0,\label{V3-5.12}
\end{align}
then, by Lemma \ref{V2-L1.1} (ii), \eqref{V3-5.10} is also valid for $j=1$.

\item[(iv)] For $\frac52<s<\frac{11}4$, and $j=0,1,2$, $\frac12<\frac{s+2-j}5$. Besides
\begin{align}
h_3(0)-p_3(0)=h_3(0)-g''(0)=0.\label{V3-5.13}
\end{align}
Then, taking into account that \eqref{V3-5.11}, \eqref{V3-5.12}, and  \eqref{V3-5.13} are fulfilled, by Lemma \ref{V2-L1.1} (ii) we have that \eqref{V3-5.10} is valid for $j=0,1,2$.
\end{enumerate}

In consequence, by Lemma \ref{V3-L4}, and estimates \eqref{V3-5.10}, \eqref{V3-5.7}, \eqref{V3-5.8}, and \eqref{V3-5.9}, it follows that
\begin{align}
\notag \| \eta(\cdot_t) W_0^t(0,h_1-p_1,h_2-p_2,h_3-p_3)(\cdot_x,\cdot_t)\|_{X^{s,b,\alpha}} &\leq C \left( \sum_{j=0}^2 \| \chi_{(0,+\infty)} (h_{j+1}-p_{j+1})\|_{H^{\frac{s+2-j}{5}}(\mathbb R_t)} \right)\\
\notag &\leq C \sum_{j=0}^2 \| h_{j+1}-p_{j+1} \|_{H^{\frac{s+2-j}{5}}(\mathbb R_t^+)}\\
\notag &\leq C \sum_{j=0}^2 \left( \| h_{j+1} \|_{H^{\frac{s+2-j}{5}}(\mathbb R_t^+)} + \| p_{j+1} \|_{H^{\frac{s+2-j}{5}}(\mathbb R_t)}\right)\\
&\leq C \sum_{j=0}^2 \left( \| h_{j+1} \|_{H^{\frac{s+2-j}{5}}(\mathbb R_t^+)} + \|g\|_{H^s(\mathbb R_x^+)} +T^{b^*-b} \|u\|^2_{X^{s,b,\alpha}}\right).\label{V3-5.13.2}
\end{align}

From estimates \eqref{V3-5.3}, \eqref{V3-5.6}, and \eqref{V3-5.13.2}, we obtain
\begin{align}
\| \Gamma_T u \|_{X^{s,b,\alpha}} \leq C \sum_{j=0}^2 \| h_{j+1} \|_{H^{\frac{s+2-j}{5}}(\mathbb R_t^+)} + C \|g\|_{H^s(\mathbb R_x^+)} + C T^{b^*-b} \|u\|_{X^{s,b,\alpha}}^2. \label{V3-5.14}
\end{align}

Let $B_{X^{s,b,\alpha}}(0;R)$ be the closed ball in $X^{s,b}$ with center in 0 and radius
$$R:=2C\left[\left( \sum_{j=0}^2 \|h_{j+1}\|_{H^{\frac{s+2-j}5}(\mathbb R^+_t)} \right)+\|g\|_{H^s(\mathbb R_x^+)}\right]>0.$$

From \eqref{V3-5.14}, it follows that, for $u\in B_{X^{s,b,\alpha}}(0;R)$,
$$\|\Gamma_T u\|_{X^{s,b,\alpha}}\leq  \tfrac R2+CT^{-b+b^*}R^2.$$

For $T>0$ such that
\begin{align}
CT^{b^*-b}R^2\leq \tfrac R2,\label{V3-5.15}
\end{align}

we have that $\Gamma_T$ sends the ball $B_{X^{s,b,\alpha}}(0;R)$ into itself.\\

It remains to prove that $\Gamma_T:B_{X^{s,b,\alpha}}(0;R)\to B_{X^{s,b,\alpha}}(0;R)$ is a contraction. Let us assume that $\frac12<s<\frac{11}4$, with $s\neq\frac32$, and $s\neq\frac52$ (being the case $0\leq s<\frac12$ simpler). Let us recall that $\max\{\frac s5-\frac1{20},\frac25\}<b<b^*<\frac12$, and $\frac12<\alpha<1-b^*$.\\

Let $v$ and $w$ elements of $B_{X^{s,b,\alpha}}(0;R)$. Then, taking into account the definition of $W_0^t$ given in \eqref{N3.38}, using Lemmas \ref{V2-L3.3}, \ref{V2-L3.4}, \ref{V3-L4}, \ref{V3-L4.6}, and \ref{V2-L3.6}, we obtain

\begin{align}
\notag &\|\Gamma_T v-\Gamma_T w\|_{X^{s,b,\alpha}}\\
\notag&=\Big\|\eta(\cdot_t)\int_0^{\cdot_t} [W_{\mathbb R}(\cdot_t{\scriptstyle-}t')(F_T(v(t'))-F_T(w(t')))](\cdot_x) dt'+\eta(\cdot_t)(W_0^t(0,h_1-p_{1v},h_2-p_{2v},h_3-p_{3v})(\cdot_x,\cdot_t)\\
\notag&\text{} \hspace{7cm} -W_0^t(0,h_1-p_{1w},h_2-p_{2w},h_3-p_{3w})(\cdot_x,\cdot_t)) \Big\|_{X^{s,b,\alpha}}\\
\notag&\leq \left\|\eta(\cdot_t)\int_0^{\cdot_t} [W_{\mathbb R}(\cdot_t{\scriptstyle-}t')(F_T(v(t'))-F_T(w(t')))](\cdot_x) dt'\right\|_{X^{s,b,\alpha}}\\
\notag&\text{} \hspace{7cm} + \left\| \eta(\cdot_t)(W_0^t(0,p_{1w}-p_{1v},p_{2w}-p_{2v},p_{3w}-p_{3v})(\cdot_x,\cdot_t) \right\|_{X^{s,b,\alpha}} \\
\notag &\leq CT^{b^*-b}\|\partial_x(v^2-w^2)\|_{X^{s,-b}}+C \sum_{j=0}^2\|\chi_{(0,+\infty)}(p_{j+1,w}-p_{j+1,v})\|_{H^{\frac{s+2-j}5}(\mathbb R_t)}\\
&\leq CT^{b^*-b}\|v+w\|_{X^{s,b}}\|v-w\|_{X^{s,b}}+C \sum_{j=0}^2 \|p_{j+1,w}-p_{j+1,v}\|_{H^{\frac{s+2-j}5}(\mathbb R^+_t)}.\label{V3-5.16}
\end{align}

Taking into account that, for $j=0,1,2$
\begin{align*}
\|p_{j+1,w} - p_{j+1,v} \|_{H^{\frac{s+2-j}{5}}(\mathbb R_t)} = \left\| \eta(\cdot_t) \partial_x^j \int_0^{\cdot_t} [W_{\mathbb R}(\cdot_t{\scriptstyle-}t')(F_T(v(t'))-F_T(w(t')))](0) dt' \right\|_{H^{\frac{s+2-j}{5}}(\mathbb R_t)},
\end{align*}

then, by Lemmas \ref{V3-L4.6} (ii), \ref{V2-L3.4}, \ref{V2-L3.6}, and \ref{V2-L3.7}, it follows that

\begin{align}
\notag \|p_{j+1,w} - p_{j+1,v} \|_{H^{\frac{s+2-j}{5}}(\mathbb R_t)} &\leq C \|F_T w - F_T v\|_{X^{\frac12,\frac{2s-1-10b^*}{10}}} + C \| F_Tw - F_Tv \|_{X^{s,-b^*}}\\
\notag &\leq CT^{b^*-b} \| \partial_x(v^2 - w^2)\|_{X^{\frac12,\frac{2s-1-10b}{10}}} + CT^{b^*-b} \| \partial_x(v^2 - w^2)\|_{X^{s,-b}}\\
&\leq CT^{b^*-b} \| v + w \|_{X^{s,b}} \| v - w \|_{X^{s,b}}. \label{V3-5.17}
\end{align}

From \eqref{V3-5.16}, and \eqref{V3-5.17}, we conclude that
\begin{align}
\notag \| \Gamma_Tv - \Gamma_Tw\|_{X^{s,b,\alpha}} & \leq CT^{b^*-b}  \| v + w \|_{X^{s,b}} \| v - w \|_{X^{s,b}}\\
& \leq 2CR T^{b^*-b} \| v - w \|_{X^{s,b,\alpha}}. \label{V3-5.18}
\end{align}

Therefore, if we choose $T>0$ such that $2CRT^{b^*-b}<1$, it is clear that such $T$ also satisfies \eqref{V3-5.15}. Hence, from \eqref{V3-5.18}, we conclude that $\Gamma_T:B_{X^{s,b,\alpha}}(0;R)\to B_{X^{s,b,\alpha}}(0;R)$ is a contraction. Let us observe that we can also assume that $T\in (0,\tfrac 12]$. In consequence, there exist $T\in (0,\tfrac12]$, and a unique $u\in B_{X^{s,b}}(0;R)$ such that $\Gamma_T u=u$.\\

Now we will prove that $u\in C(\mathbb R_t;H^s(\mathbb R_x))$. Let us notice that the first term of the right hand side of \eqref{V3-5.2} is continuous from $\mathbb R_t$ with values in $H^s(\mathbb R_x)$. Continuity of the third term of the right hand side of \eqref{V3-5.2} as a function from $\mathbb R_t$ with values in $H^s(\mathbb R_x)$ follows from Lemma \ref{V3-L2} (i). To prove that the second term of the right hand side of \eqref{V3-5.2} is continuous from $\mathbb R_t$ with values in $H^s(\mathbb R_x)$, it is enough to see that it belongs to $X^{s,\tilde b}$, for some $\tilde b>\tfrac12$. Let us prove that the second term of the right hand side of \eqref{V3-5.2} belongs to $X^{s,\alpha}$, for $\frac12<\alpha<1-b^*$. In fact, by \eqref{V4.2-4.5} from Lemma \ref{V2-L3.3} (i), Lemma \ref{V2-L3.4} with $0<T\leq\frac12$, and Lemma \ref{V2-L3.6},

\begin{align*}
\left\| \eta(\cdot_t)\int_0^{\cdot_t}W_{\mathbb R}(\cdot_t{\scriptstyle-}t')F_T(u(t'))dt'\right\|_{X^{s,\alpha}}&\leq C\|F_T(u(\cdot_t))\|_{X^{s,-b^*}}\leq CT^{b^*-b}\|\partial_x u^2\|_{X^{s,-b}}\\
&\leq CT^{b^*-b}\|u\|^2_{X^{s,b}}<\infty.
\end{align*}
This way we have proved that $u=\Gamma_T u\in C(\mathbb R_t; H^s(\mathbb R_x))$.\\

Now we prove that $\partial_x^j u\in C(\mathbb R_x;H^{\frac{s+2-j}5}(\mathbb R_t))$, $j=0,1,2$, for the case $\tfrac12<s<\tfrac{11}4$  with $s\neq\frac32$, and $s\neq\frac52$ (the case $0\leq s<\tfrac12$ being easier).\\

By Lemmas \ref{V3-L1}, and \ref{V3-L2} (i),
\begin{align*}
\eta(\cdot_t)\partial_x^j[W_{\mathbb R}(\cdot_t)g_l](\cdot_x)&\in C(\mathbb R_x;H^{\frac{s+2-j}{5}}(\mathbb R_t)),\\
\eta(\cdot_t)\partial_x^j W_0(0,h_1-p_1,h_2-p_2,h_3-p_3)&\in C(\mathbb R_x;H^{\frac{s+2-j}{5}}(\mathbb R_t)),
\end{align*}
and by Lemmas \ref{V3-L4.6} (ii), \ref{V2-L3.4}, \ref{V2-L3.6}, and \ref{V2-L3.7},
\begin{align*}
\left\|\eta(\cdot_t) \partial_x^j \int_0^{\cdot_t}W_{\mathbb R}(\cdot_t{\scriptstyle-}t')F_T(u(t'))(\cdot_x)dt' \right\|_{C(\mathbb R_x;H^{\frac{s+2-j}5}(\mathbb R_t))}&\leq C\left\|\eta(\tfrac{\cdot_t}{2T})\partial_xu^2\right\|_{X^{\frac12,\frac{2s-1-10b^*}{10}}}+C\left\|\eta(\tfrac{\cdot_t}{2T})\partial_xu^2 \right\|_{X^{s,-b^*}}\\
&\leq CT^{b^*-b}\|u\|^2_{X^{s,b}}<\infty.
\end{align*}

In consequence, $\partial_x^j u=\partial_x^j \Gamma_T u\in C(\mathbb R_x;H^{\frac{s+2-j}5}(\mathbb R_t))$, $j=0,1,2$.\\

Continuous dependence on the initial and boundary data follows from the fixed point arguments and from the paragraphs above. Theorem \ref{V3-T5.1} is proved.
\end{proof}

\begin{Definition}\label{V4-D5.1} For $T>0$ and $U\in L^2(\mathbb R^+ \times (0,T))$ we will say that $U$ is a generalized solution of the IBVP \eqref{maineq} in $\mathbb R^+\times(0,T)$ if for $\phi(\cdot_x,\cdot_t)\in L^\infty([0,T]; H^5(\mathbb R^+_x))$ such that
$$\phi_t\in L^\infty([0,T]; L^2(\mathbb R^+_x)),\, \phi|_{t=T}=0,\, \phi|_{x=0}=\phi_x|_{x=0}=0,$$
it follows that
\begin{align}
\notag &\int_0^T \int_0^{+\infty} [U(x,t)(\phi_t(x,t)+\partial_x^5 \phi(x,t))+\frac12 U^2(x,t)\phi_x(x,t)] dx dt\\
&+\int_0^{+\infty} g(x)\phi(x,0) dx + \int_0^T h_1(t) \partial_x^4 \phi(0,t) dt - \int_0^T h_2(t) \partial_x^3 \phi(0,t) dt + \int_0^T h_3(t) \partial_x^2 \phi(0,t) dt=0.\label{V4.2-5.20}
\end{align}
\end{Definition}

\begin{Theorem}\label{V4.2-T5.2} Let $U$ be the restriction of $u$ to $[0,+\infty)\times [0,T]$, $0<T\leq \frac12$, where $u$ is the solution of the integral equation \eqref{V2-1.7}, which is guaranteed by Theorem \ref{V3-T5.1}. Then $U$ is a generalized solution of the IBVP \eqref{maineq}.
\end{Theorem}

\begin{proof} Since $u\in C(\mathbb R_t;H^s(\mathbb R_x))$ with $s\geq 0$, it is clear that $U\in L^2(\mathbb R^+\times(0,T))$. Let $\phi\in L^\infty([0,T];H^5(\mathbb R^+_x))$ satisfying the conditions in Definition \ref{V4-D5.1}. Using density arguments and integration by parts it can be proved that
\begin{align}
\notag &\int_0^T \int_0^{+\infty} [W_{\mathbb R}(t) g_l](x) (\phi_t(x,t) + \partial_x^5 \phi(x,t)) dx dt\\
\notag  = &-\int_0^{+\infty} g(x) \phi(x,0) dx - \int_0^T [W_{\mathbb R}(t) g_l](0) \partial_x^4 \phi(0,t) dt + \int_0^T \partial_x [W_{\mathbb R}(t) g_l](0)\partial_x^3 \phi(0,t) dt\\
&-\int_0^T \partial_x^2 [W_{\mathbb R}(t) g_l](0) \partial_x^2 \phi(0,t) dt,\label{V4.2-5.21}
\end{align}

\begin{align}
\notag &\int_0^T \int_0^{+\infty} \left( \int_0^t [W_{\mathbb R}(t-t') \left( -\tfrac12 \partial_x(u(t'))^2 \right)](x) dt' \right) \left(\phi_t(x,t) + \partial_x^5 \phi(x,t)  \right) dx dt\\
\notag =&-\int_0^T \int_0^{+\infty} \tfrac12 u^2(x,t) \partial_x\phi(x,t) dx dt - \int_0^T G(u) (0,t) \partial_x^4 \phi(0,t) dt + \int_0^T \partial_x G(u) (0,t) \partial_x^3 \phi(0,t) dt\\
&- \int_0^T \partial_x^2 G(u) (0,t) \partial_x^2 \phi(0,t) dt,\label{V4.2-5.22}
\end{align}

where
\begin{align*}
G(u)(x,t):= \int_0^t [W_{\mathbb R}(t-t')\left(-\tfrac12 \partial_x (u(t'))^2 \right)](x) dt',
\end{align*}

and
\begin{align}
\notag &\int_0^T \int_0^{+\infty} W_0^t(0,h_1-p_1, h_2-p_2,h_3-p_3) (x,t) \left( \phi_t(x,t) + \partial_x^5 \phi(x,t) \right) dx dt\\
&= -\int_0^T (h_1(t)-p_1(t)) \partial_x^4 \phi(0,t) dt + \int_0^T (h_2(t)-p_2(t)) \partial_x^3 \phi(0,t) dt - \int_0^T (h_3(t) - p_3(t)) \partial_x^2 \phi(0,t) dt. \label{V4.2-5.23}
\end{align}

Since $u$ is a solution of the integral equation \eqref{V2-1.7}, from \eqref{V4.2-5.21}, \eqref{V4.2-5.22}, and \eqref{V4.2-5.23}, it follows that
\begin{align*}
&\int_0^T \int_0^{+\infty} U(x,t) (\phi_t(x,t) + \partial_x^5 \phi(x,t)) dx dt\\
= &\int_0^T \int_0^{+\infty} \biggl\{ [W_{\mathbb R}(t) g_l](x) + \int_0^t[W_{\mathbb R}(t-t')\left(-\tfrac12 \partial_x (u(t'))^2\right)](x) dt'\\
& + W_0^t (0,h_1-p_2,h_2-p_2,h_3-p_3)(x,t)\biggl\}(\phi_t(x,t) + \partial_x^5 \phi(x,t)) dx dt\\
=& -\int_0^{+\infty} g(x) \phi(x,0) dx - \int_0^T p_1(t) \partial_x^4 \phi(x,0) dt + \int_0^T p_2(t) \partial_x^3 \phi(x,0) dt - \int_0^T p_3(t) \partial_x^2 \phi(x,0) dt\\
&- \int_0^T \int_0^{+\infty} \tfrac12 U^2(x,t) \partial_x \phi(x,t) dx dt - \int_0^T(h_1(t) - p_1(t))\partial_x^4 \phi(x,0) dt + \int_0^T(h_2(t) - p_2(t))\partial_x^3 \phi(x,0) dt\\
&- \int_0^T(h_3(t) - p_3(t))\partial_x^2 \phi(x,0) dt\\
=& -\int_0^{+\infty} g(x) \phi(x,0) dx - \int_0^T h_1(t) \partial_x^4 \phi(0,t) dt + \int_0^T h_2 \partial_x^3 \phi(0,t) dt - \int_0^T h_3 \partial_x^2 \phi(0,t) dt\\
& - \int_0^T \int_0^{+\infty} \tfrac12 U^2(x,t) \partial_x \phi(x,t) dx dt.
\end{align*}

i.e., $U$ is a generalized solution of the IBVP \eqref{maineq} in $\mathbb R^+\times(0,T)$. Theorem \ref{V4.2-T5.2} is proved.

\end{proof}

\section{Uniqueness}\label{Uniq}

Our objective in this section is to prove that the generalized solution of the IBVP \eqref{maineq} given by Theorem \ref{V4.2-T5.2} is unique. The method we use is based on the one followed by Faminskii in \cite{F2005}.\\

Let us take $\psi\in C_0^\infty(\mathbb R)$ such that $\psi\geq 0$, $\psi\equiv 1$ in $[-1,1]$, $\supp \psi \subset [-2,2]$, and for $\delta>0$, let us define $\psi_\delta\in C_0^\infty(\mathbb R)$ by $\psi_\delta(t)=\psi(\tfrac t\delta)$, for $t\in\mathbb R$.

\begin{lemma} \label{V4.2-L6.1} Let $s\geq 0$, $b\in(0,\tfrac12)$, $\alpha\in(\tfrac12,\tfrac12+b)$, and $T>0$. There exists $C>0$ such that for $f\in X^{s,b,\alpha}$, and $\delta\in[0,T]$,
\begin{align}
\|\psi_\delta(\cdot_t) f\|_{X^{s,b,\alpha}} \leq C \delta^{\frac12-\alpha} \|f\|_{X^{s,b,\alpha}}.\label{V4.2-6.4}
\end{align}

\end{lemma}

\begin{proof} We begin by proving that
\begin{align}
\|\psi_\delta(\cdot_t) f\|_{X^{s,b}}\leq C \delta^{\frac12-\alpha} \|f\|_{X^{s,b}}.\label{V4.2-6.5}
\end{align}

Let us define
$$I_{\tau_0}:=\int_{\mathbb R} (1+|\tau-\tau_0|)^{2b}|\widehat{\psi_\delta f}(\xi,\tau)|^2 d\tau,$$
and let us prove that there exists $C>0$, independent of $\tau_0$, such that
\begin{align}
I_{\tau_0}\leq C \delta^{1-2\alpha} \int_{\mathbb R} (1+|\tau-\tau_0|)^{2b} |\widehat f(\xi,\tau)|^2 d\tau,\label{V4.2-6.6}
\end{align}
for every $\tau_0\in\mathbb R$.\\

From the definition of $I_{\tau_0}$, by expressing the Fourier transform of the product of $\psi_\delta$ and $f$ as a convolution, it can be seen that
\begin{align}
\notag I_{\tau_0}&\leq C \int_{\mathbb R}\left(\int_{\mathbb R} (1+|\zeta|)^b |\widehat \psi_\delta(\zeta)| |\widehat f(\xi,\tau-\zeta)| d\zeta \right)^2 d\tau + C \int_{\mathbb R}\left(\int_{\mathbb R} |\widehat \psi_\delta(\zeta)| (1+|\tau-\zeta-\tau_0|)^b |\widehat f(\xi,\tau-\zeta)| d\zeta \right)^2 d\tau\\
&\equiv C I_1 + C I_2.\label{V4.2-6.7}
\end{align}

Let us estimate $I_1$:
\begin{align*}
I_1=\|g\ast h\|_{L^2}^2,
\end{align*}

where $g(\zeta)=(1+|\zeta|)^b |\widehat \psi_\delta(\zeta)|$, and $h(\tau)=\widehat f(\xi,\tau)$.\\

Let us choose $p\in(1,2)$, and $q\in(1,2)$ such that $\frac1p+\frac1q=\frac32$, with $\frac1p\in(1-\alpha+b,b+\frac12)$. Since $\frac1p + \frac1q-1=\frac32-1=\frac12$, using Young's inequality, the definition of $\psi$, an adequate change of variables, and taking into account that $\delta\in(0,T]$, it can be deduced that
\begin{align}
\notag I_1&\leq \left( \int_{\mathbb R} (1 + |\zeta|)^{bq} |\widehat \psi_\delta(\zeta)|^q d\zeta \right)^{\frac2q} \left( \int_{\mathbb R} |\widehat f(\xi,\tau)|^p d\tau \right)^{\frac2p}\\
\notag &= \delta^{2-\frac2q-2b} \left( \int_{\mathbb R} (\delta + |\zeta'|)^{bq} \delta^{q-1} |\widehat \psi(\zeta')|^q d\zeta' \right)^{\frac2q} \left( \int_{\mathbb R} |\widehat f(\xi,\tau)|^p d\tau \right)^{\frac2p}\\
& \leq C_1 \delta^{2-2b-\frac2q} \left( \int_{\mathbb R} |\widehat f(\xi,\tau)|^p d\tau \right)^{\frac2p} \leq C_1 \delta^{1-2\alpha} \left( \int_{\mathbb R} |\widehat f(\xi,\tau)|^p d\tau \right)^{\frac2p} , \label{V4.2-6.10}
\end{align}

where $C_1$ is a constant independent of $\delta\in(0,T]$. Now, we use Hölder's inequality with exponents $\frac2{2-p}$, and $\frac2p$, to bound the integral on the right hand side of \eqref{V4.2-6.10}:

\begin{align}
\left( \int_{\mathbb R} |\widehat f(\xi,\tau)|^p d\tau \right)^{\frac2p} & \leq \left( \int_{\mathbb R} \frac1{(1+ |\tau-\tau_0|)^{\frac{2bp}{2-p}}} d\tau \right)^{\frac{2-p}p} \int_{\mathbb R} (1 + |\tau - \tau_0|)^{2b} |\widehat f(\xi,\tau)|^2 d\tau \leq C \int_{\mathbb R} (1 + |\tau - \tau_0|)^{2b} |\widehat f(\xi,\tau)|^2 d\tau, \label{V4.2-6.11}
\end{align}

since $\frac{2bp}{2-p}>1$. Hence, from \eqref{V4.2-6.10}, and \eqref{V4.2-6.11},
\begin{align}
I_1 \leq C_1 \delta^{1-2\alpha} \int_{\mathbb R} (1 + |\tau - \tau_0|)^{2b} |\widehat f(\xi,\tau)|^2 d\tau.\label{V4.2-6.12}
\end{align}

We proceed by estimating $I_2$. For this purpose we define $k(\zeta):= (1 + |\zeta - \tau_0|)^b |\widehat f(\xi,\tau)|$, and we consider the integral in the variable $\zeta$ of the definition of $I_2$ as the convolution $|\widehat \psi_\delta| \ast k$, to see that
\begin{align*}
I_2 & = \| |\widehat \psi_\delta| \ast k \|_{L^2}^2 \leq \left( \int_{\mathbb R} |\widehat \psi_\delta (\zeta)| d\zeta \right)^2 \int_{\mathbb R} (1 + |\tau - \tau_0|)^{2b} |\widehat f(\xi,\tau)|^2 d\tau.
\end{align*}

Let us observe that
\begin{align}
\notag\left( \int_{\mathbb R} |\widehat \psi_\delta (\zeta)| d\zeta \right)^2 & \leq \left( \int_{\mathbb R} \frac1{(1 + |\zeta|)^{2\alpha}} d\zeta \right) \left( \int_{\mathbb R} (1 + |\zeta|)^{2\alpha} |\widehat \psi_\delta (\zeta)|^2 d\zeta \right)\\
& \leq C \int_{\mathbb R} (1 + |\zeta|)^{2\alpha} |\widehat \psi_\delta (\zeta)|^2 d\zeta,\label{V4.2-6.13}
\end{align}
since $\alpha>\frac12$.

Therefore
\begin{align}
I_2\leq C \left(\int_{\mathbb R} (1 + |\zeta|)^{2\alpha} |\widehat \psi_\delta(\zeta) |^2 d\zeta \right)  \int_{\mathbb R} (1 + |\tau - \tau_0|)^{2b} |\widehat f(\xi,\tau)|^2 d\tau.\label{V4.2-6.14}
\end{align}

But, taking into account the definition of $\psi_\delta$, and an adequate change of variables, we have that
\begin{align}
\int_{\mathbb R} (1 + |\zeta|)^{2\alpha} |\widehat \psi_\delta (\zeta)|^2 d\zeta = \delta^{1-2\alpha} \int_{\mathbb R} (\delta + |\zeta'|)^{2\alpha} |\widehat \psi(\zeta')|^2 d\zeta'\leq C_1 \delta^{1-2\alpha},\label{V4.2-6.15}
\end{align}
where $C_1$ is a constant independent of $\delta\in(0,T]$. From \eqref{V4.2-6.14}, and \eqref{V4.2-6.15} if follows that
\begin{align}
I_2 \leq C \delta^{1-2\alpha} \int_{\mathbb R} (1+|\tau-\tau_0|)^{2b} |\widehat f(\xi,\tau)|^2 d\tau,\label{V4.2-6.16}
\end{align}

and, from \eqref{V4.2-6.7}, \eqref{V4.2-6.12}, and \eqref{V4.2-6.16}, we conclude that
\begin{align*}
I_{\tau_0} \leq C \delta^{1-2\alpha} \int_{\mathbb R} (1+ |\tau - \tau_0|)^{2b} |\widehat f(\xi,\tau)|^2 d\tau.
\end{align*}

In particular, for every $\xi\in\mathbb R$, it follows that
\begin{align*}
\int_{\mathbb R} (1 + |\tau + \xi^5|)^{2b} |\widehat {\psi_\delta f} (\xi,\tau) |^2 d\tau \leq C \delta^{1-2\alpha} \int_{\mathbb R} (1+ |\tau +\xi^5|)^{2b} |\widehat f(\xi,\tau)|^2 d\tau,
\end{align*}

and, by multiplying by $\langle \xi \rangle^{2s}$, and integrating with respect to the variable $\xi$ over $\mathbb R$, we conclude that \eqref{V4.2-6.5} is valid.\\

Now, we continue by proving that
\begin{align}
\left[ \int_{-1}^1 \int_{\mathbb R} (1 + |\tau|)^{2\alpha} |\widehat{\psi_\delta f} (\xi,\tau)|^2 d\tau d\xi \right]^{\frac12} \leq C \delta^{\frac12-\alpha} \left[\int_{-1}^1 \int_{\mathbb R} (1 + |\tau|)^{2\alpha} |\widehat f(\xi,\tau)|^2 d\tau d\xi \right]^{\frac12}.\label{V4.2-6.17}
\end{align}

Taking into account that $\widehat{\psi_\delta f}(\xi,\tau) = (\widehat \psi_\delta \ast_\tau \widehat f(\xi,\cdot) )(\tau)$ for every $(\xi,\tau)\in [-1,1]\times \mathbb R$, we notice that

\begin{align}
\notag  \int_{-1}^1 \int_{\mathbb R} (1 + |\tau|)^{2\alpha} |\widehat{\psi_\delta f} (\xi,\tau)|^2 d\tau d\xi  \leq & C \int_{-1}^1 \int_{\mathbb R} \left( \int_{\mathbb R} (1 + |\zeta|)^\alpha |\widehat \psi_\delta (\zeta)| |\widehat f(\xi,\tau-\zeta)| d\zeta \right) d\tau d\xi\\
\notag& + C \int_{-1}^1 \int_{\mathbb R} \left( (1 + |\tau-\zeta|)^\alpha |\widehat \psi_\delta(\zeta)| |\widehat f(\xi,\tau-\zeta)| d\zeta \right)^2 d\tau d\xi\\
\notag\leq & C \int_{-1}^1 \| (1 + |\cdot|)^\alpha |\widehat \psi_\delta(\cdot)| \ast |\widehat f(\xi,\cdot)| \|^2_{L_\tau^2} d\xi \\
\notag&+C \int_{-1}^1 \| |\widehat \psi_\delta (\cdot)| \ast (1 + |\cdot|)^\alpha |\widehat f(\xi,\cdot)| \|_{L^2_\tau}^2 d\xi\\
\notag\leq & C \int_{-1}^1 \left( \int_{\mathbb R} (1 + |\tau|)^{2\alpha} |\widehat\psi_\delta(\tau)|^2 d\tau \right) \left( \int_{\mathbb R} |\widehat f(\xi,\tau)| d\tau \right)^2 d\xi\\
&+C \int_{-1}^1 \left( \int_{\mathbb R} |\widehat \psi_\delta(\tau)| d\tau \right)^2
 \left( \int_{\mathbb R} (1 + |\tau|)^{2\alpha} |\widehat f(\xi,\tau)|^2 d\tau \right) d\xi. \label{V4.2-6.18}
\end{align}

Therefore, by using \eqref{V4.2-6.15}, and \eqref{V4.2-6.13}, from \eqref{V4.2-6.18} it follows \eqref{V4.2-6.17}.\\

It is clear that the validity of \eqref{V4.2-6.5}, and \eqref{V4.2-6.17} implies \eqref{V4.2-6.4}. Therefore, Lemma \ref{V4.2-L6.1} is proved.
\end{proof}

Using the methodology introduced in \cite{KPV1991} it can be proved the following lemma, concerning some properties of the group $\{W_{\mathbb R}(t)\}_{t\in\mathbb R}$.

\begin{lemma}\label{dic2021-L1}
\begin{enumerate}
\item[(i)] $\| D_x^2 W_{\mathbb R}(t)u_0\|_{L_x^\infty L_t^2} \leq C \|u_0\|_{L_x^2}$.
\item[(ii)] $\|W_{\mathbb R}(t) u_0\|_{L_x^{12}L_t^{12}} \leq C \|u_0\|_{L^2_x}$.
\item[(iii)] $\|D_x^{\frac12} W_{\mathbb R}(t) u_0\|_{L_x^6 L_t^6} \leq C\|u_0\|_{L_x^2}$.
\end{enumerate}
\end{lemma}

Following the ideas of Kenig, Ponce, and Vega in Lemmas 2.4, 2.6, and 2.2 of \cite{KPV1993a}, it can be established the following lemma.

\begin{lemma}\label{V4.2-L6.2} Let us define $F_{\rho}(x,t)$, and $H_{\rho}(x,t)$ through the Fourier transform by
$$\widehat F_{\rho}(\xi,\tau):=\frac{f(\xi,\tau)}{(1 + |\tau + \xi^5|)^\rho},\quad \widehat H_{\rho}(\xi,\tau):=\frac{h(\xi,\tau)}{(1 + |\tau|)^\rho}.$$
Then
\begin{enumerate}
\item[(i)] If $\rho>\frac38$, and $0\leq \theta \leq \frac38$, it follows that
\begin{align}
\|D_x^\theta F_\rho\|_{L_{xt}^4} \leq C \|f\|_{L_{\xi\tau}^2}.\label{V4.2-6.19}
\end{align}
\item[(ii)] If $\rho>\frac\theta2$, with $\theta\in[0,1]$, it follows that
\begin{align}
\| D_x^{2\theta} F_\rho \|_{L_x^{\frac2{1-\theta}}L_t^2} \leq C \|f\|_{L_{\xi\tau}^2}. \label{V4.2-6.20}
\end{align}
\item[(iii)] If $\rho>\frac12$, it follows that
\begin{align}
\| H_\rho \|_{L_x^2 L_t^\infty} \leq C \|h\|_{L^2_{\xi\tau}}.\label{V4.2-6.21}
\end{align}
\end{enumerate}
\end{lemma}

\begin{lemma}\label{V4.2-L6.3} Let $b\in[\frac{13}{32},\frac12)$, and $T>0$. Then there exists $\alpha_0=\alpha_0(b)\in(\frac12,1)$ such that for every $\alpha\in(\frac12,\alpha_0)$, there exists $\epsilon=\epsilon(b,\alpha)>0$ such that for $u,v\in X^{0,b,\alpha}$, and $\delta\in(0,T]$
\begin{align}
\|\psi_\delta^2(\cdot_t)(uv)_x\|_{Y^{0,-b,\alpha}} \leq C\delta^\epsilon \|u\|_{X^{0,b,\alpha}} \|v\|_{X^{0,b,\alpha}}.\label{V4.2-6.22}
\end{align}
\end{lemma}

\begin{proof} Our proof follows the ideas of Colliander and Kenig in \cite{CK2002}, and Faminskii in \cite{F2005}.\\

Taking into account \eqref{V4.2-6.4} in Lemma \ref{V4.2-L6.1}, inequality \eqref{V4.2-6.22} follows if we manage to prove, for every $u,v\in X^{0,b,\alpha}$ with support contained in $\{(x,t):|t|\leq 2\delta\}$, that
\begin{align}
\|(uv)_x\|_{Y^{0,-b,\alpha}} \leq C \delta^\theta \|u\|_{X^{0,b,\alpha}} \|v\|_{X^{0,b,\alpha}}, \label{V4.2-6.23}
\end{align}
where $\theta$ is such that $1-2\alpha+\theta>0$.\\

Taking into account that $b<\frac12$, by Cauchy Schwarz inequality, it is possible to see that
$$\|\partial_x(uv)\|_{Y^{0,-b,\alpha}} \leq C \left[ \left(\int_{-\infty}^{+\infty}\int_{-\infty}^{+\infty} \frac{|[\partial_x(uv)]^\wedge(\xi,\tau)|^2}{\langle \tau+\xi^5 \rangle^{2b}} d\xi d\tau \right)^{\frac12} + \left( \int_{-\infty}^{+\infty} \int_{-1}^1 \frac{|[\partial_x(uv)]^\wedge(\xi,\tau)|^2}{\langle \tau \rangle^{2(1-\alpha)}} d\xi d\tau \right)^{\frac12} \right].$$

This way, by a duality argument, to obtain \eqref{V4.2-6.23} it is enough to prove that for every non negative function $w\in L^2(\mathbb R^2)$,
\begin{align}
\int_{\mathbb R^2_{\xi\tau}} \int_{\mathbb R^2_{\xi_1\tau_1}} |\xi| w(\xi,\tau) \gamma(\xi,\tau) \frac{U(\xi_1,\tau_1)}{\beta(\xi_1,\tau_1)} \frac{V(\xi-\xi_1,\tau-\tau_1)}{\beta(\xi-\xi_1,\tau-\tau_1)} d\xi_1 d\tau_1 d\xi d\tau\leq C \delta^\theta \|w\|_{L^2(\mathbb R^2)} \|U\|_{L^2(\mathbb R^2)} \|V\|_{L^2(\mathbb R^2)}, \label{V4.2-6.24}
\end{align}
where
\begin{align*}
\begin{array}{ll}
U(\xi,\tau):=\beta(\xi,\tau) |\widehat u(\xi,\tau)|,& V(\xi,\tau):=\beta(\xi,\tau)|\widehat v(\xi,\tau)|,\\
\beta(\xi,\tau):=\langle \tau+\xi^5\rangle^b + \langle \tau \rangle^\alpha \chi(\xi),& \gamma(\xi,\tau):=\langle \tau + \xi^5 \rangle^{-b} + \langle \tau \rangle^{\alpha-1} \chi(\xi),
\end{array}
\end{align*}
being $\chi$ the characteristic function of the interval $(-1,1)$.\\

By considerations of symmetry we can assume from now on that $|\xi-\xi_1|\geq |\xi_1|$, and consider the following integration regions:
\begin{align*}
A&:=\{(\xi,\tau,\xi_1,\tau_1):|\xi-\xi_1|\geq |\xi_1|,|\xi|<1\},\\
B&:=\{(\xi,\tau,\xi_1,\tau_1):|\xi-\xi_1|\geq |\xi_1|,|\xi|\geq 1\}.
\end{align*}

Besides, we will divide region $B$ into
\begin{align*}
B_1&:=\{(\xi,\tau,\xi_1,\tau_1):|\xi-\xi_1|\geq |\xi_1|,|\xi|\geq 1, |\xi_1|\geq 1\},\\
B_2&:=\{(\xi,\tau,\xi_1,\tau_1):|\xi-\xi_1|\geq |\xi_1|,|\xi|\geq 1, |\xi_1|< 1\}.
\end{align*}

We will further subdivide region $B_1$ into
\begin{align*}
B_{1a}&:=\{(\xi,\tau,\xi_1,\tau_1):|\xi-\xi_1|\geq |\xi_1|,|\xi|\geq 1, |\xi_1|\geq 1,|\tau+\xi^5|=M\},\\
B_{1b}&:=\{(\xi,\tau,\xi_1,\tau_1):|\xi-\xi_1|\geq |\xi_1|,|\xi|\geq 1, |\xi_1|\geq 1,|\tau_1+\xi_1^5|=M\},\\
B_{1c}&:=\{(\xi,\tau,\xi_1,\tau_1):|\xi-\xi_1|\geq |\xi_1|,|\xi|\geq 1, |\xi_1|\geq 1,|(\tau-\tau_1)+(\xi-\xi_1)^5|=M\},
\end{align*}
where $M:=\max\{|\tau+\xi^5|,|\tau_1+\xi_1^5|,|(\tau-\tau_1)+(\xi-\xi_1^5)|\}$.\\

Let us begin by considering region $A$: we define $D$, $G_1$, and $G_2$ through the Fourier transform by
$$\widehat D(\xi,\tau):=w(\xi,\tau),\quad \widehat G_1(\xi,\tau):=\frac{U(\xi,\tau)}{\langle \tau + \xi^5\rangle^b},\quad \text{and}\quad \widehat G_2(\xi,\tau):=\frac{V(\xi,\tau)}{\langle \tau + \xi^5\rangle^b}.$$

From Lemma \ref{V4.2-L6.2} (i), with $\theta=0$, and $\rho:=b>\frac38$, we have that
$$\|G_2\|_{L^4_{xt}}\leq C \|V\|_{L^2_{\xi\tau}}.$$

Therefore
\begin{align}
\notag \int_{A} |\xi| w(\xi,\tau) \gamma(\xi,\tau) \frac{U(\xi_1,\tau_1)}{\beta(\xi_1,\tau_1)} \frac{V(\xi-\xi_1,\tau-\tau_1)}{\beta(\xi-\xi_1,\tau-\tau_1)} d\xi_1 d\tau_1 d\xi d\tau&\leq C \|w\|_{L^2_{\xi\tau}} \|G_1\|_{L^4} \|G_2\|_{L^4}\\
&\leq C \|w\|_{L^2} \|G_1\|_{L^4}\|V\|_{L^2}.\label{V4.2-6.27}
\end{align}

Now we consider region $B$. It can be seen from the definition of $M$ that
\begin{align}
\frac52 |\xi \xi_1 (\xi-\xi_1)| (\xi^2+\xi_1^2+(\xi-\xi_1)^2)\leq 3M.\label{V4.2-6.28}
\end{align}

Since $|\xi|\geq 1$, from \eqref{V4.2-6.28},
\begin{align}
|\xi| |\xi_1| |\xi-\xi_1|\leq \frac65 M. \label{V4.2-6.29}
\end{align}

On the other hand, since $|\xi-\xi_1|\geq \frac{|\xi|}2$,
\begin{align}
|\xi| |\xi_1| |\xi-\xi_1|\geq\frac12 |\xi|^2 |\xi_1|, \label{V4.2-6.30}
\end{align}

which implies, for $|\xi_1|\geq1$, that
\begin{align}
|\xi| |\xi_1| |\xi-\xi_1|\geq\frac12 |\xi|^2. \label{V4.2-6.31}
\end{align}

Let us observe that for $(\xi,\tau,\xi_1,\tau_1)\in B_{1a}$, taking into account \eqref{V4.2-6.29}, and \eqref{V4.2-6.31}, we have that
\begin{align*}
|\xi| \gamma(\xi,\tau) & =\frac{|\xi|}{M^b} \leq \frac{C|\xi|}{|\xi|^b |\xi_1|^b |\xi - \xi_1|^b} = \frac {C|\xi|^{1-b}}{|\xi_1|^b |\xi-\xi_1|^b}\leq \frac{C |\xi_1|^{1-b} |\xi-\xi_1|^{1-b}}{|\xi_1|^b |\xi-\xi_1|^b}\\
&=C|\xi_1|^{1-2b} |\xi-\xi_1|^{1-2b} \leq C |\xi-\xi_1|^{2-4b}.
\end{align*}

This way
\begin{align}
\notag\int_{B_{1a}} |\xi| w(\xi,\tau) \gamma(\xi,\tau) \frac{U(\xi_1,\tau_1)}{\beta(\xi_1,\tau_1)} \frac{V(\xi-\xi_1,\tau-\tau-1)}{\beta(\xi-\xi_1,\tau-\tau_1)} d\xi_1 d\tau_1 d\xi d\tau & \leq C \|D\|_{L_{xt}^2} \|G_1 D_x^{2-4b} G_2\|_{L^2_{xt}}\\
&\leq C \|w\|_{L^2_{\xi\tau}} \|G_1\|_{L^4_{xt}} \|D_x^{2-4b} G_2\|_{L_{xt}^4}. \label{V4.2-6.32}
\end{align}

Taking into account that $2-4b\leq \frac38$ (i.e. $\frac{13}{32}\leq b$), from Lemma \ref{V4.2-L6.2} (i), with $\theta:=2-4b$, and $\rho:=b>\frac38$, it follows that
\begin{align}
\|D_x^{2-4b}G_2\|_{L^2_{xt}}\leq C \|V\|_{L^2_{\xi\tau}}.\label{V4.2-6.33}
\end{align}

From \eqref{V4.2-6.32}, and \eqref{V4.2-6.33}, we conclude that
\begin{align}
\int_{B_{1a}} |\xi| w(\xi,\tau) \gamma(\xi,\tau) \frac{U(\xi_1,\tau_1)}{\beta(\xi_1,\tau_1)} \frac{V(\xi-\xi_1,\tau-\tau-1)}{\beta(\xi-\xi_1,\tau-\tau_1)} d\xi_1 d\tau_1 d\xi d\tau & \leq C \|w\|_{L^2} \|G_1\|_{L^4} \|V\|_{L^2}. \label{V4.2-6.34}
\end{align}

In region $B_{1b}$ we have that
$$\frac{|\xi|}{|\tau_1+\xi_1^5|^b}\leq C\frac{|\xi|}{|\xi|^b |\xi_1|^b |\xi-\xi_1|^b}\leq C |\xi|^{1-2b}.$$

Hence
\begin{align}
\int_{B_{1b}} |\xi| w(\xi,\tau) \gamma(\xi,\tau) \frac{U(\xi_1,\tau_1)}{\beta(\xi_1,\tau_1)} \frac{V(\xi-\xi_1,\tau-\tau-1)}{\beta(\xi-\xi_1,\tau-\tau_1)} d\xi_1 d\tau_1 d\xi d\tau & \leq C \|U\|_{L^2_{\xi\tau}} \left\|D_x^{1-2b} \mathcal F^{-1} \left(\frac{w(\xi,\tau)}{\langle \tau + \xi^5\rangle^b} \right) \right\|_{L^4_{xt}} \|H_2\|_{L^4_{xt}}, \label{V4.2-6.35}
\end{align}

where
$$H_2(x,t):=\mathcal F^{-1} \left(\frac{V({\scriptstyle-}\cdot_\xi, {\scriptstyle-}\cdot_\tau)}{\langle \cdot_\tau + {\cdot_\xi}^5\rangle^b} \right).$$

By using \eqref{V4.2-6.19} from Lemma \ref{V4.2-L6.2} (i), with $\theta=1-2b$, and $b>\frac38$, we have that
\begin{align}
\|D_x^{1-2b} \mathcal F^{-1} \left( \frac{w(\xi,\tau)}{\langle \tau+\xi^5\rangle^b} \right)\|_{L^4_{xt}}\leq C \|w\|_{L^2_{\xi\tau}}.\label{V4.2-6.36}
\end{align}

Since $\|H_2\|_{L^4_{xt}}= \|G_2\|_{L^4_{xt}}$, from \eqref{V4.2-6.35}, and \eqref{V4.2-6.36}, we conclude that
\begin{align}
\int_{B_{1b}} |\xi| w(\xi,\tau) \gamma(\xi,\tau) \frac{U(\xi_1,\tau_1)}{\beta(\xi_1,\tau_1)} \frac{V(\xi-\xi_1,\tau-\tau-1)}{\beta(\xi-\xi_1,\tau-\tau_1)} d\xi_1 d\tau_1 d\xi d\tau & \leq C \|w\|_{L^2} \|U\|_{L^2_{\xi\tau}} \|G_2\|_{L^4}.\label{V4.2-6.37}
\end{align}

Now, from \eqref{V4.2-6.29}, it is clear that in region $B_{1c}$, the following inequality is valid:
$$|\xi| |\xi_1| |\xi-\xi_1| \leq \frac65 |(\tau-\tau_1)+ (\xi-\xi_1)^5|,$$
which implies that
$$\frac{|\xi|}{|(\tau-\tau_1)+(\xi-\xi_1)^5|^b}\leq C\frac{|\xi|}{|\xi|^b |\xi_1|^b |\xi-\xi_1|^b\|}\leq C |\xi|^{1-2b}.$$
Therefore
\begin{align}
\int_{B_{1c}} |\xi| w(\xi,\tau) \gamma(\xi,\tau) \frac{U(\xi_1,\tau_1)}{\beta(\xi_1,\tau_1)} \frac{V(\xi-\xi_1,\tau-\tau-1)}{\beta(\xi-\xi_1,\tau-\tau_1)} d\xi_1 d\tau_1 d\xi d\tau & \leq C \|w\|_{L^2} \|G_1\|_{L^4} \|V\|_{L^2}.\label{V4.2-6.38}
\end{align}

In region $B_2$, we take into account that
$$|\xi-\xi_1|\geq \frac{|\xi|}2\geq \frac12.$$

Hence, by using the Cauchy-Schwarz inequality, it is possible to see that
\begin{align}
\notag &\int_{B_{2}} |\xi| w(\xi,\tau) \gamma(\xi,\tau) \frac{U(\xi_1,\tau_1)}{\beta(\xi_1,\tau_1)} \frac{V(\xi-\xi_1,\tau-\tau_1)}{\beta(\xi-\xi_1,\tau-\tau_1)} d\xi_1 d\tau_1 d\xi d\tau \\
\notag & \leq C \int_{B_2} \frac{|\xi| w(\xi,\tau)}{\langle \tau + \xi^5 \rangle^b} \frac{U(\xi_1,\tau_1)}{\langle \tau_1 \rangle^\alpha} \frac{|\xi-\xi_1| V(\xi-\xi_1,\tau-\tau_1)}{\langle (\tau-\tau_1)+(\xi-\xi_1)^5\rangle^b} d\xi_1 d\tau_1 d\xi d\tau \\
 &\leq  C \|H_\alpha\|_{L^2_x L^\infty_t} \|D_x G\|_{L^4_x L^2_t} \|D_x G_2\|_{L^4_x L^2_t},\label{V4.2-6.40}
\end{align}

where
$$\widehat G(\xi,\tau):=\frac{w(\xi,\tau)}{\langle \tau+\xi^5\rangle^b},\quad \text{and}\quad \widehat H_\alpha (\xi,\tau):=\frac{U(\xi,\tau)}{\langle \tau \rangle^\alpha}.$$

From \eqref{V4.2-6.21} in Lemma \ref{V4.2-L6.2} (iii), with $\rho:=\alpha>\frac12$, it follows that
\begin{align}
\|H_\alpha\|_{L^2_x L^\infty_t} \leq C \|U\|_{L^2}, \label{V4.2-6.41}
\end{align}

and from \eqref{V4.2-6.20} in Lemma \ref{V4.2-L6.2} (ii), with $\theta:=\frac12$, $\rho:=b>\frac\theta2=\frac14$, we have that
\begin{align}
\|D_x G\|_{L^4_x L^2_t}\leq C \|w\|_{L^2}. \label{V4.2-6.42}
\end{align}

This way, from \eqref{V4.2-6.40}, \eqref{V4.2-6.41}, and \eqref{V4.2-6.42}, we obtain
\begin{align}
\int_{B_{2}} |\xi| w(\xi,\tau) \gamma(\xi,\tau) \frac{U(\xi_1,\tau_1)}{\beta(\xi_1,\tau_1)} \frac{V(\xi-\xi_1,\tau-\tau-1)}{\beta(\xi-\xi_1,\tau-\tau_1)} d\xi_1 d\tau_1 d\xi d\tau & \leq C \|w\|_{L^2} \|U\|_{L^2} \|D_x G_2\|_{L^4_x L^2_t}.\label{V4.2-6.43}
\end{align}

From \eqref{V4.2-6.27}, \eqref{V4.2-6.34}, \eqref{V4.2-6.37}, \eqref{V4.2-6.38}, and \eqref{V4.2-6.43}, we conclude that
\begin{align}
\notag \int_{\{(\xi,\tau,\xi_1,\tau_1):|\xi-\xi_1|\geq |\xi_1|\}} |\xi| w(\xi,\tau) \gamma(\xi,\tau) \frac{U(\xi_1,\tau_1)}{\beta(\xi_1,\tau_1)} \frac{V(\xi-\xi_1,\tau-\tau-1)}{\beta(\xi-\xi_1,\tau-\tau_1)} d\xi_1 d\tau_1 d\xi d\tau\\
 \leq C \|w\|_{L^2} (\|G_1\|_{L^4} \|V\|_{L^2} + \|U\|_{L_2} \|G_2\|_{L^4}+\|U\|_{L^2} \|D_x G_2\|_{L^4_x L^2_t}),\label{V4.2-6.44}
\end{align}

and, by symmetry,
\begin{align}
\notag \int_{\{(\xi,\tau,\xi_1,\tau_1): |\xi_1|>|\xi-\xi_1|\}} |\xi| w(\xi,\tau) \gamma(\xi,\tau) \frac{U(\xi_1,\tau_1)}{\beta(\xi_1,\tau_1)} \frac{V(\xi-\xi_1,\tau-\tau-1)}{\beta(\xi-\xi_1,\tau-\tau_1)} d\xi_1 d\tau_1 d\xi d\tau\\
 \leq C \|w\|_{L^2} (\|U\|_{L^2}\|G_2\|_{L^4} + \|G_1\|_{L^4} \|V\|_{L_2}\| + \|D_x G_1\|_{L^4_x L^2_t} \|V\|_{L^2}).\label{V4.2-6.45}
\end{align}

From \eqref{V4.2-6.44}, and \eqref{V4.2-6.45}, for $b\in[\frac{13}{32},\frac12)$, and $\alpha>\frac12$, we have the estimative
\begin{align}
\notag \int_{\mathbb R^4} |\xi| w(\xi,\tau) \gamma(\xi,\tau) \frac{U(\xi_1,\tau_1)}{\beta(\xi_1,\tau_1)} \frac{V(\xi-\xi_1,\tau-\tau-1)}{\beta(\xi-\xi_1,\tau-\tau_1)} d\xi_1 d\tau_1 d\xi d\tau\\
 \leq C \|w\|_{L^2} [\|U\|_{L^2}(\|G_2\|_{L^4} + \|D_x G_2\|_{L_x^4 L^2_t}) + \|V\|_{L_2}\| ( \|G_1\|_{L^4} + \|D_x G_1\|_{L^4_x L^2_t}).\label{V4.2-6.46}
\end{align}

If we manage to prove that
\begin{align}
\|G_1\|_{L^4} + \|D_x G_1\|_{L_x^4 L_t^2} \leq C \delta^\theta \|U\|_{L^2}, \label{V4.2-6.47}
\end{align}

then we could conclude that \eqref{V4.2-6.24} is valid.\\

From Lema \ref{V4.2-L6.2} (ii) with $\theta:=\frac12$, $\rho:=b_1>\frac\theta2=\frac14$, it follows that
$$\|D_x G_1\|_{L_x^4 L_t^2} \leq C \| U(\xi,\tau) \langle \tau + \xi^5 \rangle^{b_1-b}\|_{L^2_{\xi\tau}}.$$

Due to \eqref{V4.2-6.19} in Lemma \eqref{V4.2-L6.2} (i), with $\theta:=0$, $\rho:=b_1>\frac38$, it can be concluded that
$$\|G_1\|_{L^4}\leq C \|U\langle \tau + \xi^5\rangle^{b_1-b}\|_{L^2}.$$

Hence, by Hölder's inequality with exponents $p$, and $\frac p{p-1}$ $(p>1)$,
\begin{align*}
\|G_1\|_{L^4} + \|D_x G_1\|_{L_x^4 L_t^2} \leq & C \left( \int_{\mathbb R^2} |U(\xi,\tau)|^2 \langle \tau + \xi^5\rangle^{2(b_1-b)} d\xi d\tau \right)^{\frac12}\\
\leq & C \|U\|^{\frac 1p}_{L^2_{\xi\tau}} \|\beta(\xi,\tau) |\widehat u| \langle \tau + \xi^5\rangle^{\frac{(b_1-b)p}{p-1}} \|^{\frac{p-1}{p}}_{L^2_{\xi\tau}}\\
\leq & C \|U\|_{L^2_{\xi\tau}}^{\frac1p} \left( \| |\widehat u| \langle \tau + \xi^5 \rangle^{\alpha - \frac{p(b-b_1)}{p-1}}\|_{L^2_{\xi\tau}} \right)^{\frac{p-1}{p}}.
\end{align*}

We choose $p>1$ such that $\alpha-\frac{p(b-b_1)}{p-1}<0$. Therefore
\begin{align}
\|G_1 \|_{L^4} + \|D_x G_1\|_{L^4_x} \leq C \|U\|_{L^2_{\xi\tau}}^{\frac1p} \|\widehat u\|_{L^2_{\xi\tau}}^{\frac{p-1}p}=C \|U\|_{L^2_{\xi\tau}}^{\frac1p} \|u\|_{L^2_{xt}}^{\frac{p-1}p}.\label{V4.2-6.48}
\end{align}

Since $u\equiv 0$ for $|t|>2\delta$, by Hölder's inequality with exponents $\frac1{1-2b}$, and $\frac1{2b}$, $b<\frac12$, and taking into account that $H^b(\mathbb R_t) \hookrightarrow L^{\frac2{1-2b}}(\mathbb R_t)$, we have that
\begin{align}
 \notag\|u\|_{L^2_{xt}} &=\left( \int_{\mathbb R_x} \int_{-2\delta}^{2\delta} |[W_{\mathbb R}(-t) u(\cdot_x,t)](x)|^2 dt dx\right)^{\frac12} \leq C \delta^b \|[W_{\mathbb R} ({\scriptstyle-}\cdot_t) u(\cdot_t)](\cdot_x)\|_{L^2_x L^{\frac2{1-2b}}_t}\\
\notag&\leq C \delta^b \left( \int_{\mathbb R_x} \|[W_{\mathbb R} ({\scriptstyle-}\cdot_t)u(\cdot_t)](x)\|^2_{H^b(\mathbb R_t)} dx \right)^{\frac12}\leq C \delta^b \left(\int_{\mathbb R_\xi} \int_{\mathbb R_\tau} \langle \tau \rangle^{2b} |\widehat u(\xi,-\tau-\xi^5)|^2 d\tau d\xi \right)^{\frac12}\\
&\leq C \delta^b \|U\|_{L^2_{\xi\tau}}.\label{V4.2-6.50}
\end{align}

From \eqref{V4.2-6.48}, and \eqref{V4.2-6.50} we obtain that
\begin{align}
\|G_1\|_{L^4} + \|D_x G_1\|_{L^4_x L^2_t} \leq C \|U\|^{\frac1p}_{L^2_{\xi\tau}} \delta^{\frac{b(p-1)}p} \|U\|^{\frac{p-1}p}_{L^2_{\xi\tau}} \leq C \delta^{\frac{b(p-1)}p} \|U\|_{L^2_{\xi\tau}},\label{V4.2-6.51}
\end{align}
which proves inequality \eqref{V4.2-6.24} with $\theta:=\frac{b(p-1)}p$.\\

Since
$$\lim_{\alpha\to{\frac12}^+} (1-2\alpha+\theta)=\theta>0,$$

there exists $\alpha_0>\frac12$ such that for $\alpha\in(\frac12,\alpha_0)$, $1-2\alpha+\theta>\frac\theta2$. This way, choosing $\epsilon:=\frac\theta2>0$, we can affirm that for $\delta\in(0,T]$, and $\alpha\in(\frac12,\alpha_0)$,
\begin{align*}
\|\psi_\delta^2 (t)(uv)_x\|_{Y^{0,-b,\alpha}}\leq C \delta^{1-2\alpha+\theta} \|u\|_{X^{0,b,\alpha}} \|v\|_{X^{0,b,\alpha}}\leq C\delta^\epsilon \|u\|_{X^{0,b,\alpha}} \|v\|_{X^{0,b,\alpha}}.
\end{align*}
Lemma \ref{V4.2-L6.3} is proved.
\end{proof}

In our aim to prove the uniqueness of the generalized solution for IBVP \eqref{maineq}, it will suffice to demonstrate that the linear and homogeneous IBVP
\begin{align}
\left. \begin{array}{rlr}
u_t+\partial_x^5 u+ \partial_x [h(x,t) u]&\hspace{-2mm}=0,&\quad x>0,\; 0<t<T,\\
u(x,0)&\hspace{-2mm}=0,&\quad x\geq 0\\
u(0,t)=u_x(0,t)=u_{xx}(0,t)&\hspace{-2mm}=0,&\quad t\in[0,T],
\end{array} \right\}\label{V4.2-6.51}
\end{align}
where $h\in X^{0,b,\alpha}$, only has the generalized solution identically zero.\\

In the study of the uniqueness of the generalized solution for IBVP \eqref{V4.2-6.51}, the following IBVP in $\Pi^+_T:=(0,+\infty)\times(0,T)$ plays a special role:

\begin{align}
\left. \begin{array}{rlr}
u_t+\partial_x^5 u&\hspace{-2mm}=f(x,t),&\quad (x,t)\in \Pi^+_T,\\
u(x,0)&\hspace{-2mm}=g(x),&\quad x\geq 0,\\
u(0,t)=h_1(t),\, \partial_x u(0,t)&\hspace{-2mm}=h_2(t),\, \partial_x^2 u(0,t)=h_3(t),&\quad 0\leq t\leq T,
\end{array} \right\}\label{V4.2-6.52}
\end{align}

where $f\in (C^{\infty}([0,T]; S_+))'$, $g\in (S_+)'$, and $h_{j+1}\in L^1(0,T)$, for $j=0,1,2$. Here $S_+:=\{f|_{[0,+\infty)}:f\in S(\mathbb R)\}$ is the set of restrictions to $[0,+\infty)$ of the Schwartz functions in $\mathbb R$.\\

Let us recall the concept of generalized solution for IBVP \eqref{V4.2-6.52}:

A distribution $u\in (C^{\infty}([0,T]; S_+))'$ is called a generalized solution of IBVP \eqref{V4.2-6.52} if, for any $\varphi\in C^{\infty}([0,T];S_+)$ such that
$$\varphi\big|_{t=T}=0,\quad \varphi\big|_{x=0}=0,\quad \varphi_x\big|_{x=0}=0,$$
it satisfies
\begin{align*}
\langle u,\varphi_t + \partial_x^5 \varphi \rangle + \langle f,\varphi \rangle + \langle g,\varphi(x,0) \rangle + \int_0^T h_1(t) \partial_x^4 \varphi(0,t) dt - \int_0^T h_2(t) \partial_x^3 \varphi(0,t) dt + \int_0^T h_3(t) \partial_x^2 \varphi(0,t) dt =0.
\end{align*}

According to Holgrem's theory, the uniqueness of generalized solutions for IBVP \eqref{V4.2-6.52} follows from the solvability of the dual problem, which consists of proving that given $f\in C_0^{\infty}(\Pi^+_T)$, there exists $u\in C^{\infty}([0,T];S_+)$, a classical solution to the following IBVP:

\begin{align}
\left. \begin{array}{rlr}
u_t+\partial_x^5 u&\hspace{-2mm}=f,&\quad (x,t)\in \Pi^+_T,\\
u(x,T)&\hspace{-2mm}=0,&\quad x\geq 0,\\
u(0,t)=0,\quad u_x(0,t)&\hspace{-2mm}=0, &\quad t\in[0, T].
\end{array} \right\}\label{V4.2-6.53}
\end{align}

Equivalently, the dual problem of \eqref{V4.2-6.52}, after a change of variables ($x'=-x$, $t'=T-t$), consists of proving that for a given $f\in C_0^{\infty}(\Pi^-_T)$, where $\Pi_T^-:=(-\infty,0)\times(0,T)$, there exists a classical solution $u\in C^{\infty}([0,T];S_-)$, to the following IBVP:

\begin{align}
\left. \begin{array}{rlr}
u_t+\partial_x^5 u&\hspace{-2mm}=f,&\quad (x,t)\in \Pi^-_T,\\
u(x,0)&\hspace{-2mm}=0,&\quad x\leq 0,\\
u(0,t)=0,\quad u_x(0,t)&\hspace{-2mm}=0, &\quad t\in[0, T].
\end{array} \right\}\label{V4.2-6.54}
\end{align}

Here, $S_-:=\left\{g|_{(-\infty,0]}: g\in S(\mathbb R) \right\}$ is the set of restrictions to $(-\infty,0]$ of the Schwartz functions in $\mathbb R$.\\

Consequently, to obtain the uniqueness result for the generalized solution of IBVP \eqref{V4.2-6.52}, following the procedure outlined by Faminskii in \cite{F1999a}, it is sufficient to prove the following lemma.

\begin{lemma}\label{V4.2-L6.4} Let $f\in C_0^\infty(\Pi_T^-)$. Then there exists $u\in C^\infty([0,T];S_-)$, a classical solution of IBVP \eqref{V4.2-6.54}.
\end{lemma}

\begin{proof} Let us define, for $(x,t)\in\mathbb R^2$,
$$u_1(x,t):=\int_0^t [W_{\mathbb R}(t-t')F(\cdot,t')](x) dt',$$
where $F\in C_0^\infty(\mathbb R^2)$ is the extension of $f$ to $\mathbb R^2$ which vanishes outside $\Pi_T^-$.\\

Let us observe that
$$u_1(x,t)=C \int_0^t \left(\int_{\mathbb R_\xi} e^{ix\xi} e^{-i(t-t')\xi^5} [F(\cdot_x,t')]^{\wedge_x}(\xi) d\xi \right) dt',\quad (x,t)\in\mathbb R^2.$$

Taking derivative under the integral sign, it can be seen that
$$\partial_t u_1(x,t) + \partial_x^5 u_1(x,t) = F(x,t)\quad \text{for }(x,t)\in\mathbb R^2.$$

It can be seen also that $u_1\in C^\infty([0,T];H^\infty(\mathbb R_x^-))$. Besides, it is clear that $u_1(x,0)=0$ for each $x\in\mathbb R$.\\

Let us define $h_1(t):=u_1(0,t)$, and $h_2(t):=\partial_x u_1(0,t)$, for $t\in\mathbb R$. It can be proved that $h_j\in H^\infty(\mathbb R_t)$, and that $\supp h_j\subset [0,+\infty)$, $j=1,2$.\\

Let us consider the IBVP
\begin{align}
\left. \begin{array}{rlr}
u_t-\partial_x^5 u&\hspace{-2mm}=0,&\quad x>0,\quad t>0,\\
u(x,0)&\hspace{-2mm}=0,&\quad x>0,\\
u(0,t)=-h_1(t),\quad u_x(0,t)&\hspace{-2mm}=h_2(t), &\quad t>0.
\end{array} \right\}\label{V4.2-6.55}
\end{align}

By using the Laplace transform method of Section \ref{ILBVP} we can see that there exists $v:[0,+\infty)\times[0,+\infty)\to\mathbb R$, $(x,t)\mapsto v(x,t)$, such that $v\in C^{\infty}(\overline{\mathbb R_t^+};H^\infty(\overline{\mathbb R_x^+}))$, and such that $v$ is a classical solution of IBVP \eqref{V4.2-6.55}.

Let $u:\overline{\Pi_T^-}\to\mathbb R$ be defined by $u(x,t):=u_1(x,t)+v(-x,t)$, for $(x,t)\in \overline{\Pi_T^-}$. Then, for $(x,t)\in \Pi_T^-$, we have that
\begin{align*}
\partial_t u(x,t) + \partial_x^5 u(x,t) = \partial_t u_1(x,t) + \partial_x^5 u_1(x,t) + \partial_t v(-x,t) - \partial_x^5 v(-x,t) = f(x,t).
\end{align*}

On the other hand,
\begin{align*}
u(x,0)=u_1(x,0) + v(-x,0)=&0\quad x\leq 0,\\
u(0,t)=u_1(0,t) + v(0,t) = &0\quad t\in[0,T],\\
u_x(0,t)=\partial_xu_1(0,t) - \frac{\partial v}{\partial x}(-x,t)\big|_{x=0}=&0\quad t\in[0,T].
\end{align*}

We have proved that $u\in C^{\infty}([0,T];H^\infty(\mathbb R_x^-))$ is a classical solution of IBVP \eqref{V4.2-6.54}. To conclude the proof of Lemma \ref{V4.2-L6.4}, it can be shown as in \cite{F1999a}, proceeding by induction on $m\in\mathbb N \cup \{0\}$, that $|x|^m\partial_x^k u\in L^\infty([0,T];L^2(\mathbb R_x^-))$, and in this manner $u\in C^{\infty}([0,T];S_-)$.
\end{proof}

\begin{remark}\label{V4.2-R6.1} From Lemma \ref{V4.2-L6.4}, and from Holgrem's theory we can conclude that IBVP \eqref{V4.2-6.54} has a unique generalized solution.
\end{remark}

\begin{lemma}\label{V4.2-L6.5} Let $b\in [\tfrac{13}{32},\tfrac12]$, and $\alpha\in(\tfrac12,\alpha_0(b))$, where $\alpha_0(b)\in(\tfrac12,1)$ has properties of Lemma \ref{V4.2-L6.3}, and let us assume that $h\in X^{0,b,\alpha}$. Then $u(x,t)\equiv 0$ is the only generalized solution of IBVP \eqref{V4.2-6.51} in $X^{0,b,\alpha}(\Pi_T^+)$.
\end{lemma}

\begin{proof} Let $w\in X^{0,b,\alpha}(\Pi_T^+)$ a generalized solution of IBVP \eqref{V4.2-6.51}. For $\delta\in (0,T]$, we consider the restriction of $w$ to $\Pi^+_\delta$. Then $w\big|_{\Pi_\delta^+}\in X^{0,b,\alpha}(\Pi_\delta^+)$. Let $\tilde w$ be an extension of $w\big|_{\Pi_\delta^+}$ to $\mathbb R^2$ such that
\begin{align}
\|\tilde w\|_{X^{0,b,\alpha}} \leq 2\|w\big|_{\Pi_\delta^+}\|_{X^{0,b,\alpha}(\Pi^+_\delta)}.\label{V4.2-6.61}
\end{align}

Let us consider the linear inhomogeneous IBVP
\begin{align}
\left. \begin{array}{rlr}
u_t+\partial_x^5 u + \Psi_\delta^2(t)(h(\tilde w))_x&\hspace{-2mm}=0,&\quad (x,t)\in \Pi_\delta^+,\\
u(x,0)&\hspace{-2mm}=0,&\quad x\geq0,\\
u(0,t)=0,\quad u_x(0,t)=0,\quad u_{xx}(0,t)&\hspace{-2mm}=0, &\quad t\in[0,\delta].
\end{array} \right\}\label{V4.2-6.62}
\end{align}

Then $w\big|_{\Pi_\delta^+}$ is a generalized solution in $\Pi_\delta^+$ of IBVP \eqref{V4.2-6.62}.\\

On the other hand, let us observe that if we define, for $x\geq0$, and $t\in[0,\delta]$,
\begin{align*}
v(x,t):=\int_0^t \left\{W_{\mathbb R}(t-t') [-\Psi_\delta^2(t')(h(\cdot_x,t')\tilde w(\cdot_x,t'))] \right\}(x) dt' + W_0^t(0,-p_1,-p_2,-p_3)(x,t),
\end{align*}

where
$$p_{j+1}(t)=\partial_x^j \int_0^t \left\{W_{\mathbb R}(t-t')[-\Psi_\delta^2(t')(h(\cdot_x,t')\tilde w(\cdot_x,t'))] \right\}(0) dt'$$
for $j=0,1,2$, then $v$ is also a generalized solution in $\Pi_\delta^+$ of IBVP \eqref{V4.2-6.62}.\\

According with Remark \ref{V4.2-R6.1}, IBVP \eqref{V4.2-6.62} has a unique generalized solution. Hence $w\big|_{\Pi_\delta^+}=v$.\\

From Lemma \ref{V4.2-L6.3}, $\Psi_{\delta}^2(\cdot_t)(h\tilde w)_x\in Y^{0,-b,\alpha}$, and
\begin{align*}
\|\Psi_\delta^2(\cdot_t)(h\tilde w)_x\|_{Y^{0,-b,\alpha}}\leq C \delta^\epsilon\|h\|_{X^{0,b,\alpha}} \|\tilde w\|_{X^{0,b,\alpha}},
\end{align*}

and, taking into account \eqref{V4.2-6.61}, it follows that
\begin{align}
\|\Psi_\delta^2(\cdot_t)(h\tilde w)_x\|_{Y^{0,-b,\alpha}} \leq C \delta^\epsilon\|h\|_{X^{0,b,\alpha}} \|w\big|_{\Pi_\delta^+}\|_{X^{0,b,\alpha}(\Pi_\delta^+)}.\label{V4.2-6.63}
\end{align}

Without loss of generality, let us assume that $0<T\leq \frac12$, and let $\eta(\cdot_t)\in C_0^\infty(\mathbb R_t)$ such that $\supp \eta\subset [-1,1]$, and $\eta\equiv 1$ in $[-\tfrac12,\tfrac12]$.\\

For $(x,t)\in\mathbb R^2$, let us define
\begin{align*}
\tilde v(x,t):=\eta(t)\int_0^t \left\{W_{\mathbb R}(t-t')[-\Psi_\delta^2(t')(h(\cdot_x,t') \tilde w(\cdot_x,t'))] \right\}(x) dt' + \eta(t) W_0^t(0,-\tilde p_1,-\tilde p_2, -\tilde p_3)(x,t) dt',
\end{align*}

where
\begin{align*}
\tilde p_{j+1}(t) = \eta(t) \partial_x^j \int_0^t \left\{ W_{\mathbb R}(t-t')[-\Psi_\delta^2(t')(h(\cdot_x,t')\tilde w(\cdot_x,t'))] \right\}(0) dt',
\end{align*}

for $j=0,1,2$. $\tilde v$ is an extension of $v$, and, by using Lemmas \ref{V3-L4}, \ref{V2-L3.3} (ii), and \ref{V3-L4.6} (i), we have that
\begin{align}
\notag\|w\big|_{\Pi^+_\delta}\|_{X^{0,b,\alpha}(\Pi_\delta^+)} = \|v\|_{X^{0,b,\alpha}(\Pi^+_\delta)}\leq \|\tilde v\|_{X^{0,b,\alpha}} \leq & C \|\Psi_\delta^2(\cdot_t) (h\tilde w)(\cdot_x,\cdot_t)\|_{Y^{0,-b,\alpha}} + C \left(\sum_{j=0}^2 \|\chi_{(0,+\infty)} \tilde p_{j+1}\|_{H^{\frac{2-j}5}(\mathbb R_t)} \right)\\
\notag\leq & C \|\Psi_\delta^2(\cdot_t) (h\tilde w)(\cdot_x,\cdot_t)\|_{Y^{0,-b,\alpha}} + C \|\Psi_\delta^2(\cdot_t)(h\tilde w)_x (\cdot_x,\cdot_t)\|_{X^{0,-b}}\\
\leq &C \|\Psi_\delta^2(\cdot_t) (h\tilde w)(\cdot_x,\cdot_t)\|_{Y^{0,-b,\alpha}}.\label{V4.2-6.64}
\end{align}

From \eqref{V4.2-6.63}, and \eqref{V4.2-6.64}, it follows that
\begin{align*}
\|w\big|_{\Pi_\delta^+}\|_{X^{0,b,\alpha}(\Pi_\delta^+)}\leq C \delta^\epsilon \|h\|_{X^{0,b,\alpha}} \|w\big|_{\Pi_\delta^+}\|_{X^{0,b,\alpha}(\Pi_\delta^+)}.
\end{align*}

Choosing $\delta>0$ such that $C\delta^\epsilon\|h\|_{X^{0,b,\alpha}}\leq \frac12$, we conclude that $w\equiv 0$ in $\Pi_\delta^+$. It can be proven in a finite number of steps that $w\equiv 0$ in $\Pi_{T}^+$. Lemma \ref{V4.2-L6.5} is now established.
\end{proof}

\begin{Theorem}\label{V4.2-T6.1} Let $b\in [\frac{13}{32},\frac12)$, and $\alpha\in(\frac12,\alpha_0(b))$, where $\alpha_0(b)\in(\frac12,1)$ has the same properties than in Lemma \ref{V4.2-L6.3}. Then IBVP \eqref{maineq} has at most one generalized solution in $X^{0,b,\alpha}(\Pi_T^+)$.
\end{Theorem}

\begin{proof} Let $u_1,u_2$ be in $X^{0,b,\alpha}(\Pi_T^+)$, two generalized solutions of IBVP \eqref{maineq}, and let us define $u:=u_1-u_2$. Then $u$ is a generalized solution of the following IBVP in $X^{0,b,\alpha}(\Pi_T^+)$:
\begin{align}
\left. \begin{array}{rlr}
u_t+\partial_x^5 u + (h(x,t)u)_x&\hspace{-2mm}=0,&\quad x>0,\quad 0<t<T,\\
u(x,0)&\hspace{-2mm}=0,&\quad x\geq0,\\
u(0,t)=0,\quad u_x(0,t)=0,\quad u_{xx}(0,t)&\hspace{-2mm}=0, &\quad t\in[0,T],
\end{array} \right\}\label{V4.2-6.65}
\end{align}

where $h\in X^{0,b,\alpha}$ is a certain extension of $\frac{u_1+u_2}2\in X^{0,b,\alpha}(\Pi_T^+)$. But, from Lemma \ref{V4.2-L6.5}, the unique generalized solution of IBVP \eqref{V4.2-6.65} in $X^{0,b,\alpha}$ is $u\equiv 0$. Therefore $u_1=u_2$, and Theorem \ref{V4.2-T6.1} is proven.
\end{proof}

\section{Regularity property of the solution of the IBVP \eqref{maineq}}\label{Regular}

This property tells us that the non-linear part of the solution to the IBVP \eqref{maineq} at a time $t\in(0,T]$, where $[0,T]$ is the interval of existence, is smoother than the initial data.\\

Let us recall that the solution $u$ of the IBVP \eqref{maineq} is such that, for $t\in[0,T]$,
$$u(t)=W_{\mathbb R}(t) g_l + \int_0^t W_{\mathbb R}(t-t')F_T(u(t')) dt' + W_0^t(0,h_1-p_1,h_2-p_2,h_3-p_3),$$

where, for $j=0,1,2$,
$$p_{j+1}(t)=\eta(t) \partial_x^j [W_{\mathbb R}(t) g_l](0) + \eta(t) \partial_x^j \int_0^t [W_{\mathbb R}(t-t')F_T(u(t'))](0) dt'\equiv q_{j+1}(t) + r_{j+1}(t).$$

If we define
$$W_0^t(g,h_1,h_2,h_3)(\cdot_x,t):=W_0^t(0,h_1-q_1,h_2-q_2,h_3-q_3)(t) + W_{\mathbb R}(t) g_l,$$

Then, the non-linear part of the solution $u$ to the IBVP \eqref{maineq} is given by
$$u(\cdot_x,t)-W_0^t(g,h_1,h_2,h_3)(\cdot_x,t).$$

Let us notice that
$$u(x,t)-W_0^t(g,h_1,h_2,h_3)(t)=\int_0^t W_{\mathbb R}(t-t')F_T(u(t')) dt' - W_0^t(0,r_1,r_2,r_3)(t).$$

\begin{Theorem}\label{V4.2-T7.1} Let $s\in[0,\frac{11}4)\setminus\{\frac12,\frac32,\frac52\}$, $g\in H^2(\mathbb R^+)$, and $h_{j+1}\in H^{\frac{s+2-j}5}(\mathbb R^+)$, $j=0,1,2$, such that the compatibility conditions stated in Theorem \ref{V3-T5.1} are satisfied. Let $u$ be the solution to the integral equation \eqref{V2-1.7} guaranteed by the theorem, restricted to $\mathbb R^+\times[0,T]$, where $0<T\leq\frac12$, and $T$ satisfies the properties of Theorem \ref{V3-T5.1}. Then
\begin{itemize}
\item[(i)] If $0\leq s<\frac12$, $\frac9{20}<b<\frac12$, and $0\leq a<\frac12-s$, then
$$\|u-W_0^t(g,h_1,h_2,h_3)\|_{C([0,T];H^{s+a}(\mathbb R_x^+))}\leq C\|u\|^2_{X^{s,b}}<\infty.$$
\item[(ii)] If $s\in(\frac12,\frac{11}4)\setminus\{\frac32,\frac52\}$, and $a,b$ are such that
\begin{align}
\max\left\{\frac{s+a}5-\frac1{20},\frac25 \right\}<b<\frac12, \text{ and } 0\leq a<\min\left\{\frac{11}4-s,10b-4 \right\},\label{V4.2-7.3}
\end{align}
then
$$\|u-W_0^t (g,h_1,h_2,h_3)\|_{C([0,T];H^{s+a}(\mathbb R_x^+))}\leq C \|u\|^2_{X^{s,b}}<\infty.$$

\end{itemize}
\end{Theorem}

\begin{remark} It can be readily verified that, if
$$(\frac12<s<\frac32,\, \frac9{20}<b<\frac12,\, 0\leq a\leq \frac12), \quad \text{or, if} \quad (\frac32<s<\frac52,\, \frac{19}{40}<b<\frac12,\, 0\leq a\leq \frac18),$$
$$\text{or, if}\quad (\frac52<s<\frac{11}4,\, \frac{19}{40}<b<\frac12,\, 0\leq a\leq\frac12(\frac{11}{4}-s)),$$
then statement \eqref{V4.2-7.3} is valid.
\end{remark}

\begin{proof} Let us observe that
\begin{align}
\notag \|u-W_0^t(g,h_1,h_2,h_3)\|_{C([0,T];H^{s+a}(\mathbb R_x^+))} \leq & \|\eta(\cdot_t) \int_0^{\cdot_t} [W_{\mathbb R}(\cdot_t{\scriptstyle-}t')F_T(u(t'))](\cdot_x)dt'\|_{C([0,T];H^{s+a}(\mathbb R_x))}\\
&+\|\eta(\cdot_t)W_0^t(0,r_1,r_2,r_3)\|_{C([0,T];H^{s+a}(\mathbb R_x))}.\label{V4.2-7.4}
\end{align}

Let us estimate both terms on the right hand side of \eqref{V4.2-7.4}. From now on, we choose $b^*$, and $\tilde b$, such that $b<b^*<\frac12<\tilde b$, and $b^*+\tilde b\leq 1$. With respect to the first term on the right hand side of \eqref{V4.2-7.4}, using the Sobolev immersion $X^{s+a,\tilde b}\hookrightarrow C(\mathbb R_t;H^{s+a}(\mathbb R_x))$, Lema \ref{V2-L3.3} with $b':=- b^*$, and $b:=\tilde b$, Lema \ref{V2-L3.4} with $b_1:=-b^*$, and $b_2:=-b$, and Lemma 4.6 (i), it follows, for
\begin{align}
\frac25<b<\frac12,\quad \text{and}\quad 0\leq a\leq 10b-4,\label{V4.2-7.5}
\end{align}

that
\begin{align}
\notag \|\eta(\cdot_t) \int_0^{\cdot_t} [W_{\mathbb R}(\cdot_t{\scriptstyle-}t')F_T(u(t'))](\cdot_x) dt'\|_{C([0,T];H^{s+a}(\mathbb R_x))}\leq &  C \|\eta(\cdot_t)\int_0^{\cdot_t} [W_{\mathbb R}(\cdot_t{\scriptstyle-}t')F_T(u(t'))](\cdot_x) dt'\|_{X^{s+a,\tilde b}}\\
\notag\leq & C \|F_T(u(\cdot))\|_{X^{s+a,-b^*}} \leq CT^{-b+b^*} \|\partial_x u^2\|_{X^{s+a,-b}}\\
\leq & C \|u\|^2_{X^{s,b}}<\infty.\label{V4.2-7.6}
\end{align}

With respect to the second term on the right hand side of \eqref{V4.2-7.4}, considering Lemma \ref{V3-L2}, Remark \ref{V3-R2}, the definition of $r_{j+1}$, Lemma \ref{V3-L4.6} (i), Lemma \ref{V2-L3.4}, and Lemma \ref{V2-L3.6} (inequality \eqref{V1-3.58}), it follows that, if $0\leq s+a\leq \frac12$, $\frac25<b<\frac12$, and $0\leq a\leq 10b-4$, then
\begin{align}
\notag \|\eta(\cdot_t) W_0^t(0,r_1,r_2,r_3)\|_{C([0,T];H^{s+a}(\mathbb R_x))} & \leq C \sum_{j=0}^2 \|r_{j+1}\|_{H^{\frac{(s+a)+2-j}5}(\mathbb R_t)}\\
\notag & \leq C \|F_T(u(\cdot))\|_{X^{s+a,-b^*}} \leq C T^{b^*-b} \|\partial_x u^2\|_{X^{s+a,-b}}\\
&\leq C T^{b^*-b} \|u\|^2_{X^{s,b}} < +\infty. \label{V4.2-7.7}
\end{align}

From \eqref{V4.2-7.4}, \eqref{V4.2-7.6}, and \eqref{V4.2-7.7} we obtain statement (i) of the theorem.\\

On the other hand, if $\frac12\leq s+a\leq \frac{11}2$, taking into account Lemma \ref{V3-L2}, Lemma \ref{V3-L4.6} (ii), Lemma \ref{V2-L3.4}, Lemma \ref{V2-L3.6} (inequality \eqref{V1-3.58}), and Lemma \ref{V2-L3.7}, it follows that, if
$$\max\left\{\frac{s+a}5-\frac1{20},\frac25\right\}<b<\frac12,\quad \text{and}\quad 0\leq a<\min\left\{\frac{11}4-s,10b-5 \right\},$$
then
\begin{align}
\notag \|\eta(\cdot_t) W_0^t(0,r_1,r_2,r_3)\|_{C([0,T];H^{s+a}(\mathbb R_x))} & \leq C \sum_{j=0}^2 \|r_{j+1}\|_{H^{\frac{(s+a)+2-j}5}(\mathbb R_t)}\\
\notag & \leq C \|F_T(u(\cdot))\|_{X^{\frac12,\frac{2(s+a)-1-10b^*}{10}}} + C \|F_T(u(\cdot))\|_{X^{s+a,-b^*}}\\
\notag& \leq  T^{b^*-b} \|\partial_x u^2\|_{X^{\frac12,\frac{2(s+a)-1-10b}{10}}}+C T^{b^*-b} \|\partial_x u^2\|_{X^{s+a,-b}}\\
&\leq C T^{b^*-b} \|u\|^2_{X^{s,b}} < +\infty. \label{V4.2-7.8}
\end{align}

From \eqref{V4.2-7.4}, \eqref{V4.2-7.6}, and \eqref{V4.2-7.8}, we obtain statement (ii) of the theorem.
\end{proof}

\textbf{Acknowledgments}\\

This work was partially supported by Universidad Nacional de Colombia, Sede-Medellín– Facultad de Ciencias – Departamento de Matemáticas – Grupo de investigación en Matemáticas de la Universidad Nacional de Colombia Sede Medellín, carrera 65 No. 59A -110, post 50034, Medellín Colombia. Proyecto: Análisis no lineal aplicado a problemas mixtos en ecuaciones diferenciales parciales, código Hermes 60827. Fondo de Investigación de la Facultad de Ciencias empresa 3062.

\end{document}